\documentclass[10pt,a4paper]{article}
\usepackage{latexsym}
\usepackage{amssymb}
\usepackage{bm}
\usepackage{amsmath}
\usepackage{mathtools}
\usepackage{graphicx}
\usepackage{epstopdf}
\usepackage[dvips]{epsfig}
\usepackage{wrapfig}
\usepackage{tikz}
\usepackage{amsthm}
\usepackage{float}
\usepackage{caption}
\usepackage{subcaption}
\usepackage{morefloats}
\usepackage{stmaryrd}

\usepackage{amsmath,amssymb}

\usepackage{graphicx} 
\usepackage{stmaryrd}
\usepackage{setspace}
\usepackage{amssymb,amsmath}
\usepackage{enumerate}
\usepackage[shortlabels]{enumitem}
\usepackage{supertabular}
\usepackage{epstopdf}
\usepackage[dvips]{epsfig}
\usepackage{enumerate}
\usepackage{wrapfig}
\usepackage{tikz}
\usepackage{float}
\usepackage{caption}
\usepackage{subcaption}
\usepackage{sobolev}
\usepackage{breqn}

\usepackage{hyperref}
\usepackage{amsthm}
\usepackage{thmtools}
\usepackage{cleveref}

\usepackage{textpos}
\usepackage[utf8]{inputenc}
\usepackage[T1]{fontenc}
\usepackage{helvet}
\usepackage{amsfonts}
\usepackage{amsmath}
\usepackage{amssymb}
\usepackage{mathtools}
\usepackage{dsfont}
\usepackage{amsthm}
\usepackage{bigints}
\usepackage{xparse}
\usepackage{relsize}
\usepackage{cases}
\usepackage{mathrsfs}
\usepackage{enumerate}
\usepackage{multirow}
\usepackage{booktabs}

\usepackage{lipsum}
\usepackage{caption}
\usepackage{derivative}

\usepackage{graphicx}%
\usepackage{amsmath,amssymb,amsfonts}%
\usepackage{amsthm}%
\usepackage{mathrsfs}%
\usepackage[title]{appendix}%
\usepackage{xcolor}%
\usepackage{textcomp}%
\usepackage{manyfoot}%
\usepackage{booktabs}%
\usepackage{algorithm}%
\usepackage{algorithmicx}%
\usepackage{algpseudocode}%
\usepackage{listings}%
\usepackage{physics}

\usepackage{anyfontsize}

\usepackage{hyperref}
\usepackage{url}
\hypersetup{colorlinks=true,linkcolor=blue,citecolor=blue,urlcolor=blue}
\usepackage{epstopdf}
\usepackage{xcolor}
\usepackage[top=2.5cm, left=2.0cm, right=2.0cm, bottom=2.5cm]{geometry}
\newtheorem{theorem} {\bf Theorem}
\newtheorem{lemma} {\bf Lemma}

\newtheorem{example} {\bf Example}

\newtheorem{remark}{Remark}[section]

\usepackage{booktabs}
\usepackage{appendix}
\allowdisplaybreaks
\begin{document}
\title{\bf Nitsche method for the Stokes-Poisson-Boltzmann equation with Navier slip boundary condition}

\author{
 Ayush Agrawal \thanks{Department of Mathematics, Indian Institute of Technology Roorkee, Roorkee 247667, India. Email: \texttt{ayush\_a@ma.iitr.ac.in}.}
  \and Aparna Bansal \thanks{Department of Mathematics, Indian Institute of Technology Roorkee, Roorkee 247667, India. Email: \texttt{a\_bansal@ma.iitr.ac.in}.}
  \and D. N. Pandey\thanks{Department of Mathematics, Indian Institute of Technology Roorkee, Roorkee 247667, India. Email: \texttt{dwij@ma.iitr.ac.in}.}
}
\date{}
\maketitle

\begin{abstract}
We study the Stokes--Poisson--Boltzmann equations with Dirichlet and Navier boundary conditions. The system consists of the incompressible Stokes equations coupled with a nonlinear Poisson--Boltzmann equation through electrostatic forcing and convective transport effects. To handle the Navier boundary conditions in a unified framework, we employ Nitsche's method for their weak imposition within a conforming finite element setting. We derive a consistent and stable discrete formulation and establish the well-posedness of the resulting problem. By carefully choosing the penalty parameters, the bilinear form is shown to be coercive and continuous. A priori error estimates are proved in the natural energy norms, yielding optimal-order convergence under suitable regularity assumptions. Furthermore, we develop residual-based a posteriori error estimators that incorporate element residuals, inter-element jump residuals, and boundary residuals arising from the Nitsche formulation. The estimators are shown to be reliable and locally efficient. Numerical experiments confirm the theoretical results and demonstrate the robustness and accuracy of the proposed method for the Stokes--Poisson--Boltzmann system.
\end{abstract}

\vspace{0.5cm}


 \vspace{0.5cm}

\section{Introduction}
Electrically charged fluid flows play an important role in several chemical applications, such as the design of nanopores for biomedical devices and the modelling of water purification systems. One of the simplest mechanisms for transporting fluid without mechanical pumping is electro-osmosis. In this process, the motion of an electrolyte is driven by thin electric double layers that attract and repel ions near charged surfaces. The resulting fluid motion can be described by the Stokes equations, where the forcing term depends on the electric charge of the electrolyte and an externally applied electric field.

In many simplified situations, the fluid flow and pressure gradients are not strong enough to significantly influence the distribution of the electrostatic potential within the double layer. As a result, the electric charge density can be directly related to the electrostatic potential through simplified forms of the Poisson--Boltzmann equation. This observation forms the key idea of the modeling approach, and following \cite{alsohaim2024analysis}, we consider the Stokes--Poisson--Boltzmann (SPB) system in the form given by \eqref{spbe}.

Regarding the unique solvability and finite element discretisation of the
Poisson--Boltzmann equation, we refer to
\cite{chen2007finite,holst2012adaptive,iglesias2022weak,kraus2020reliable}.
In these works, the authors employ a decomposition of the solution into
regular and singular components, use convex minimisation techniques,
and establish an $L^{\infty}$ bound for the solution through a
cutoff-function argument. However, only limited literature is available for the
Stokes--Poisson--Boltzmann (SPB) system. In 2025, AlSohaim et al.~\cite{alsohaim2024analysis} investigated a finite element discretization of the SPB equation, establishing well-posedness and proving optimal convergence rates for their scheme. Recently, Mishra et al.~\cite{mishra2026higher} introduced an equal-order virtual element method for the Stokes–Poisson–Boltzmann (SPB) system and derived optimal a priori error estimates in the energy norm.

Navier boundary conditions impose a constraint only on the normal component of the velocity field, which makes their numerical implementation somewhat challenging. Several techniques have been developed to enforce slip boundary conditions weakly, including Lagrange multiplier method, penalty methods, and Nitsche method, among others. The first approach used to weakly impose slip conditions was the Lagrange multiplier method. In this approach, an additional variable is introduced, which increases the computational cost due to the extra degrees of freedom in the discrete problem \cite{layton1999weak,verfurth1986finite}. 

Later, penalty methods were proposed to enforce slip conditions through a regularization term; see \cite{babuvska1973finite,barrett1986finite}. Although this approach is inconsistent and requires additional regularity of the solution, it is still widely used because of its simplicity. For the Navier--Stokes (Stokes) equations with slip boundary conditions, both Lagrange multiplier and penalty approaches have been studied in \cite{caglar2009weak,dione2015penalty,urquiza2014weak}, particularly for problems with curved boundaries where a Babu\v{s}ka-type paradox may occur. More recently, penalty techniques combined with Lagrange finite elements \cite{kashiwabara2016penalty,zhou2016penalty} and Crouzeix--Raviart finite elements \cite{kashiwabara2019penalty,zhou2021crouzeix} have been investigated. These studies carefully analyze the effect of variational crimes arising in the numerical approximation.

The Nitsche method, introduced by J. Nitsche \cite{nitsche1971variationsprinzip}, is a consistent penalty-based approach for weakly imposing Dirichlet boundary conditions. Several variants of this method have been studied in \cite{chouly2026review}.
The treatment of Navier boundary conditions using the Nitsche method was studied in \cite{winter2018nitsche}, building on the formulation proposed by Juntunen and Stenberg \cite{juntunen2009nitsche} for the robust discretization of Robin-type boundary conditions. Furthermore, Gjerde and Scott \cite{gjerde2022nitsche} investigated the convergence properties when curved boundaries are approximated by polygonal meshes. By projecting the normal and tangential vectors appropriately, they avoided the occurrence of the Babu\v{s}ka paradox and obtained optimal convergence rates. Araya et al.~\cite{araya2024stokes} studied both the symmetric and non-symmetric variants of the Nitsche method for imposing slip boundary conditions in the Stokes equations using stabilized finite element methods. Bansal et al.~\cite{bansal2026equal,bansal2024nitsche} studied the Navier–Stokes equations using inf–sup stable and equal-order stabilized finite element methods, and this work is further extended to general dynamic boundary conditions in~\cite{GazcaOrozco2025Nitsche}.

In adaptive finite element methods, a posteriori error estimators are used to measure the local distribution of discretization errors. A reliable estimator provides an upper bound for the true error and can also be used as a stopping criterion in the adaptive mesh refinement procedure. Furthermore, an efficient estimator guarantees that its convergence rate is comparable to that of the actual error. Most existing studies on a posteriori error estimation for finite element methods focus on diffusion problems \cite{ainsworth1997posteriori,verfurth2013posteriori,burger2024divergence}, while comparatively fewer works address advection--diffusion equations. For the SPB equation, residual-based error estimators have been developed in \cite{allendes2020stabilized,araya2014adaptive,berrone2001adaptive}. In this work, we provide both a priori and a posteriori error analyses for the stationary SPB equation with slip boundary conditions using the symmetric variant of the Nitsche's method.

The main contributions of this work are twofold. First, we prove the well-posedness of the SPB equation with slip boundary conditions and derive a \textit{priori} error estimates using a symmetric variant of Nitsche’s method. Secondly, we develop a residual-based a \textit{posteriori} error estimator and show that it is reliable and locally efficient.

\subsection*{Outline}

The remainder of this paper is organized as follows. In \Cref{sec:mp}, we introduce the mathematical model and discuss the analysis of the continuous problem. In \Cref{sec:dp}, we present a Nitsche-based numerical scheme and establish the existence and uniqueness of the discrete solution by combining Banach's fixed-point theorem, the Babu\v{s}ka--Brezzi theory, and the Minty--Browder theorem under suitable small-data assumptions.
In \Cref{sec:pri}, we derive \emph{a priori} error estimates and prove the optimal convergence of the proposed method. \Cref{sec:post} is devoted to the development and analysis of residual-based \emph{a posteriori} error estimator, where its reliability and efficiency are established. In \Cref{sec:ne}, we present numerical experiments that verify the theoretical convergence rates and demonstrate the robustness of the proposed estimators. Finally, in \Cref{sec:conc}, we summarize the main conclusions of this work.

\subsection*{Notations}
Throughout this manuscript, we utilize the classical Sobolev spaces $L^2(\Omega)$ and $H^1(\Omega)$, equipped with their respective norms $\|\cdot\|_{L^2(\Omega)}$ and $\|\cdot\|_{H^1(\Omega)}$. The $L^2$-inner product is denoted by $(\cdot, \cdot)$, and, for any arbitrary Hilbert space $H$, the duality pairing with its dual space $H'$ is represented by $\langle\cdot, \cdot\rangle_{H', H}$. We follow the convention of denoting scalars and vectors by $a$ and $\boldsymbol{a}$, respectively. 

For the sake of simplicity, throughout the analysis, we use $C$ to denote a generic positive constant that is independent of the mesh size $h$ but may depend on the model parameters. Moreover, whenever an inequality involves positive constants that are independent of the mesh size but may depend on model parameters, we use the symbols $\lesssim$ or $\gtrsim$ to omit explicit constants. The assumption of homogeneity in the boundary conditions is made to simplify the subsequent analysis, as lifting operators have already been established~\cite{MR3974685}. Non-homogeneous boundary conditions are used in the numerical tests in \Cref{sec:ne}.

\section{Model problem}\label{sec:mp}
Let $\Omega \subset \mathbb{R}^n$, $n = 2,3$, be a bounded Lipschitz domain with boundary $\Gamma = \Gamma_D \cup \Gamma_{\mathrm{Nav}}$, where $\Gamma_D$ and $\Gamma_{\mathrm{Nav}}$ correspond to different boundary conditions and satisfy $\Gamma_D \cap \Gamma_{\mathrm{Nav}} = \emptyset$. The domain is occupied by an
incompressible electrolytic fluid driven by pressure gradients and electric fields. In the stationary regime, the governing equations are expressed in terms of the
fluid velocity $\boldsymbol{u}$, pressure $p$, and electrostatic double-layer
potential $\psi$. The coupled Stokes--Poisson--Boltzmann system is given by
\begin{equation}\label{spbe}
\begin{aligned}
- \mu \Delta \boldsymbol{u} + \nabla p
&= \boldsymbol{f} - \varepsilon \Delta \psi \, \boldsymbol{E}
\quad && \text{in } \Omega, \\
\operatorname{div} \boldsymbol{u}
&= 0
&& \text{in } \Omega,  \\
\mathcal{K} (\psi) + \boldsymbol{u} \cdot \nabla \psi - \varepsilon \Delta \psi
&= g
&& \text{in } \Omega, \\
\boldsymbol{u} &= \boldsymbol{0}
\quad &&\text{on } \Gamma_D,\\
\psi &= 0 \quad &&\text{on } \Gamma,
\end{aligned}
\end{equation}
with the Navier boundary condition on $\Gamma_{\mathrm{Nav}}$ is defined as
\begin{equation}\label{navbc}
    \begin{aligned}
        \boldsymbol{u}\cdot \boldsymbol{n}  &= 0 \qquad \text{on } \Gamma_{\mathrm{Nav}},\\
        \mu \boldsymbol{n}^t \nabla \boldsymbol{u} \boldsymbol{\tau}^i + \beta \boldsymbol{u}\cdot\boldsymbol{\tau}^i &= 0\qquad \text{on } \Gamma_{\mathrm{Nav}}, \quad i=1,\ldots(d-1),
    \end{aligned}
\end{equation}
where $\mu$ is the fluid viscosity, $\beta > 0$ is the friction coefficient, and $\varepsilon$ represents the electric permittivity of the electrolyte. The vector $\boldsymbol{E}$ denotes the externally applied electric field, while $\boldsymbol{n}$ and $\boldsymbol{\tau}^i$ represent the unit normal and tangential vectors to the boundary $\Gamma$, respectively. Furthermore, $\boldsymbol{f} \in \boldsymbol{L}^2(\Omega)$ denotes the external body force, and $g \in L^2(\Omega)$ represents an external source or sink of electric potential.
The charge density is modelled by the nonlinear function
\[
\mathcal{K}(\psi) = k_0 \sinh(k_1 \psi),
\]
where $k_0 > 0$ depends on the ion valence, electron charge, and bulk ion concentration, while $k_1 > 0$ additionally depends on the Boltzmann constant and the reference absolute temperature.
\subsection*{Assumptions}\label{assumptions} 
Following AlSohaim et al.~\cite{alsohaim2024analysis} and Holst et al.~\cite[Lemma 2.1]{holst2012adaptive}, we assume that
\begin{enumerate}[label=(\textrm{A.\arabic*}), ref=$\mathrm{(A.\arabic*)}$]
    \item \label{(A1)} The potential is uniformly bounded, i.e. there exist two constants $\alpha_1,\alpha_2 \in \mathbb{R}$ satisfying $\alpha_1 \leq 0 \leq \alpha_2$ such that
    \[ \alpha_1 \le \psi(t)\le \alpha_2,\qquad \forall t \in \mathbb{R}.\]
    \item  \label{(A2)} We assume that the nonlinear density function $\mathcal{K}(\psi)$ satisfies $\mathcal{K}(0)=0$. Moreover, there exist positive constants $\overline{K}$ and $\underline{K}$ such that
\begin{equation}
    \begin{aligned}
        |\mathcal{K}(s_1)-\mathcal{K}(s_2)| \leq \overline{K}\,|s_1-s_2|
\quad \text{for all } s_1,s_2 \in [\alpha_1,\alpha_2], \\
|\mathcal{K}(s_1)-\mathcal{K}(s_2)| \geq \underline{K}\,|s_1-s_2|
\quad \text{for all } s_1,s_2 \in [\alpha_1,\alpha_2].
    \end{aligned}        
\end{equation}
\item \label{(A3)} $\boldsymbol{E} \in \boldsymbol{L}^\infty(\Omega)$ is assumed to be uniformly bounded by $\bar{E}>0$.
\end{enumerate}

\subsection*{Weak formulation}

We introduce the trial and test spaces for the velocity field, the
electrostatic double-layer potential, and the pressure as follows:
\begin{align*}
    \boldsymbol{V} := \{\boldsymbol{v} \in \boldsymbol{H}^1(\Omega) : \boldsymbol{v}\cdot \boldsymbol{n} = \boldsymbol{0} \text{ on } \Gamma_{\mathrm{Nav}}, \boldsymbol{v} = \boldsymbol{0} \text{ on } \Gamma_D\}, \qquad 
\Phi := \H10,\qquad
 Q := L_0^2(\Omega).
\end{align*}
equipped with the norm $ \|(\boldsymbol{u}, p, \psi)\| = \|\nabla\boldsymbol{u}\|\, + \,\|p\| \,+\, \|\nabla\psi\| $ for $(\boldsymbol{u}, p, \psi) \in \boldsymbol{V} \times Q\times \Phi $.

Multiplying \labelcref{spbe} with suitable test functions,
integrating by parts, and incorporating the boundary conditions, yields the following weak formulation: find
$(\boldsymbol{u}, p, \psi) \in \boldsymbol{V} \times Q\times \Phi $ such that
\begin{equation}\label{wf}
\begin{aligned}
A(\boldsymbol{u},\boldsymbol{v}) + B(\boldsymbol{v},p) + \boldsymbol{c}_1(\boldsymbol{u};\psi,\boldsymbol{v}) + (\mathcal{K}(\psi)\boldsymbol{E},\boldsymbol{v})
&= F(\boldsymbol{v}),
\qquad && \forall \boldsymbol{v} \in \boldsymbol{V}, \\
B(\boldsymbol{u},q) &= 0,
&& \forall q \in Q, \\
(\mathcal{K}(\psi),\phi) + \boldsymbol{c}_2(\boldsymbol{u};\psi,\phi) + d(\psi,\phi)
&= G(\phi),
&& \forall \phi \in \Phi.
\end{aligned}
\end{equation}

The bilinear and trilinear forms are defined by
\begin{align*}
A(\boldsymbol{u},\boldsymbol{v}) := a(\boldsymbol{u},\boldsymbol{v}) + a^{\partial}_{\tau}(\boldsymbol{u},\boldsymbol{v}),
\qquad
&B(\boldsymbol{v},q) := - \int_\Omega (\operatorname{div}\boldsymbol{v})\, q\, \dd{x},    \\
\boldsymbol{c}_1(\boldsymbol{w};\psi,\boldsymbol{v}) := \int_\Omega (\boldsymbol{w} \cdot \nabla \psi)\,\boldsymbol{E}\cdot \boldsymbol{v} \, \dd{x}, \qquad 
\boldsymbol{c}_2(\boldsymbol{w};\psi,\phi) :=& \int_\Omega (\boldsymbol{w} \cdot \nabla \psi)\, \phi \, \dd{x}, \qquad
d(\psi,\phi) := \varepsilon \int_\Omega \nabla \psi \cdot \nabla \phi\, \dd{x}.
\end{align*}
with forms 
\begin{align*}
     a(\boldsymbol{u},\boldsymbol{v}) := \mu \int_\Omega \nabla \boldsymbol{u} : \nabla \boldsymbol{v}\, \dd{x}, \qquad  
     a^{\partial}_{\tau}(\boldsymbol{u},\boldsymbol{v}) := \beta\int_{\Gamma_{\mathrm{Nav}}}  \sum\limits_{i=0}^{d-1} (\boldsymbol{u}\cdot\boldsymbol{\tau}^i)(\boldsymbol{v}\cdot\boldsymbol{\tau}^i)\,\dd{s} .
\end{align*}
The linear functionals are given by
\[
F(\boldsymbol{v}) := \int_\Omega ( \boldsymbol{f} + g \boldsymbol{E} ) \cdot \boldsymbol{v}\, \dd{x},
\qquad
G(\phi) := \int_\Omega g\, \phi\, \dd{x}.
\]
In a compact form, we may write the following weak formulation: find $(\boldsymbol{u}, p, \psi) \in \boldsymbol{V} \times Q\times \Phi $ such that
\begin{align}
    \mathcal{B}(\boldsymbol{u},p,\psi; \boldsymbol{v}, q, \phi) = F(\boldsymbol{v}) + G(\phi), \qquad\qquad \forall (\boldsymbol{v}, q, \phi) \in \boldsymbol{V}\times Q\times \Phi, 
\end{align}
where
\begin{align*}
    \mathcal{B}(\boldsymbol{u}, p, \psi;\boldsymbol{v}, q, \phi) &= A(\boldsymbol{u},\boldsymbol{v}) + B(\boldsymbol{v},p) + \boldsymbol{c}_1(\boldsymbol{u};\psi,\boldsymbol{v}) + (\mathcal{K}(\psi)\boldsymbol{E},\boldsymbol{v}) + B(\boldsymbol{u},q) + (\mathcal{K}(\psi),\phi) \\& \quad + \boldsymbol{c}_2(\boldsymbol{u};\psi,\phi) + d(\psi,\phi).
\end{align*}

Note that the weak form follows from rewriting the last source term in first equation using third equation in \eqref{spbe}. Consequently, the Laplacian of the potential in the momentum equation does not need to be integrated by parts, which leads to a more convenient structure for our analysis. The rest of this section is devoted to proving the continuity, ellipticity, and inf–sup conditions for a collection of operators essential to the forthcoming analysis.
\medskip
\begin{lemma}\label{lma1.1}
   Assume \ref{(A3)}. There exist positive constants $C_1, C_2,$ and $ C_3$, such that the following continuity estimates hold:
\begin{align*}
|A(\boldsymbol{u},\boldsymbol{v})|
&\le C_1 \|\boldsymbol{u}\|_{1}\|\boldsymbol{v}\|_{1}, 
&|B(\boldsymbol{v},q)|
&\le \|\boldsymbol{v}\|_{1}\|q\|, \\
|d(\psi,\phi)|
&\le \varepsilon \|\psi\|_{1}\|\phi\|_{1}, 
&|c_1(\boldsymbol{w};\psi,\boldsymbol{v})|
&\le C_2
\|\boldsymbol{w}\|_{1}\|\psi\|_{1}\|\boldsymbol{v}\|_{1}, \\
|c_2(\boldsymbol{w};\psi,\phi)|
&\le C_3
\|\boldsymbol{w}\|_{1}\|\psi\|_{1}\|\phi\|_{1},.
\end{align*}
     for all $\boldsymbol{u},\boldsymbol{v},\boldsymbol{w}\in \boldsymbol{V}$, $q\in Q$, and $\psi,\phi\in \Phi$. Here $C_1=\min\{\mu, \beta\}, C_2 = \bar{E}C_{\mathrm{Sob}}^2 $ and $C_3=C_{\mathrm{Sob}}^2$.
\end{lemma}
\begin{proof}
    These inequalities follow directly from the Cauchy-Schwarz, H\"older inequalities and the Sobolev embedding $\H1 \hookrightarrow L^q(\Omega)$  holds for $1\leq q < \infty$ when $d = 2$ and $1 \leq q \leq 6$ when $d = 3$.
\end{proof}
Let $Z$ denote the kernel of the bilinear form $B$, defined by
\[
Z := \{\boldsymbol{v}\in \boldsymbol{V} \;:\; B(\boldsymbol{v}, p) = 0 \quad \forall\, p \in Q \}.
\]
The following lemma establishes the ellipticity of the bilinear form $A$ when restricted to the space $Z$.
\begin{lemma}\cite{bansal2024nitsche}\label{lma1.2}
There exists a constant $\xi>0$, depending on $\beta$, $\mu$, $\Omega$, and $\Gamma_{\mathrm{Nav}}$,
such that
\[
A(\boldsymbol{v},\boldsymbol{v}) \ge \xi \|\boldsymbol{v}\|_{1}^{2},
\qquad \forall\, \boldsymbol{v} \in \boldsymbol{Z}.
\]
\end{lemma}
Next, we present the continuous inf-sup condition for the bilinear form $B$.

\begin{lemma}\cite{bansal2024nitsche}\label{lma1.3}
There exists a positive constant $\theta>0$, depending only on the shape of the
domain $\Omega$, such that
\[
\sup_{0 \neq \boldsymbol{v} \in \boldsymbol{V}}
\frac{|B(\boldsymbol{v},q)|}{\|\boldsymbol{v}\|_1}
\ge \theta \|q\|,
\qquad \forall\, q \in Q.
\]
\end{lemma}

Now we introduce the bilinear form
\begin{equation}\label{eqc}
\mathcal{C}(\boldsymbol{u},p,\psi;\boldsymbol{v},q,\phi)
= A(\boldsymbol{u},\boldsymbol{v}) + B(\boldsymbol{u},q) + B(\boldsymbol{v},p) + d(\psi,\phi).
\end{equation}

\begin{theorem}\label{thm1.4}
There exists a positive constant $\alpha = \alpha(\xi,\theta)$ such that
\[
\sup_{0 \neq (\boldsymbol{v},q,\phi)\in \boldsymbol{V} \times Q \times \Phi}
\frac{\mathcal{C}(\boldsymbol{u},p,\psi;\boldsymbol{v},q,\phi)}
{\|(\boldsymbol{v},q,\phi)\|}
\ge \alpha\,\|(\boldsymbol{u},p,\psi)\|,
\qquad
\forall\, (\boldsymbol{u},p,\psi)\in \boldsymbol{V} \times Q \times \Phi
\]
Here, $\xi$ and $\theta$ denote the coercivity and inf-sup stability
constants, respectively.
\end{theorem}
\begin{proof}
By \Cref{lma1.1}, the bilinear form $\mathcal{C}(\cdot\,;\cdot)$ is uniformly bounded. 
Moreover, the ellipticity of $A$ on $Z$ (see \Cref{lma1.2}) together with the inf-sup condition for $B$ (see \Cref{lma1.3}), and Proposition~2.36 in~\cite{ern2004theory}, ensure that the required inf-sup condition holds.
\end{proof}

\begin{remark}
We first establish the basic properties of the bilinear operators 
$A(\cdot\,,\cdot)$, $B(\cdot\,,\cdot)$ and $\mathcal{C}(\cdot\,;\cdot)$, which are essential for proving 
well-posedness. The existence and uniqueness of the solution follow from 
Banach’s fixed-point theorem combined with the Babu\v{s}ka–Brezzi theory 
\cite{boffi2013mixed} and the Minty–Browder theorem 
\cite{ciarlet2025linear}. Since the proof proceeds along the same lines 
as in the discrete case, we omit the details here to avoid repetition.
\end{remark}

\section{Discrete problem}\label{sec:dp}
We consider the finite element scheme. Let $\mathcal{T}_h$ be a shape-regular simplicial triangulation of the domain $\Omega$ (consisting of triangles in two dimensions or tetrahedra in three dimensions). For each element $K \in \mathcal{T}_h$, let $h_K$ denote its diameter, and define the mesh size by $h = \max_{K \in \mathcal{T}_h} h_K.$

For any facet $e$ (an edge in two dimensions or a face in three dimensions), we denote by $K^{-}$ and $K^{+}$ the two elements of $\mathcal{T}_h$ sharing $e$. Let $h_e$ denote the diameter of the facet $e$. We define by $\mathcal{E}_h$ the set of all facets of the mesh, and decompose it as $ \mathcal{E}_h = \mathcal{E}_{\Omega} \cup \mathcal{E}_{D} \cup \mathcal{E}_{\mathrm{Nav}},$ where $\mathcal{E}_{\Omega}$ denotes the set of interior facets, while
$\mathcal{E}_{D}$ and $\mathcal{E}_{\mathrm{Nav}}$ denote the sets of facets lying on the boundaries $\Gamma_{D}$ and $\Gamma_{\mathrm{Nav}}$, respectively. Furthermore, let $\boldsymbol{n}_e^{+}$ and $\boldsymbol{n}_e^{-}$ be the outward unit normal vectors on the boundaries of the elements $K^{+}$ and $K^{-}$ associated with the facet $e$, respectively. The jump operator $\llbracket \cdot \rrbracket$ across a facet $e \in \mathcal{E}_h^{i}$ is defined by
$\llbracket \boldsymbol{v} \rrbracket := \boldsymbol{v}^{-} - \boldsymbol{v}^{+}, \qquad \llbracket w \rrbracket := w^{-} - w^{+}.$
while for boundary jumps, we adopt the conventions $\llbracket \boldsymbol{v} \rrbracket = \boldsymbol{v}$ and $\llbracket w \rrbracket = w$.

Let us introduce the following finite element spaces for the approximation of the velocity, pressure, and electrostatic potential:
\begin{align*}
\boldsymbol{V}_h & := \{\boldsymbol{v}_h \in C^0(\Omega) :
\boldsymbol{v}_h|_K \in [\mathbb{P}_{k+1}(K)]^n \ \forall K \in \mathcal{T}_h,
\ \boldsymbol{v}_h := \boldsymbol{0} \text{ on } \Gamma_D \}, \\
Q_h &:= \{ q_h \in C^0(\Omega) :
q_h|_K \in \mathbb{P}_k(K) \ \forall K \in \mathcal{T}_h \} \cap Q, \\
\Phi_h &:= \{ \phi_h \in C^0(\Omega) :
\phi_h|_K \in \mathbb{P}_{k+1}(K) \ \forall K \in \mathcal{T}_h,
\ \phi_h = 0 \text{ on } \Gamma \}.
\end{align*}
where $\mathbb{P}_k(K)$ denotes the space of polynomials of degree at most $k \ge 1$ on $K$, we
We introduce the following discrete norm on these spaces, which will be used throughout the analysis, defined by
\begin{equation}\label{norm}
\begin{aligned}
\|\boldsymbol{v}_h\|_{1,h}^{2}
&:= \|\nabla \boldsymbol{v}_h\|^{2}
+ \sum_{e\in\mathcal{E}_{\mathrm{Nav}}}
\frac{1}{h_e}\,
\|\boldsymbol{v}_h \cdot \boldsymbol{n}\|_{0,e}^{2},\\
\|(\boldsymbol{u}_h,p_h,\psi_h)\| &:=
\|\nabla \boldsymbol{u}_h\|^{2}
+ \sum_{e\in\mathcal{E}_{\mathrm{Nav}}}
\frac{1}{h_e}\,
\|\boldsymbol{u}_h \cdot \boldsymbol{n}\|_{0,e}^{2} \,+\, \|p_h\|^2 \,+\, \|\nabla\psi_h\|^2.
\end{aligned}
\end{equation}
\begin{remark}
    It is easy to verify that $\boldsymbol{V}_h$ is not a subspace of $\boldsymbol{V}$ so Nitsche’s method can be considered as a non-conforming finite element method.
\end{remark}
\subsection{Nitsche’s method}
Nitsche’s method \cite{stenberg1995some,winter2018nitsche} is designed to impose boundary conditions weakly, ensuring that they are fulfilled only in an asymptotic sense. In this study, we employ this technique solely for the Navier boundary condition $\boldsymbol{u}\cdot \boldsymbol{n} = 0$. As a result, the weak formulation corresponding to the symmetric Nitsche approach is stated below:
\begin{align}\label{ns1}
    \mathcal{B}_h(\boldsymbol{u}_h,p_h,\psi_h, \boldsymbol{v}_h, q_h, \phi_h) = F(\boldsymbol{v}_h) + G(\phi_h), \qquad\qquad \forall (\boldsymbol{v}_h, q_h, \phi_h) \in \boldsymbol{V}_h\times Q_h\times \Phi_h, 
\end{align}
where
\begin{align*}
    \mathcal{B}_h(\boldsymbol{u}_h, p_h, \psi_h;\boldsymbol{v}_h, q_h, \phi_h) :=& \sum\limits_{K\in \mathcal{T}_h} \bigg(\mu(\nabla \boldsymbol{u}_h, \nabla \boldsymbol{v}_h)_K - (p_h, \nabla\cdot\boldsymbol{v}_h)_K + (\boldsymbol{u}_h\cdot\nabla\psi_h \boldsymbol{E}, \boldsymbol{v}_h)_K - (q_h, \nabla\cdot\boldsymbol{u}_h)_K \\
    &\quad+ (\mathcal{K}(\psi_h),\phi_h)_K + (\boldsymbol{u}_h\cdot\nabla\psi_h, \phi_h)_K + \varepsilon(\nabla\psi, \nabla\phi)_K + (\mathcal{K}(\psi_h)\boldsymbol{E},\boldsymbol{v}_h)_K \bigg)  \\
    & \quad + \sum_{e \in \mathcal{E}_{\mathrm{Nav}}}
\bigg(
- \int_e \boldsymbol{n}^{t}\big(\mu \nabla\boldsymbol{u}_h - p_h \boldsymbol{I}\big)\boldsymbol{n}\,
(\boldsymbol{n}\cdot \boldsymbol{v}_h)\,\dd{s}
- \int_e \boldsymbol{n}^t\big(\mu \nabla\boldsymbol{v}_h - q_h \boldsymbol{I}\big)\boldsymbol{n}\, (\boldsymbol{n}\cdot \boldsymbol{u}_h)\,\dd{s} \\
&\quad
+ \int_e \beta \sum_{i=1}^{d-1}
(\boldsymbol{\tau}^i \cdot \boldsymbol{v}_h)\,
(\boldsymbol{\tau}^i \cdot \boldsymbol{u}_h)\,\dd{s}
+ \frac{\gamma}{h_e} \int_e 
(\boldsymbol{u}_h \cdot \boldsymbol{n})\,
(\boldsymbol{v}_h \cdot \boldsymbol{n})\,\dd{s}
\bigg),
\end{align*}
where $\boldsymbol{I}$ denotes the identity tensor in $\mathbb{R}^{n\times n}$ and $\gamma>0$ is the Nitsche parameter. 
Now \eqref{ns1} can be written in the following systematic form:
Find $(\boldsymbol{u}_h, p_h, \psi_h)\in \boldsymbol{V}_h \times Q_h \times \Phi_h$ such that

\begin{equation}\label{ns2}
\begin{aligned}
A_h(\boldsymbol{u}_h,\boldsymbol{v}_h) + B_h(\boldsymbol{v}_h,p_h) + \boldsymbol{c}_1(\boldsymbol{u}_h;\psi_h,\boldsymbol{v}_h) + (\mathcal{K}(\psi_h)\boldsymbol{E},\boldsymbol{v}_h)
&= F(\boldsymbol{v}_h),
\qquad && \forall \boldsymbol{v}_h \in \boldsymbol{V}_h, \\
B_h(\boldsymbol{u}_h,q_h) &= 0,
&& \forall q_h \in Q_h, \\
(\mathcal{K}(\psi_h),\phi_h) + \boldsymbol{c}_2(\boldsymbol{u}_h;\psi_h,\phi_h) + d(\psi_h,\phi_h)
&= G(\phi_h),
&& \forall \phi_h \in \Phi_h.
\end{aligned}
\end{equation}
where
\begin{align*}
A_h(\boldsymbol{u}_h,\boldsymbol{v}_h)
&:= a(\boldsymbol{u}_h,\boldsymbol{v}_h)
+ a_{\tau}^{\partial}(\boldsymbol{u}_h,\boldsymbol{v}_h)
+ a_{\gamma}^{\partial}(\boldsymbol{u}_h,\boldsymbol{v}_h)
- a_{c}^{\partial}(\boldsymbol{u}_h,\boldsymbol{v}_h)
- a_{c}^{\partial}(\boldsymbol{v}_h,\boldsymbol{u}_h), \\[0.5em]
B_h(\boldsymbol{u}_h,q_h)
&:= b(\boldsymbol{u}_h,q_h) + b^{\partial}(\boldsymbol{u}_h,q_h).
\end{align*}
with the forms defined as
\begin{align*}
a(\boldsymbol{u}_h,\boldsymbol{v}_h)
&:= \sum_{K\in\mathcal{T}_h}
\mu\,(\nabla\boldsymbol{u}_h, \nabla\boldsymbol{v}_h)_K, &b(\boldsymbol{u}_h,q_h)
&:= - \sum_{K\in\mathcal{T}_h}
(\operatorname{div}\boldsymbol{u}_h,q_h)_K, \\[0.5em]
a_{\tau}^{\partial}(\boldsymbol{u}_h,\boldsymbol{v}_h)
&:= \sum_{e\in\mathcal{E}_{\mathrm{Nav}}}
\int_e \beta \sum_{i=1}^{d-1}
(\boldsymbol{\tau}^i\cdot \boldsymbol{u}_h)
(\boldsymbol{\tau}^i\cdot \boldsymbol{v}_h)\,\dd{s},
&b^{\partial}(\boldsymbol{u}_h,q_h)
&:= \sum_{e\in\mathcal{E}_{\mathrm{Nav}}}
\int_e q_h\,(\boldsymbol{n}\cdot \boldsymbol{u}_h)\,\dd{s}, \\[0.5em]
a_{\gamma}^{\partial}(\boldsymbol{u}_h,\boldsymbol{v}_h)
&:= \sum_{e\in\mathcal{E}_{\mathrm{Nav}}}
\int_e \frac{\gamma}{h_e}
(\boldsymbol{u}_h\cdot \boldsymbol{n})
(\boldsymbol{v}_h\cdot \boldsymbol{n})\,\dd{s},
&c_1(\boldsymbol{u}_h;\psi_h,\boldsymbol{v}_h)
&:= \sum_{K\in\mathcal{T}_h}
(\boldsymbol{u}_h\cdot\nabla \psi_h \boldsymbol{E},\boldsymbol{v}_h)_K,
\\[0.5em]
a_{c}^{\partial}(\boldsymbol{u}_h,\boldsymbol{v}_h)
&:= \sum_{e\in\mathcal{E}_{\mathrm{Nav}}}
\int_e \mu \boldsymbol{n}^{t}\nabla\boldsymbol{u}_h\boldsymbol{n}\,
(\boldsymbol{n}\cdot \boldsymbol{v}_h)\,\dd{s}, 
&c_2(\boldsymbol{u}_h;\psi_h,\phi_h)
&:= \sum_{K\in\mathcal{T}_h}
(\boldsymbol{u}_h\cdot\nabla \psi_h,\phi_h)_K.
\\[0.5em]
\end{align*}

Now we present the well-known inverse and trace inequalities along with two fundamental results, namely the ellipticity of $A_h$ and the inf–sup condition for $B_h$.
\medskip
\begin{lemma}[Inverse inequality {\cite[Lemma~12.1]{ErnGuermond2021}}]\label{lma2.2}
Let $v_h \in V_h$. For each element $K \in \mathcal{T}_h$ and for any integers
$l,m \in \mathbb{N}$ with $0 \le m \le l$, there exists a positive constant
$C_4$, independent of $K$, such that
\[
\|v_h\|_{l,K} \le C_4\, h_K^{\,m-l}\, \|v_h\|_{m,K}.
\]
\end{lemma}
\begin{lemma}[Trace inequality I {\cite[Lemma~12.8]{ErnGuermond2021}}]\label{lma2.3}
Let $v_h \in V_h$. For each element $K \in \mathcal{T}_h$ and each facet
$e \subset \partial K$, there exists a positive constant $C_5$, independent of
$K$, such that
\[
\|v_h\|_{0,e} \le C_5\, h_K^{-1/2}\, \|v_h\|_{0,K}.
\]
\end{lemma}
\begin{lemma}[Trace inequality II {\cite[Lemma~12.15]{ErnGuermond2021}}] \label{lma2.4}
Let $v_h \in H^1(K)$, with $K \in \mathcal{T}_h$. 
Then, for any facet $e \subset \partial K$, there exists a constant 
$C_{\mathrm{tr}}$, independent of $K$, such that
\[
\|v_h\|_{0,e}
\le C_{\mathrm{tr}}
\left(
h_K^{-1/2} \|v_h\|_{0,K}
+ h_K^{1/2} \|\nabla v_h\|_{0,K}
\right).
\]
\end{lemma}
We now present the boundedness of the discrete bilinear forms with respect to the discrete energy norm $\|\cdot\|_{1,h}$.

\begin{lemma}\label{lma2.5}
There exist positive constants $C_6$ and $C_7$, independent of $h$, such that
\begin{align*}
|A_h(\boldsymbol{u}_h,\boldsymbol{v}_h)|
&\le C_6 \|\boldsymbol{u}_h\|_{1,h}\,\|\boldsymbol{v}_h\|_{1,h},
&& \forall\, \boldsymbol{u}_h,\boldsymbol{v}_h \in \boldsymbol{V}_h, \\
|c_1(\boldsymbol{w}_h;\psi_h,\boldsymbol{v}_h)|
&\le \bar{E}C_{\mathrm{Sob}}^2 \|\boldsymbol{w}_h\|_{1,h}\,
\|\psi_h\|_{1}\,\|\boldsymbol{v}_h\|_{1,h},
&& \forall\, \boldsymbol{w}_h,\boldsymbol{v}_h \in \boldsymbol{V}_h,\; \psi_h \in \Phi_h, \\
|c_2(\boldsymbol{w}_h;\psi_h,\phi_h)|
&\le C_{\mathrm{Sob}}^2 \|\boldsymbol{w}_h\|_{1,h}\,
\|\psi_h\|_{1}\,
\|\phi_h\|_{1},
&& \forall\, \boldsymbol{w}_h \in \boldsymbol{V}_h,\; \psi_h,\phi_h \in \Phi_h, \\
|B_h(\boldsymbol{v}_h,q_h)|
&\le C_7 \|\boldsymbol{v}_h\|_{1,h}\,\|q_h\|,
&&\forall\, \boldsymbol{v}_h \in \boldsymbol{V}_h,\; q_h \in Q_h.
\end{align*}
\end{lemma}

\begin{proof}
The proof of these inequalities follows directly from \Cref{lma2.2,lma2.3}, combined with the Sobolev embedding theorem and standard Cauchy–Schwarz and H\"older inequalities.
\end{proof}

The discrete kernel of $B_h$ is defined by
\[
Z_h := \{\boldsymbol{v}_h \in \boldsymbol{V}_h : B_h(\boldsymbol{v}_h,p_h) = 0,
\ \forall\, p_h \in Q_h \}.
\]
It can equivalently be written as
\begin{equation}
Z_h :=
\left\{
\boldsymbol{v}_h \in \boldsymbol{V}_h :
\sum_{K\in\mathcal{T}_h} \int_K p_h\,\operatorname{div}\boldsymbol{v}_h\,\mathrm{d}x
-
\sum_{e\in\mathcal{E}_{\mathrm{Nav}}} \int_e p_h\,(\boldsymbol{n}\cdot\boldsymbol{v}_h)\,\dd{s}
= 0,
\ \forall\, p_h \in Q_h
\right\}.
\end{equation}
The following lemma establishes the ellipticity of $A_h$ on $\boldsymbol{V}_h$.
\medskip
\begin{lemma}\label{lma2.6}
There exist positive constants $\gamma_0, C_0, C_S = C_S(\beta,\mu,\Omega)$, independent of $h$, such that
\begin{align*}
A_h(\boldsymbol{v}_h,\boldsymbol{v}_h)
\ge C_S \|\boldsymbol{v}_h\|_{1,h}^2,
\qquad \forall\, \boldsymbol{v}_h \in Z_h,
\end{align*}
where,
\[
C_S = \min\left\{\xi - \mu C_5^2 C_0,\; \gamma - \frac{\mu}{C_0}\right\},
\quad
\gamma \ge \gamma_0 > \frac{\mu}{C_0},
\quad
C_0 < \frac{\xi}{C_5^2},
\]
and $\xi$ is the coercivity constant defined in the continuous case.
\end{lemma}
\begin{proof}
    See \cite[Lemma 8]{bansal2024nitsche}.
\end{proof}
Next, we present the discrete version of inf-sup condition for the bilinear form $B_h$.
\begin{lemma}\label{lma2.7}
There exists a constant $\hat{\theta}>0$, independent of $h$, such that
\[
\sup_{0 \neq \boldsymbol{v}_h \in \boldsymbol{V}_h}
\frac{|B_h(\boldsymbol{v}_h,q_h)|}{\|\boldsymbol{v}_h\|_{1,h}}
\ge \hat{\theta}\,\|q_h\|,
\qquad \forall\, q_h \in Q_h.
\]
\end{lemma}
\begin{proof}
    See  \cite[Lemma 9]{bansal2024nitsche}.
\end{proof}

We introduce the bilinear form
\begin{equation}\label{eqch}
\mathcal{C}_h(\boldsymbol{u}_h,p_h,\psi_h;\boldsymbol{v}_h,q_h,\phi_h)
= A_h(\boldsymbol{u}_h,\boldsymbol{v}_h) + B_h(\boldsymbol{u}_h,q_h) + B_h(\boldsymbol{v}_h,p_h) + d(\psi_h,\phi_h).
\end{equation}

\begin{theorem}\label{thm2.8}
There exists a positive constant $\hat{\alpha} = \hat{\alpha}(C_S,\hat{\theta})$ such that
\[
\sup_{0 \neq (\boldsymbol{v}_h,q_h,\phi_h)\in \boldsymbol{V}_h \times Q_h \times \Phi_h}
\frac{\mathcal{C}_h(\boldsymbol{u}_h,p_h,\psi_h;\boldsymbol{v}_h,q_h,\phi_h)}
{\|(\boldsymbol{v}_h,q_h,\phi_h)\|}
\ge \hat{\alpha}\,\|(\boldsymbol{u}_h,p_h,\psi_h)\|,
\qquad
\forall\, (\boldsymbol{u}_h,p_h,\psi_h)\in \boldsymbol{V}_h \times Q_h, \times \Phi_h
\]
Here, $C_S$ and $\hat{\theta}$ denote the coercivity and inf-sup stability
constants, respectively.
\end{theorem}
\begin{proof}
Owing to \Cref{lma2.5}, the bilinear form $\mathcal{C}_h(\cdot\,;\cdot)$ is uniformly bounded.
Furthermore, the discrete ellipticity of $A_h$ on $Z_h$ (\Cref{lma2.6}) and the discrete
inf-sup condition for $B_h$ (\Cref{lma2.7}), together with Proposition~2.36 in~\cite{ern2004theory},
imply that the stated inf-sup condition is satisfied.
\end{proof}

\subsection{Well-posedness}
We apply Banach’s fixed-point theorem in combination with the Babu\v{s}ka–Brezzi theory \cite{boffi2013mixed} and the Minty–Browder theorem \cite{ciarlet2025linear}, closely following the approach in \cite{alsohaim2024analysis}. The incompressible flow equations are first decoupled from the nonlinear transport equation. Each subproblem is then analyzed separately, and the full coupled system is recovered by reconnecting the two through Banach’s fixed-point argument. First we define a set
\begin{align}\label{eqz}
    X = \{ \widehat{\psi}\in \Phi_h \,\colon \,\|\widehat{\psi}\|_1\leq C_S[2\bar{E}C_{\mathrm{Sob}}^2]^{-1}\,\},
\end{align}
for a fixed $\widehat{\psi}_h\in X$, we have to $(\boldsymbol{u}_h,p_h)\in \boldsymbol{V}_h\times Q_h$ such that
\begin{equation}\label{eqahuh}
\begin{aligned}
A_h(\boldsymbol{u}_h,\boldsymbol{v}_h)
+ B_h(\boldsymbol{v}_h,p_h)
+ c_1(\boldsymbol{u}_h;\widehat{\psi}_h, \boldsymbol{v}_h)
&= F_h^{\widehat{\psi}_h}(\boldsymbol{v}_h),
&& \forall \boldsymbol{v}_h \in \boldsymbol{V}_h,  \\
B_h(\boldsymbol{u}_h,q_h)
&= 0,
&& \forall q_h \in Q_h, 
\end{aligned}
\end{equation}
where
\begin{align*}
    F_h^{\widehat{\psi}_h}(\boldsymbol{v}_h) =
\int_\Omega \{ \boldsymbol{f} + [g - \mathcal{K}({\widehat{\psi}_h})]\boldsymbol{E} \} \cdot \boldsymbol{v}_h,
\end{align*}
Now using \Cref{lma2.5,lma2.6} along with the definition of $X$ given by \eqref{eqz}, we have the following boundedness and coercive properties
\begin{equation}\label{eqahcb}
    \begin{aligned}
    |A_h(\boldsymbol{u}_h,\boldsymbol{v}_h)
+ c_1(\boldsymbol{u}_h;\widehat{\psi}_h, \boldsymbol{v}_h)| &\leq (C_6+ \frac{C_S}{2})\|\boldsymbol{u}_h\|_{1,h}\|\boldsymbol{v}_h\|_{1,h}, \\
|A_h(\boldsymbol{v}_h,\boldsymbol{v}_h)
+ c_1(\boldsymbol{v}_h;\widehat{\psi}_h, \boldsymbol{v}_h)| &\geq \frac{C_S}{2}\|\boldsymbol{v}_h\|_{1,h}^2 .
\end{aligned}
\end{equation}
By using the continuity and boundedness properties of $B_h(\cdot\,,\cdot)$ established in \Cref{lma2.5,lma2.7}, and applying \cite[Theorem~4.2.3]{boffi2013mixed}, we conclude that the following operator
\[
\mathcal{S}^{\mathrm{flow}} : \Phi_h \to \boldsymbol{V}_h \times Q_h, 
\qquad 
\widehat{\psi}_h \mapsto \mathcal{S}^{\mathrm{flow}}(\widehat{\psi}_h) = \bigl(\mathcal{S}_1^{\mathrm{flow}}(\widehat{\psi}_h),\, \mathcal{S}_2^{\mathrm{flow}}(\widehat{\psi}_h)\bigr) = (\boldsymbol{u}_h,p_h)
\]
is well-defined. Moreover we have
\begin{align}\label{equhfg}
    \|\boldsymbol{u}_h\|_{1,h} \leq \frac{C_p}{C_S}\Big(\|\boldsymbol{f}\|+\bar{E}\|g\|+ \bar{K}\bar{E}\|\widehat{\psi}_h\|_1\Big).
\end{align}
Now for the second part, we define the set 
\begin{align}
    \boldsymbol{W} = \{ \, \widehat{\boldsymbol{u}}_h\in \boldsymbol{V}_h \,\colon\,\|\widehat{\boldsymbol{u}}_h\|_{1,h} \le \varepsilon [2C^2_{\mathrm{Sob}}(1+C_p^2) ]^{-1} \, \},
\end{align}
so for a fix $\widehat{\boldsymbol{u}}_h \in \boldsymbol{W}$, we have to find $\psi_h\in \Phi_h$ such that
\begin{align}\label{eqp2}
    (\mathcal{K}(\psi_h),\phi_h) + \boldsymbol{c}_2(\widehat{\boldsymbol{u}}_h;\psi_h,\phi_h) + d(\psi_h,\phi_h)
= G(\phi_h),
\qquad \forall \phi_h \in \Phi_h.
\end{align}
using Cauchy-Schwarz inequality and \Cref{lma2.5}, we have 
\begin{align*}
    | (\mathcal{K}(\psi_h),\phi_h) + \boldsymbol{c}_2(\widehat{\boldsymbol{u}}_h;\psi_h,\phi_h) + d(\psi_h,\phi_h) | \le \Big(\bar{K} + \frac{\varepsilon}{2(1+C_p^2)} + \varepsilon\Big)\|\psi_h\|_1\|\phi_h\|_1.
\end{align*}
Also using the coercivity of $d(\cdot\,,\cdot)$, definition of set $\boldsymbol{W}$ and \ref{(A2)}, we obtain the following strong monotonicity
\begin{align}\label{eqkpsi}
    (\mathcal{K}(\psi_{1_h})-\mathcal{K}(\psi_{2_h}),\psi_{1_h}-\psi_{2_h})   & + \boldsymbol{c}_2(\widehat{\boldsymbol{u}}_h;\psi_{1_h}-\psi_{2_h},\psi_{1_h}-\psi_{2_h}) + d(\psi_{1_h}-\psi_{2_h},\psi_{1_h}-\psi_{2_h}) \nonumber \\
    & \ge \frac{\varepsilon}{2(1+C_p^2)}\|\psi_{1_h}-\psi_{2_h}\|^2_1, \qquad     \forall \psi_{1_h},\psi_{2_h} \in \Phi_h.
\end{align}
By  the Minty–Browder theorem \cite[Theorem 9.14-1]{ciarlet2025linear}, there exist a unique $\psi_h\in \Phi_h$ that satisfy \eqref{eqp2} and the map 
\[
\mathcal{S}^{\mathrm{elec}} : \boldsymbol{V}_h \to \Phi_h,
\qquad
\widehat{u}_h \mapsto \mathcal{S}^{\mathrm{elec}}(\widehat{u}_h) = \psi_h
\]
is well-defined. Additionally, using \eqref{eqkpsi}, we have
\begin{align}
    \|\psi_h\|_1 \le 2\varepsilon^{-1}(1+C_p^2)\|g\|.\label{eqpsig}
\end{align}
We now introduce an operator $\boldsymbol{T}$, which is equivalent to the solution map of \eqref{ns2} and is defined as:
\[
\boldsymbol{T} : \boldsymbol{W} \to \boldsymbol{W},
\qquad 
\widehat{\boldsymbol{u}} \mapsto \boldsymbol{T}(\widehat{\boldsymbol{u}_h})
= \mathcal{S}_1^{\mathrm{flow}}\!\bigl(\mathcal{S}^{\mathrm{elec}}(\widehat{\boldsymbol{u}}_h)\bigr).
\]
\begin{lemma}\label{lma2.9}
  Assume  \ref{(A1)}--\ref{(A3)}s, and let $\boldsymbol{f}\in \boldsymbol{L}^2(\Omega)$ and $g \in L^2(\Omega)$ satisfy the small-data assumption
    \begin{align}\label{smc}
       \max\left\{ 4, 2C_p \right\}\frac{C_{\mathrm{Sob}}^2(1+C_P^2)}{\varepsilon C_S}(1+\bar{E}+ 2\bar{K}\bar{E}(1+C_P^2)\varepsilon^{-1})(\|\boldsymbol{f}\|+\|g\|) \le 1.
    \end{align}
    Then the operator $\boldsymbol{T}$ is well defined and $\boldsymbol{T}(\boldsymbol{W})\subseteq \boldsymbol{W}.$
\end{lemma}
\begin{proof}
   For $\widehat{\boldsymbol{u}} \in \boldsymbol{W}$, the small data assumption implies that
\begin{align*}
    \frac{4\bar{E}C_{\mathrm{Sob}}^2(1+C_P^2)}{\varepsilon C_S}\|g\| \leq 1.
\end{align*}
Combining this with \eqref{eqpsig}, we obtain
$
    \|\mathcal{S}^{\mathrm{elec}}(\widehat{\boldsymbol{u}}_h)\|_{1}
\le 
C_S\bigl(2\bar{E}C_{\mathrm{Sob}}^2\bigr)^{-1},
$
which shows that $\mathcal{S}^{\mathrm{elec}}(\widehat{\boldsymbol{u}}_h) \in X$. 
Since both $\mathcal{S}^{\mathrm{elec}}$ and $\mathcal{S}^{\mathrm{flow}}$ are well defined, it follows that the operator $\boldsymbol{T}$ is well defined on $\boldsymbol{W}$.
Moreover, from \labelcref{equhfg,eqpsig} and the smallness condition \eqref{smc}, we obtain
\begin{align*}
    \|\mathcal{S}_1^{\mathrm{flow}}(\mathcal{S}^{\mathrm{elec}}(\widehat{\boldsymbol{u}}))\|_{1,h} \,
    \le \frac{C_p}{C_S}\Big(\|f\|+\bar{E}\|g\|\,+\, 2\varepsilon^{-1}\bar{K}\bar{E}(1+C_p^2)\|g\|\Big) 
    \le \varepsilon\Big(2C^2_{\mathrm{Sob}}(1+C_p^2)\Big)^{-1}.
\end{align*}
Hence we prove that $\boldsymbol{T}(\boldsymbol{W})\subseteq \boldsymbol{W}$.
\end{proof}
\begin{theorem}
Assume that \ref{(A1)}--\ref{(A3)} and \eqref{smc} hold. Moreover, suppose that
\begin{align}\label{smc2}
    \frac{4\bar{E}C^2_{\mathrm{Sob}}}{C_S\varepsilon^2}(1+C_p^2)\Big[\varepsilon+ 2\bar{K}(1+C_p^2)\Big]\|g\| < 1.
\end{align}
Then the operator $\boldsymbol{T}$ admits a unique fixed point $\boldsymbol{u}_h \in \boldsymbol{W}$. 
Equivalently, problem \eqref{ns2} admits a unique solution 
$(\boldsymbol{u}_h, p_h, \psi_h) \in \boldsymbol{V}_h \times Q_h \times \Phi_h$, 
with $\boldsymbol{u}_h \in \boldsymbol{W}$.
\end{theorem}
\begin{proof}
    Let $\widehat{\boldsymbol{u}}_{1_h}, \widehat{\boldsymbol{u}}_{2_h} \in \boldsymbol{W}$ and let $\psi_{1_h}, \psi_{2_h} \in \Phi_h$ satisfy
\[
\mathcal{S}^{\mathrm{elec}}(\widehat{\boldsymbol{u}}_{1_h}) = \psi_{1_h},
\qquad
\mathcal{S}^{\mathrm{elec}}(\widehat{\boldsymbol{u}}_{2_h}) = \psi_{2_h}.
\]
Then, from \eqref{eqp2}, we obtain
\begin{align*}
        (\mathcal{K}(\psi_{1_h}) - \mathcal{K}(\psi_{2_h}),\phi_h) 
        + \boldsymbol{c}_2(\widehat{\boldsymbol{u}}_{1_h};\psi_{1_h},\phi_h) 
        + \boldsymbol{c}_2(\widehat{\boldsymbol{u}}_{2_h};\psi_{2_h},\phi_h) 
        + d(\psi_{1_h} - \psi_{2_h},\phi_h)
        = 0     \qquad \forall \phi_h \in \Phi_h,
    \end{align*}
adding and subtracting 
$\boldsymbol{c}_2(\widehat{\boldsymbol{u}}_{1_h};\psi_{2_h},\phi_h)$ 
and choosing $\phi_h = \psi_{1_h} - \psi_{2_h}$, we obtain
\begin{align*}
        (\mathcal{K}(\psi_{1_h}) - \mathcal{K}(\psi_{2_h}),\psi_{1_h} - \psi_{2_h}) \,
        +&\, \boldsymbol{c}_2(\widehat{\boldsymbol{u}}_{1_h};\psi_{1_h}-\psi_{2_h},\psi_{1_h} - \psi_{2_h}) 
        + \boldsymbol{c}_2(\widehat{\boldsymbol{u}}_{1_h} -\widehat{\boldsymbol{u}}_{2_h};\psi_{2_h},\psi_{1_h} - \psi_{2_h}) \\
        &+ d(\psi_{1_h} - \psi_{2_h},\psi_{1_h} - \psi_{2_h})
        = 0,
    \end{align*}
using the boundedness of 
$\boldsymbol{c}_2(\cdot\,;\cdot,\cdot)$ and \eqref{eqkpsi}, we deduce
\begin{align}\label{eqseu}
    \|\mathcal{S}^{\mathrm{elec}}(\widehat{\boldsymbol{u}}_{1_h})
    - \mathcal{S}^{\mathrm{elec}}(\widehat{\boldsymbol{u}}_{2_h})\|_{1}
    = \|\psi_{1_h}-\psi_{2_h}\|_{1}
    \le 
    \frac{2}{\varepsilon}
    C_{\mathrm{Sob}}^{2}(1+C_p^{2})
    \|\psi_{2_h}\|_{1}
    \|\widehat{\boldsymbol{u}}_{1_h}-\widehat{\boldsymbol{u}}_{2_h}\|_{1,h}.
\end{align}
Similarly, let $\widehat{\psi}_{1_h},\widehat{\psi}_{2_h}\in \Phi_h$ and 
$(\boldsymbol{u}_{1_h},p_{1_h}),(\boldsymbol{u}_{2_h},p_{2_h})\in \boldsymbol{V}_h \times Q_h$ such that
\[  
\mathcal{S}^{\mathrm{flow}}(\widehat{\psi}_{1_h}) = (\boldsymbol{u}_{1_h},p_{1_h}),
\qquad
\mathcal{S}^{\mathrm{flow}}(\widehat{\psi}_{2_h}) = (\boldsymbol{u}_{2_h},p_{2_h}) .
\]
From \eqref{eqahuh}, we have
\begin{align*}
    A_h(\boldsymbol{u}_{1_h}-\boldsymbol{u}_{2_h},\boldsymbol{v}_h)
+ B_h(\boldsymbol{v}_h,p_{1_h}-p_{2_h})
+ c_1(\boldsymbol{u}_{1_h};\widehat{\psi}_{1_h}, \boldsymbol{v}_h)
&= F_h^{\widehat{\psi}_{1_h}}(\boldsymbol{v}_h) - F_h^{\widehat{\psi}_{2_h}}(\boldsymbol{v}_h),
\end{align*}
adding and subtracting $c_1(\boldsymbol{u}_{2_h};\widehat{\psi}_{1_h}, \boldsymbol{v}_h)$ and choosing $\boldsymbol{v}_h = \boldsymbol{u}_{1_h}-\boldsymbol{u}_{2_h}$, we obtain
\begin{align*}
    A_h(\boldsymbol{u}_{1_h}-\boldsymbol{u}_{2_h},\boldsymbol{u}_{1_h}-\boldsymbol{u}_{2_h})\,
+ \, B_h(\boldsymbol{u}_{1_h}-\boldsymbol{u}_{2_h},p_{1_h}-p_{2_h})\,
+&\, c_1(\boldsymbol{u}_{1_h}-\boldsymbol{u}_{2_h};\widehat{\psi}_{1_h}, \boldsymbol{u}_{1_h}-\boldsymbol{u}_{2_h})\\
+\, c_1(\boldsymbol{u}_{2_h};\widehat{\psi}_{1_h}-\widehat{\psi}_{2_h}, \boldsymbol{u}_{1_h}-\boldsymbol{u}_{2_h})
&= F_h^{\widehat{\psi}_{1_h}}(\boldsymbol{u}_{1_h}-\boldsymbol{u}_{2_h}) - F_h^{\widehat{\psi}_{2_h}}(\boldsymbol{u}_{1_h}-\boldsymbol{u}_{2_h}),
\end{align*}
using \eqref{eqahcb} and the second equation of \eqref{eqahuh}, we obtain
\begin{align*}
    \frac{C_S}{2}\|\boldsymbol{u}_{1_h}-\boldsymbol{u}_{2_h}\|_{1,h}^2 
    \le \bar{E}\Big(C^2_{\mathrm{Sob}}\|\boldsymbol{u}_{2_h}\|_{1,h}+\bar{K}\Big)
    \|\widehat{\psi}_{1_h}-\widehat{\psi}_{2_h}\|_1
    \|\boldsymbol{u}_{1_h}-\boldsymbol{u}_{2_h}\|_{1,h}.
\end{align*}
Simplifying, we obtain
\begin{align}\label{eqs1psi}
    \|\mathcal{S}^{\mathrm{flow}}_1(\widehat{\psi}_{1_h})  
    -  \mathcal{S}^{\mathrm{flow}}_1(\widehat{\psi}_{2_h})\|_{1,h}
    = \|\boldsymbol{u}_{1_h}-\boldsymbol{u}_{2_h}\|_{1,h} 
    \le \frac{2\bar{E}}{C_S}
    \Big(C^2_{\mathrm{Sob}}\|\boldsymbol{u}_{2_h}\|_{1,h}+\bar{K}\Big)
    \|\widehat{\psi}_{1_h}-\widehat{\psi}_{2_h}\|_1.
\end{align}
Combining \labelcref{eqseu,eqs1psi}, we obtain
\begin{align*}
    \|\boldsymbol{T}(\widehat{\boldsymbol{u}}_{1_h}) - \boldsymbol{T}(\widehat{\boldsymbol{u}}_{2_h})\|_{1,h} 
    &= \norm{\mathcal{S}^{\mathrm{flow}}_1\Big(\mathcal{S}^{\mathrm{elec}}(\widehat{\boldsymbol{u}}_{1_h}) \Big)
    -\mathcal{S}^{\mathrm{flow}}_1\Big(\mathcal{S}^{\mathrm{elec}}(\widehat{\boldsymbol{u}}_{2_h}) \Big)} \\[1ex]
    &\le \frac{4\bar{E}C_{\mathrm{Sob}}^{2}}{\varepsilon C_S}
    (1+C_p^{2}) 
    \Big(C^2_{\mathrm{Sob}}\|\boldsymbol{u}_{2_h}\|_{1,h}+\bar{K}\Big)
    \| \mathcal{S}^{\mathrm{elec}}(\widehat{\boldsymbol{u}}_{2_h})  \|_1 
    \|\widehat{\boldsymbol{u}}_{1_h}-\widehat{\boldsymbol{u}}_{2_h}\|_{1,h}.
\end{align*}
Since $ \mathcal{S}^{\mathrm{elec}}(\widehat{\boldsymbol{u}}_{2_h})$ satisfies \eqref{eqpsig} and 
$\widehat{\boldsymbol{u}}_{2_h}\in \boldsymbol{W}$, and by \Cref{lma2.9} we have 
$\boldsymbol{u}_{2_h}\in \boldsymbol{W}$, it follows that
\begin{align*}
    \|\boldsymbol{T}(\widehat{\boldsymbol{u}}_{1_h}) - \boldsymbol{T}(\widehat{\boldsymbol{u}}_{2_h})\|_{1,h} 
     \le  \frac{4\bar{E}C^2_{\mathrm{Sob}}}{C_S\varepsilon^2}
     (1+C_p^2)
     \Big[\varepsilon+ 2\bar{K}(1+C_p^2)\Big]
     \|g\|
     \|\widehat{\boldsymbol{u}}_{1_h}-\widehat{\boldsymbol{u}}_{2_h}\|_{1,h}.
\end{align*}
Together with \eqref{smc2} and Banach’s fixed-point theorem, it follows that $\boldsymbol{T}$ admits a unique fixed point in $\boldsymbol{W}$.
\end{proof}

\section{A \textit{priori} error bounds}\label{sec:pri}
In this section, we derive the convergence results for Nitsche’s scheme \eqref{ns2}. 
To this end, we first introduce the interpolation operator and recall its approximation properties. 
Let $I_h$ denote the Lagrange interpolation operator. Under the standard assumptions, the following approximation estimates hold \cite{brenner2008mathematical}.
\medskip
\begin{lemma}\label{lma3.1}
There exist constants $C_{10}>0$ and $C_{11}>0$, independent of $h$, such that for
each $u \in H^{l+1}(K)$ with $0 \le l \le k$, the following estimates hold:
\begin{align*}
\|u - I_h(u)\|_{L^2(K)}
&\le C_{10}\,  h_K^{\,l+1}\, |u|_{H^{l+1}(K)}, \\[0.5em]
|u - I_h(u)|_{H^1(K)}
&\le C_{11}\,  h_K^{\,l}\, |u|_{H^{l+1}(K)},
\end{align*}
where $h_K$ denotes the diameter of $K$, and $k$ is the polynomial degree.
\end{lemma}

Next, we introduce the sets
\begin{equation}\label{kh}
\begin{aligned}
    K_{1_h} &= \{ \boldsymbol{v}_h \in \boldsymbol{V}_h\, \colon \,\|\boldsymbol{v}_h\|_{1,h} \,\leq \frac{1}{4}\hat{\alpha}\}, \\
    K_{2_h} &= \{ \phi_h \in \Phi_h\, \colon\, \|\phi_h\|_{1}\, \leq \frac{1}{4}\hat{\alpha} \},
\end{aligned}
\end{equation}
where $\hat{\alpha}$ is the constant defined in \Cref{thm2.8}. 
For a given $(\widetilde{\boldsymbol{u}}, \widetilde{\psi}) \in K_{1_h}\times K_{2_h}$, we define the bilinear form $ \mathcal{A}_h^{(\widetilde{\boldsymbol{u}}, \widetilde{\psi})}(\cdot\,;\cdot)$ by
\begin{equation}\label{awh}
 \mathcal{A}_h^{(\widetilde{\boldsymbol{u}}, \widetilde{\psi})}(\boldsymbol{u}_h,p_h,\psi_h;\boldsymbol{v}_h,q_h,\phi_h)
= \mathcal{C}_h(\boldsymbol{u}_h,p_h,\psi_h;\boldsymbol{v}_h,q_h,\phi_h)
+ c_1(\boldsymbol{u}_h;\widetilde{\psi},\boldsymbol{v}_h)+ c_2(\widetilde{\boldsymbol{u}};\psi_h,\phi_h),
\end{equation}
where $\mathcal{C}_h$ is defined in \eqref{eqch}.

The following lemma establishes an inf-sup stability estimate for the above bilinear form.
\begin{lemma}\label{lma3.2}
Assume that $(\widetilde{\boldsymbol{u}}, \widetilde{\psi}) \in K_{1_h}\times K_{2_h}$. 
Then the bilinear form $ \mathcal{A}_h^{(\widetilde{\boldsymbol{u}}, \widetilde{\psi})}(\cdot\,;\cdot)$ satisfies the estimate
 \begin{equation*}
\sup_{0 \neq (\boldsymbol{v}_h,q_h,\phi_h)\in \boldsymbol{V}_h \times Q_h \times \Phi_h}
\frac{\mathcal{A}_h^{(\widetilde{\boldsymbol{u}}, \widetilde{\psi})}(\boldsymbol{u}_h,p_h,\psi_h;\boldsymbol{v}_h,q_h,\phi_h)}
{\|(\boldsymbol{v}_h,q_h,\phi_h)\|}
\ge \frac{\hat{\alpha}}{2} \|(\boldsymbol{u}_h,p_h,\psi_h)\|,\qquad  \forall\, (\boldsymbol{u}_h,p_h,\psi_h)\in \boldsymbol{V}_h \times Q_h \times \Phi_h . 
\end{equation*}
\end{lemma}
\begin{proof}
Consider $(\boldsymbol{u}_h,p_h,\psi_h),(\hat{\boldsymbol{v}}_h,\hat{q}_h,\hat{\phi}_h) \in \boldsymbol{V}_h \times Q_h \times \Phi_h$ with $(\hat{\boldsymbol{v}}_h,\hat{q}_h,\hat{\phi}_h)\neq 0$. From \Cref{lma2.5}, we observe that
\begin{align*}
&\sup_{0 \neq (\boldsymbol{v}_h,q_h,\phi_h)\in \boldsymbol{V}_h \times Q_h \times \Phi_h}
\frac{\mathcal{A}_h^{(\widetilde{\boldsymbol{u}}, \widetilde{\psi})}(\boldsymbol{u}_h,p_h,\psi_h;\boldsymbol{v}_h,q_h,\phi_h)}
{\|(\boldsymbol{v}_h,q_h,\phi_h)\|}
\ge
\frac{\big|\mathcal{C}_h(\boldsymbol{u}_h,p_h,\psi_h;\hat{\boldsymbol{v}}_h,\hat{q}_h, \hat{\phi}_h)\big|}
{\|(\hat{\boldsymbol{v}}_h,\hat{q}_h, \hat{\phi}_h)\|} \\& \qquad \qquad \qquad \qquad \qquad \qquad \qquad \qquad \qquad \qquad \qquad \qquad 
-
\frac{\big| c_1(\boldsymbol{u}_h;\widetilde{\psi},\hat{\boldsymbol{v}}_h) \big| + \big| c_2(\widetilde{\boldsymbol{u}};\psi_h,\hat{\phi}_h))\big|}
{\|(\hat{\boldsymbol{v}}_h,\hat{q}_h, \hat{\phi}_h)\|}, \\
& \qquad \qquad \qquad \qquad   \ge
\frac{\big|\mathcal{C}_h(\boldsymbol{u}_h,p_h,\psi_h;\hat{\boldsymbol{v}}_h,\hat{q}_h, \hat{\phi}_h)\big|}
{\|(\hat{\boldsymbol{v}}_h,\hat{q}_h, \hat{\phi}_h)\|}
- C ( \|\widetilde{\boldsymbol{u}}\|_{1,h} + \|\widetilde{\psi}\|_1)\,\|(\boldsymbol{u}_h,p_h,\psi_h)\|,
\end{align*}
combining this estimate with \Cref{thm2.8} and using the fact that
$(\hat{\boldsymbol{v}}_h,\hat{q}_h, \hat{\phi}_h)$ is arbitrary, we obtain
\begin{equation*}
\sup_{0 \neq (\boldsymbol{v}_h,q_h,\phi_h)\in \boldsymbol{V}_h \times Q_h \times \Phi_h}
\frac{\mathcal{A}_h^{(\widetilde{\boldsymbol{u}}, \widetilde{\psi})}(\boldsymbol{u}_h,p_h,\psi_h;\boldsymbol{v}_h,q_h,\phi_h)}
{\|(\boldsymbol{v}_h,q_h,\phi_h)\|}
\ge
C\big(\hat{\alpha} - \|\widetilde{\boldsymbol{u}}\|_{1,h} - \|\widetilde{\psi}\|_1\big)\,
\|(\boldsymbol{u}_h,p_h,\psi_h)\|,
\end{equation*}
where $\hat{\alpha}$ does not depend on $\widetilde{\boldsymbol{u}}$ and $\widetilde{\psi}$. Consequently, invoking the definitions of the sets $K_{1_h}$ and $K_{2_h}$ in \eqref{kh} yields the desired result.
\end{proof}

To proceed further, for any $(\boldsymbol{z}_h,\zeta_h, \omega_h)\in \boldsymbol{V}_h \times Q_h \times \Phi_h$, we decompose the errors as
\begin{equation}\label{de}
\begin{aligned}
    e_{\boldsymbol{u}}  &= \boldsymbol{u} - \boldsymbol{u}_h = (\boldsymbol{u} - \boldsymbol{z}_h) + (\boldsymbol{z}_h - \boldsymbol{u}_h)= \xi_{\boldsymbol{u}} + \chi_{\boldsymbol{u}},\\
    e_p  &= p - p_h = (p - \zeta_h) + (\zeta_h - p_h)= \xi_p + \chi_p ,\\
    e_{\psi} &= \psi-\psi_h = (\psi- \omega_h) + (\omega_h -\psi_h) =\xi_{\psi} + \chi_{\psi}.
\end{aligned}
\end{equation}
Now we provide the corresponding C\`ea’s estimate.
\begin{lemma}\label{lma3.3}
    Suppose $(\boldsymbol{u},p,\psi) \in \boldsymbol{V}\times Q \times \Phi$ and $(\boldsymbol{u}_h,p_h,\psi_h) \in \boldsymbol{V}_h\times Q_h \times \Phi_h$ be the solution of the problem \eqref{wf} and \eqref{ns2} respectively. Assume  \ref{(A1)}--\ref{(A3)}, and that $\|\boldsymbol{u}\|_1 \le M$ and $\|\psi\|_1 \le M$, where $M < \frac{\hat{\alpha}}{2}$
    then there exist a positive constant $C_{cea} = C_{cea}(\mu, \beta, \Omega)$, independent of $h$, such that 
    \begin{align*}
        \|(\boldsymbol{u}-\boldsymbol{u}_h, p-p_h,\psi-\psi_h)\| \le C_{cea} \inf_{0 \neq(\boldsymbol{v}_h, q_h,\phi_h)\in \boldsymbol{V}_h \times Q_h\times \Phi_h} \|(\boldsymbol{u}-\boldsymbol{v}_h, p-q_h,\psi-\phi_h)\|.
    \end{align*}
\end{lemma}
\begin{proof}
Considering \eqref{wf} and \eqref{ns2}, we observe the the following orthogonality result
\begin{align}\label{gop}
    \mathcal{C}_h(e_{\boldsymbol{u}},&\, e_p, e_\psi;\boldsymbol{v}_h,q_h,\phi_h) + c_1(\boldsymbol{u};\psi,\boldsymbol{v}_h) - c_1(\boldsymbol{u}_h;\psi_h,\boldsymbol{v}_h)+ ((\mathcal{K}(\psi)-\mathcal{K}(\psi_h))\boldsymbol{E}, \boldsymbol{v}_h) \nonumber\\
    &+ c_2(\boldsymbol{u};\psi,\phi_h)- c_2(\boldsymbol{u}_h;\psi_h,\phi_h) +(\mathcal{K}(\psi)-\mathcal{K}(\psi_h),\phi_h) = 0, \qquad \forall (\boldsymbol{v}_h, q_h,\phi_h)\in \boldsymbol{V}_h \times Q_h\times \Phi_h.
\end{align}
Now using \labelcref{awh,gop}, for any $(\boldsymbol{v}_h,q_h,\phi_h)\in \boldsymbol{V}_h \times Q_h \times \Phi_h$, we have the following relation 
\begin{align}
\mathcal{A}_h^{(\boldsymbol{u}_h, \psi_h)}(\chi_{\boldsymbol{u}},\chi_p,\chi_\psi;\boldsymbol{v}_h,q_h,\phi_h) 
& = 
\mathcal{C}_h(\chi_{\boldsymbol{u}},\chi_p,\chi_\psi;\boldsymbol{v}_h,q_h,\phi_h) 
+ c_1\bigl(\chi_{\boldsymbol{u}}, \psi_h, \boldsymbol{v}_h \bigr) 
+ c_2\bigl(\boldsymbol{u}_h, \chi_\psi, \phi_h \bigr), \nonumber \\[0.3em]
& =- \mathcal{C}_h(\xi_{\boldsymbol{u}}, \xi_p, \xi_\psi ; \boldsymbol{v}_h, q_h, \phi_h ) 
+ \mathcal{C}_h(e_{\boldsymbol{u}}, e_p, e_\psi ; \boldsymbol{v}_h, q_h, \phi_h ) 
+ c_1\bigl(\chi_{\boldsymbol{u}}, \psi_h, \boldsymbol{v}_h \bigr) 
\nonumber \\[0.3em] &  \quad + c_2\bigl(\boldsymbol{u}_h, \chi_\psi, \phi_h \bigr), \nonumber\\[0.3em]
&=- \mathcal{C}_h(\xi_{\boldsymbol{u}}, \xi_p, \xi_\psi ; \boldsymbol{v}_h, q_h, \phi_h )  
- c_1(\boldsymbol{u};\psi,\boldsymbol{v}_h) 
+ c_1(\boldsymbol{u}_h;\psi_h,\boldsymbol{v}_h) \nonumber\\[0.3em]
&\quad
+ c_1\bigl(\chi_{\boldsymbol{u}}, \psi_h, \boldsymbol{v}_h\bigr) 
 - c_2(\boldsymbol{u};\psi,\phi_h)   + c_2(\boldsymbol{u}_h;\psi_h,\phi_h) 
+ c_2\bigl(\boldsymbol{u}_h, \chi_\psi, \phi_h \bigr) \nonumber\\[0.3em]
&\quad
- ((\mathcal{K}(\psi)-\mathcal{K}(\psi_h))\boldsymbol{E}, \boldsymbol{v}_h) 
-(\mathcal{K}(\psi)-\mathcal{K}(\psi_h),\phi_h).\label{eq15}
\end{align}
Since,
\begin{equation}\label{eq16}
    \begin{aligned}
        c_1(\boldsymbol{u};\psi,\boldsymbol{v}_h) = c_1(\boldsymbol{u};\psi-\psi_h,\boldsymbol{v}_h) + c_1(\boldsymbol{u};\psi_h,\boldsymbol{v}_h),\\
        c_2(\boldsymbol{u};\psi,\phi_h) = c_2(\boldsymbol{u}-\boldsymbol{u}_h;\psi,\phi_h) + c_2(\boldsymbol{u}_h;\psi,\phi_h).
    \end{aligned}
\end{equation}
Using \eqref{eq16}, we have
\begin{align}
    - c_1(\boldsymbol{u};\psi,\boldsymbol{v}_h) 
+ c_1(\boldsymbol{u}_h;\psi_h,\boldsymbol{v}_h) 
+ c_1\bigl(\chi_{\boldsymbol{u}}; \psi_h, \boldsymbol{v}_h\bigr) &=  
- c_1(\boldsymbol{u};\psi-\psi_h,\boldsymbol{v}_h) 
- c_1(\boldsymbol{u} - \boldsymbol{u}_h;\psi_h,\boldsymbol{v}_h)
+ c_1\bigl(\chi_{\boldsymbol{u}}; \psi_h, \boldsymbol{v}_h\bigr) \nonumber\\
&= - c_1(\boldsymbol{u};\chi_{\psi}+\xi_{\psi},\boldsymbol{v}_h) 
- c_1\bigl(\xi_{\boldsymbol{u}}; \psi_h, \boldsymbol{v}_h\bigr), \label{eqc1} \\
    - c_2(\boldsymbol{u};\psi,\phi_h)  + c_2(\boldsymbol{u}_h;\psi_h,\phi_h) 
+ c_2\bigl(\boldsymbol{u}_h, \chi_\psi, \phi_h \bigr)
&= - c_2(\chi_{\boldsymbol{u}}+\xi_{\boldsymbol{u}}; \psi, \phi_h)
- c_2(\boldsymbol{u}_h;\xi_{\psi}, \phi_h) \label{eqc2}. 
\end{align}
Using \labelcref{eqc1,eqc2} in \eqref{eq15} and together with the definition of $\mathcal{C}_h$ provided in \eqref{eqch}, we have
\begin{align}
    \mathcal{A}_{\boldsymbol{u}_h}(\chi_{\boldsymbol{u}},\chi_p,\chi_\psi;\boldsymbol{v}_h,q_h,\phi_h) 
    &= -A_h(\xi_{\boldsymbol{u}},\boldsymbol{v}_h) - B_h(\xi_{\boldsymbol{u}},q_h) - B_h(\boldsymbol{v}_h,\xi_p) - d(\xi_{\psi},\phi_h)
    - c_1(\boldsymbol{u};\chi_{\psi}+\xi_{\psi},\boldsymbol{v}_h)  \nonumber\\[0.3em]
   & \qquad- c_1\bigl(\xi_{\boldsymbol{u}}; \psi_h, \boldsymbol{v}_h\bigr)
- c_2(\chi_{\boldsymbol{u}}+\xi_{\boldsymbol{u}}; \psi, \phi_h) 
- c_2(\boldsymbol{u}_h;\xi_{\psi}, \phi_h)\nonumber\\[0.3em]
   & \qquad- ((\mathcal{K}(\psi)-\mathcal{K}(\psi_h))\boldsymbol{E}, \boldsymbol{v}_h) -(\mathcal{K}(\psi)-\mathcal{K}(\psi_h),\phi_h),
\end{align}
using \Cref{lma2.5,lma3.2}, we get
\begin{align*}
    \frac{\hat{\alpha}}{2} \|(\chi_{\boldsymbol{u}},\chi_p,\chi_\psi)\| &\lesssim \frac{1}{\|(\boldsymbol{v}_h,q_h,\phi_h)\|}\bigg[\Big( \|\xi_{\boldsymbol{u}}\|_{1,h} + \|\xi_p\| +\|\boldsymbol{u}\|_1 \|\xi_{\psi}\|_{1} + \|\chi_{\psi}\|_1\|\boldsymbol{u}\|_1 + \|\xi_{\boldsymbol{u}}\|_{1,h}\|\psi_h\|_1\Big)\|\boldsymbol{v}_h\|_{1,h}  \\
    &\qquad + \|\xi_{\boldsymbol{u}}\|_{1,h}\|q_h\|
    + \Big(\|\xi_{\psi}\|_1 +\|\chi_{\boldsymbol{u}}\|_{1,h} \|\psi\|_1 + \|\xi_{\boldsymbol{u}}\|_{1,h} \|\psi\|_1 +\|\boldsymbol{u}_h\|_{1,h}\|\xi_{\psi}\|_1\Big)\|\phi_h\|_1 \\
    & \qquad +\epsilon (\|\xi_{\psi}\|_1 + \|\chi_{\psi}\|_1) \Big(\frac{1}{\epsilon}\|v_h\|_{1,h}\Big) +\epsilon (\|\xi_{\psi}\|_1 + \|\chi_{\psi}\|_1) \Big(\frac{1}{\epsilon}\|\phi_h\|_1\Big) \bigg],\\
    &\lesssim \|\xi_{\boldsymbol{u}}\|_{1,h}(1+\|\psi_h\|_1)+ \|\xi_p\| + (1+\|\boldsymbol{u}\|_1+\|\boldsymbol{u}_h\|_{1,h} + \epsilon)\|\xi_{\psi}\|_{1,h} 
    + \|\chi_{\boldsymbol{u}}\|_{1,h} \|\psi\|_1 + (\epsilon + \|\boldsymbol{u}\|_1) \|\chi_{\psi}\|_1  ,
\end{align*}
so
\begin{align*}
    \Big(1-\frac{2}{\hat{\alpha}}\|\psi\|_1\Big)\|\chi_{\boldsymbol{u}}\|_{1,h} + \|\chi_p\|+ \Big(1 - \frac{2}{\hat{\alpha}}\|\boldsymbol{u}\|_1 -\frac{2\epsilon}{\hat{\alpha}}\Big)\|\chi_{\psi}\|_1 
    \lesssim \frac{2}{\hat{\alpha}}\Big[(1+\|\psi_h\|_1) \|\xi_{\boldsymbol{u}}\|_{1,h} + \|\xi_p\| + (1+\|\boldsymbol{u}\|_1+\|\boldsymbol{u}_h\|_{1,h} + \epsilon)\|\xi_{\psi}\|_1\Big].
\end{align*}
Taking into account that $\boldsymbol{u}_h \in K_{1_h},\, \psi_h \in K_{2_h}$, the bounds on 
$\boldsymbol{u}$ and $\boldsymbol{\psi}$, and choosing $\varepsilon > 0$ sufficiently small, we obtain
\begin{align}\label{eq17}
    \|\chi_{\boldsymbol{u}}\|_{1,h} + \|\chi_p\|+ \|\chi_{\psi}\|_1\, \lesssim \|\xi_{\boldsymbol{u}}\|_{1,h} + \|\xi_p\| + \|\xi_{\psi}\|_1.
\end{align}
Applying triangular inequality in \eqref{de} and using \eqref{eq17}, we obtain
\begin{align*}
    \|(e_{\boldsymbol{u}}, e_p, e_{\psi})\|\leq 
    \|(\xi_{\boldsymbol{u}}, \xi_p, \xi_{\psi})\| + \|(\chi_{\boldsymbol{u}}, \chi_p, \chi_{\psi})\| \lesssim  \|(\xi_{\boldsymbol{u}}, \xi_p, \xi_{\psi})\|.
\end{align*}
Since $(\boldsymbol{z}_h,\zeta_h, \omega_h)\in \boldsymbol{V}_h \times Q_h \times \Phi_h$, are arbitrary, this concludes the proof.
\end{proof}
\medskip
\begin{theorem}
    Let $(\boldsymbol{u}, p, \psi) \in \boldsymbol{V} \times Q \times \Phi$ and 
$(\boldsymbol{u}_h, p_h, \psi_h) \in \boldsymbol{V}_h \times Q_h \times \Phi_h$
be the solutions of problems \eqref{wf} and \eqref{ns2}, respectively. Assume  \ref{(A1)}--\ref{(A3)}, and that $\|\boldsymbol{u}\|_1 \le M$ and $\|\psi\|_1 \le M$, where $M < \frac{\hat{\alpha}}{2}$. Moreover, suppose that
$\boldsymbol{u} \in \boldsymbol{V} \cap \boldsymbol{H}^{l+1}(\Omega), 
 p \in Q \cap H^{l}(\Omega), 
 \psi \in \Phi \cap H^{l+1}(\Omega)$ with $l \geq 1$.
Then there exists a constant $C = C(\mu, \beta, \Omega)$, independent of $h$, such that
    \begin{align*}
         \|(\boldsymbol{u}-\boldsymbol{u}_h, p-p_h,\psi-\psi_h)\| \le C h^l (|\boldsymbol{u}|_{l+1}+ |p|_{l} + |\psi|_{l+1}).
    \end{align*}
\end{theorem}
\begin{proof}
    The proof follows directly from \Cref{lma3.1,lma3.3}.
\end{proof}

\section{A \textit{posteriori} error  boundes}\label{sec:post}
For each $K\in \mathcal{T}_h$ and each $e\in \mathcal{E}_h$, we define element-wise and facet-wise residuals as follows:
\begin{subequations}
    \begin{align}
        \boldsymbol{R}_{K} &= \Big\{ \boldsymbol{f}_h + [g_h-\mathcal{K}(\psi_h)]\boldsymbol{E} + \mu \Delta \boldsymbol{u}_h-\nabla p_h - \boldsymbol{u}_h\cdot\nabla \psi_h \boldsymbol{E}\Big\}|_K, \\[1.2ex]
        R_{1,K} &= \{\nabla \cdot\boldsymbol{u}_h\}|_K,\qquad R_{2,K} = \{g_h - \mathcal{K}(\psi_h)-\boldsymbol{u}_h\cdot\nabla\psi_h + \varepsilon\Delta\psi_h\}|_K, \\[1.2ex]
        \boldsymbol{R}_e &=
            \begin{cases}
            \displaystyle \frac{1}{2}
            \llbracket \bigl( p_h \boldsymbol{I} - \mu\nabla \boldsymbol{u}_h\bigr)\boldsymbol{n} \rrbracket,
            & \text{for } e \in \mathcal{E}_{\Omega}, \\[1.2ex]
            0, & \text{for } e \in \Gamma,
            \end{cases}\\[1.2ex]
        R_{1,e} &=
            \begin{cases}
            \displaystyle \frac{1}{2}
            \llbracket \varepsilon\nabla\psi_h\cdot \boldsymbol{n} \rrbracket,
            & \text{for } e \in \mathcal{E}_{\Omega}, \\[1ex]
            0, & \text{for } e \in \Gamma.
            \end{cases}\\[1.2ex]
        R_{J_K}^1 &=
            \begin{cases}
            \displaystyle 0,
            & \text{for } e \in \Gamma_D, \\[1ex]
            \sum\limits_{i=1}^{d-1} \Big(\mu \boldsymbol{n}^t \nabla \boldsymbol{u}_h \boldsymbol{\tau}^i + \beta \boldsymbol{u}_h\cdot\boldsymbol{\tau}^i \Big), & \text{for } e \in \Gamma_{\mathrm{Nav}}.
            \end{cases}\\[1.2ex]
        R_{J_K}^2 &=
            \begin{cases}
            \displaystyle 0,
            & \text{for } e \in \Gamma_D, \\[1ex]
              \boldsymbol{u}_h \cdot \boldsymbol{n} , & \text{for } e \in \Gamma_{\mathrm{Nav}}.
            \end{cases}
    \end{align}
\end{subequations}
Now we introduce the element-wise error estimator $\Psi_K^2 = \Psi_{R_K}^2 + \Psi_{e_K}^2 + \Psi_{J_K}^2$ with contributions defined as
\begin{subequations}\label{eqpsidf}
\begin{align}
\Psi_{R_K}^2
&= h_K^{2}\,\|\mathbf{R}_K\|_{0,K}^{2}\, +\, \|R_{1,K}\|_{0,K}^{2} \, + \,  h_K^{2}\,\|R_{2,K}\|_{0,K}^{2}, \\[1ex]
\Psi_{e_K}^2
&= \sum_{e\in\partial K}
h_e\bigl(\|\mathbf{R}_e\|_{0,e}^{2}
+ \|\mathbf{R}_{1,e}\|_{0,e}^{2}\bigl), \\
\Psi_{J_K}^2 &= \sum_{e \in \Gamma} \bigr( h_e \|  R_{J_K}^1 \|_{0,e}^2  + h_e^{-1} \|  R_{J_K}^2 \|^2_{0,e}
\bigr).
\end{align}
\end{subequations}
Let $\boldsymbol{f}_h$ and $g_h$ be piecewise polynomial representations of the data
$\boldsymbol{f}$ and $g$, respectively, which are allowed to be discontinuous
across element interfaces. We then define the associated data oscillation terms by
\begin{align}
\Theta_{1,K}^2 = h_K^2 \|\boldsymbol{f}-\boldsymbol{f}_h\|^2,
\qquad
\Theta_{2,K}^2 = h_K^2 \|g-g_h\|^2.
\end{align}
Consequently, a global \emph{a posteriori} error estimator and data oscillation error for the nonlinear coupled problem \(\eqref{ns2}\) are defined by
\begin{align}\label{eedo}
    \Psi = \left( \sum_{K\in\mathcal{T}_h} \Psi_K^{2} \right)^{1/2} \text{ and}\quad  \Theta = \left( \sum_{K\in\mathcal{T}_h} \Theta_{1,K}^{2} + \Theta_{2,K}^{2} \right)^{1/2}.
\end{align}
For any numerical approximation of \eqref{ns2}, the estimator $\Psi$ is both reliable and
efficient with respect to the energy norm $\|(\cdot\,,\cdot\,,\cdot)\|$ provided in \eqref{norm}.
The detailed proofs are presented in the subsequent sections. In what follows, we first concentrate on establishing the reliability of the estimator.
\subsection{Reliability}

\begin{theorem}[Global inf-sup stability]\label{thm4.1}
    Let $(\widetilde{\boldsymbol{u}}, \widetilde{\psi})  \in \boldsymbol{H}^1(\Omega) \times H^1(\Omega)$ satisfying $\|\widetilde{\boldsymbol{u}}\|_{1,h}\leq M$ and $\|\widetilde{\psi}\|_{1}\leq M$, for a sufficiently small $M>0$. For any $(\boldsymbol{u}, p, \psi) \in \boldsymbol{V} \times Q\times \Phi$, there exist an $(\boldsymbol{v}, q, \phi) \in \boldsymbol{V} \times Q\times \Phi$ with  $\|(\boldsymbol{v}, q, \phi)\|\lesssim 1$ such that 
    \begin{align*}
        \mathcal{A}_h^{(\widetilde{\boldsymbol{u}}, \widetilde{\psi})}(\boldsymbol{u},p,\psi;\boldsymbol{v},q,\phi) \gtrsim \|(\boldsymbol{u}, p, \psi)\|.
    \end{align*}
\end{theorem}
\begin{proof}
For any $(\boldsymbol{u},p,\psi)\in \boldsymbol{V} \times Q\times \Phi$, there holds
\begin{align*}
    \mathcal{A}_h^{(\widetilde{\boldsymbol{u}}, \widetilde{\psi})}(\boldsymbol{u}, p, \psi;\boldsymbol{u}, -p, \psi) &= A_h(\boldsymbol{u},\boldsymbol{u}) + d(\psi,\psi) +  c_1(\boldsymbol{u};\widetilde{\psi},\boldsymbol{u})+ c_2(\widetilde{\boldsymbol{u}};\psi,\psi),\\
    &= \mu (\nabla \boldsymbol{u}, \nabla \boldsymbol{u}) +\sum\limits_{e \in \mathcal{E}_{\mathrm{Nav}}} \sum\limits_{i=1}^{d-1} \beta \|\boldsymbol{u}\cdot\boldsymbol{\tau}^i\|^2_{0,e}\, + \,\varepsilon\|\nabla \psi\|^2+ (\boldsymbol{u}\cdot\nabla\widetilde{\psi}\boldsymbol{E},\boldsymbol{v})+ (\widetilde{\boldsymbol{u}}\cdot\nabla\psi,\psi),\\
    &\ge \mu \|\nabla\boldsymbol{u}\|^2 +  \sum\limits_{i=1}^{d-1} \beta \|\boldsymbol{u}\cdot\boldsymbol{\tau}^i\|^2_{0,\mathcal{E}_{\mathrm{Nav}}}\, + \,\varepsilon\|\nabla \psi\|^2 - M\bar{E}C_p^2\|\nabla \boldsymbol{u}\|^2- MC_p\|\nabla \psi\|^2,\\
    &\ge (\mu - M\bar{E}C_p^2) \|\nabla\boldsymbol{u}\|^2 +  \,(\varepsilon- MC_p)\|\nabla \psi\|^2.
\end{align*}
Applying the inf-sup condition, we get that for any $p\in Q$, there exists $\boldsymbol{v}\in \boldsymbol{V}$ such that
\[
B_h(\boldsymbol{v},p)\ge \beta_1 \|p\|^2,
\qquad
\|\boldsymbol{v}\|_{1}\le \|p\|,
\]
where $\beta_1>0$ is the inf-sup constant depending only on $\Omega$. Then, we have
\begin{align*}
  \mathcal{A}_h^{(\widetilde{\boldsymbol{u}}, \widetilde{\psi})}(\boldsymbol{u}, p, \psi;\boldsymbol{v}, 0, 0) 
&= A_h(\boldsymbol{u},\boldsymbol{v})
+ B_h(\boldsymbol{v},p) + c_1(\boldsymbol{u};\widetilde{\psi},\boldsymbol{v}),\\
&\ge \beta_1 \|p\|^2
- |A_h(\boldsymbol{u},\boldsymbol{v})| - |c_1(\boldsymbol{u};\widetilde{\psi},\boldsymbol{v})|, \\
&\ge \beta_1 \|p\|^2 - \mu \|\nabla\boldsymbol{u}\|\|\nabla\boldsymbol{v}\|
- \sum\limits_{i=1}^{d-1} \beta \int_{\mathcal{E}_{\mathrm{Nav}}}\abs{(\boldsymbol{u}\cdot\boldsymbol{\tau}^i)(\boldsymbol{v}\cdot\boldsymbol{\tau}^i)} \, \dd{s}- C_p^2M\bar{E}\|\nabla\boldsymbol{u}\|\|\nabla\boldsymbol{v}\|, \\
&\ge \beta_1 \|p\|^2 - \Big(\mu + \beta C_{trc} + C_p^2M\bar{E}\Big)\|\nabla\boldsymbol{u}\|\|\nabla\boldsymbol{v}\|, \\
&\ge \Big(\beta - \frac{1}{2\epsilon}\Big) \|p\|^2 - \frac{\epsilon}{2}\Big(\mu + \beta C_{trc} + C_p^2M\bar{E}\Big)^2\|\nabla\boldsymbol{u}\|^2,
\end{align*}
where $\epsilon>0$ will choose later according to our need. Now, we introduce $\delta>0$ such that
\begin{align*}
 \mathcal{A}_h^{(\widetilde{\boldsymbol{u}}, \widetilde{\psi})}(\boldsymbol{u},p,\psi;\boldsymbol{u}+\delta\boldsymbol{v},-p,\psi)
&=  \mathcal{A}_h^{(\widetilde{\boldsymbol{u}}, \widetilde{\psi})}(\boldsymbol{u}, p, \psi;\boldsymbol{u}, -p, \psi) + \delta  \mathcal{A}_h^{(\widetilde{\boldsymbol{u}}, \widetilde{\psi})}(\boldsymbol{u}, p, \psi;\boldsymbol{v}, 0, 0), \\
&\ge \bigg[\mu - M\bar{E}C_p^2 - \delta \Big(\frac{\epsilon}{2}\Big(\mu + \beta C_{trc} + C_p^2M\bar{E}\Big)^2\Big)\bigg]\|\nabla\boldsymbol{u}\|^2\\
&\qquad+ \delta\Big(\beta_1 - \frac{1}{2\epsilon}\Big) \|p\|^2 + ( \varepsilon - MC_p)\|\nabla\psi\|^2 .
\end{align*}
Choosing $\epsilon=\dfrac{1}{\beta_1}$ and for sufficiently small $\delta$ and $M$ , we obtain
\begin{align*}
   \mathcal{A}_h^{(\widetilde{\boldsymbol{u}}, \widetilde{\psi})}(\boldsymbol{u},p,\psi;\boldsymbol{u}+\delta\boldsymbol{v},-p,\psi)
& \gtrsim \Big( \|\nabla\boldsymbol{u}\|^2 + \|p\|^2 +\|\nabla\psi\|^2\Big).
\end{align*}
Finally, using the triangle inequality, we assert that
\begin{align*}
\|(\boldsymbol{u}+\delta\boldsymbol{v},-p,\psi)\|^2
&= \|\nabla\boldsymbol{u}+\delta\nabla\boldsymbol{v}\|^2
+ \|p\|^2
+ \|\nabla\psi\|^2,\\
&\le 2\bigl(\|\nabla\boldsymbol{u}\|^2+\delta^2\|\nabla\boldsymbol{v}\|^2\bigr)
+ \|p\|^2
+ \|\psi\|_{1}^2, \\
&\le \max\{2,1+2\delta^2\}
\bigl(\|\nabla\boldsymbol{u}\|^2+\|p\|^2+ \|\nabla\psi\|^2\bigr),\\
& \lesssim \| (\boldsymbol{u},p, \psi)\|^2.
\end{align*}
This concludes the proof.
\end{proof}

Since $\boldsymbol{V}_h$ is not a conforming space. Now, we define the conforming space 
$\boldsymbol{V}_h^c = \boldsymbol{V}_h \cap \boldsymbol{V}$. Finally, we decompose the approximate velocity uniquely into 
$\boldsymbol{u}_h = \boldsymbol{u}_h^c + \boldsymbol{u}_h^r$, where 
$\boldsymbol{u}_h^c \in V_h^c$ and $\boldsymbol{u}_h^r \in (V_h^c)^\perp$, and we note that 
$\boldsymbol{u}_h^r = \boldsymbol{u}_h - \boldsymbol{u}_h^c \in V_h$.
\begin{lemma}\label{lma5.2}
    There holds
\[
\|\boldsymbol{u}_h^r\|_{1,h} \le C_r \left( \sum_{K \in \mathcal{T}_h}  h_e^{-1} \|  R_{J_K}^2 \|^2_{0,e}\right)^{1/2}.
\]
\end{lemma}
\begin{proof}
    It follows straightforwardly from the decomposition 
$\boldsymbol{u}_h = \boldsymbol{u}_h^r + \boldsymbol{u}_h^c$ and \cite[Lemma 11]{bansal2026equal}.
\end{proof} 
\begin{theorem}
Let $(\boldsymbol{u},p,\psi)$ be the weak solution of \eqref{wf} and
$(\boldsymbol{u}_h,p_h,\psi_h)$ be its discrete approximation defined by \eqref{ns2}.
Let $\Psi$ denotes the error estimator term, as defined in \eqref{eedo}.
Assume  \ref{(A1)}--\ref{(A3)}, and that $\|\boldsymbol{u}\|_{\infty}\le M$ and $\|\psi\|_{1,\infty}< M$ for sufficiently small $M$, then there exists a constant $C>0$ such that the following estimate holds:
\begin{align}\label{thm2:eq}
    \|(e^{\boldsymbol{u}},e^p,e^\psi)\| &\le C\Big( \Psi + \| \boldsymbol{f} - \boldsymbol{f}_h \| + \,\| g - g_h \| \Big),
\end{align}
for all $(\boldsymbol{v},q,\phi)$ satisfying $\|(\boldsymbol{v},q,\phi)\|\le 1$.
\end{theorem}
\begin{proof}
Using $\boldsymbol{u}_h = \boldsymbol{u}_h^r + \boldsymbol{u}_h^c,\, e^{\boldsymbol{u}} = e^{\boldsymbol{u}}_c- \boldsymbol{u}^r_h $ 
and the triangle inequality implies
\begin{align*}
    \| (e^{\boldsymbol{u}},e^p,e^\psi)\| \le \| (e^{\boldsymbol{u}}_c,e^p,e^\psi)\|\, +\, \|\boldsymbol{u}^r_h\|_{1,h},
\end{align*}
where $(e^{\boldsymbol{u}}_c,e^p,e^\psi)\in \boldsymbol{V}\times Q \times \Phi$. Then from \Cref{thm4.1}. We have
\begin{align}
    C\| (e^{\boldsymbol{u}}_c,e^p,e^\psi)\| &\le \mathcal{A}_h^{(\boldsymbol{u}_h, \psi_h)}(e^{\boldsymbol{u}}_c,e^p,e^\psi;\boldsymbol{v},q,\phi),\nonumber\\ 
    &\le \mathcal{A}_h^{(\boldsymbol{u}_h, \psi_h)}(e^{\boldsymbol{u}},e^p,e^\psi;\boldsymbol{v},q,\phi) + \mathcal{A}_h^{(\boldsymbol{u}_h, \psi_h)}(\boldsymbol{u}^r_h,0,0;\boldsymbol{v},q,\phi),\nonumber\\ 
    &\le \mathcal{A}_h^{(\boldsymbol{u}_h, \psi_h)}(\boldsymbol{u},p,\psi;\boldsymbol{v},q,\phi) -\mathcal{A}_h^{(\boldsymbol{u}_h, \psi_h)}(\boldsymbol{u}_h,p_h,\psi_h;\boldsymbol{v},q,\phi) + C\|\boldsymbol{u}^r_h\|_{1,h}(\|\boldsymbol{v}\|_{1,h} + \|q\|),\label{eq9}
\end{align}
with $\|(\boldsymbol{v},q,\phi)\|\le 1$. Since we have a relation
\begin{align}
\mathcal{A}_h^{(\boldsymbol{u}_h, \psi_h)}(\boldsymbol{u},p,\psi;\boldsymbol{v},q,\phi)
= \mathcal{A}_h^{(\boldsymbol{u}, \psi)}(\boldsymbol{u},p,\psi;\boldsymbol{v},q,\phi) + c_1(\boldsymbol{u};\psi_h-\psi,\phi) + c_2(\boldsymbol{u}_h-\boldsymbol{u};\psi,\phi) .
\label{eq10}
\end{align}
Using \eqref{eq10} in \eqref{eq9}, we obtain
\begin{align}\label{eq11}
C\,\| (e^{\boldsymbol{u}},e^p,e^\psi)\| & 
\le
\mathcal{A}_h^{(\boldsymbol{u}, \psi)}(\boldsymbol{u},p,\psi;\boldsymbol{v},q,\phi)  + c_1(\boldsymbol{u};\psi_h-\psi,\phi) + c_2(\boldsymbol{u}_h-\boldsymbol{u};\psi,\phi) \nonumber \\ & \quad  -\mathcal{A}_h^{(\boldsymbol{u}_h, \psi_h)}(\boldsymbol{u}_h,p_h,\psi_h;\boldsymbol{v},q,\phi)  + C\|\boldsymbol{u}^r_h\|_{1,h}.
\end{align}
Next, using the inequality $\|(\boldsymbol{v},q,\phi)\|\le 1$, we estimate the following
\begin{equation}
\begin{aligned}
c_1(\boldsymbol{u};\psi_h-\psi,\phi) \le&\, \bar{E}\, \|\boldsymbol{u}\|_{\infty}\,\|\nabla e_{\psi}\|\,\|\boldsymbol{v}\| \le \bar{E} M \|\nabla e_{\psi}\|, \\
c_2(\boldsymbol{u}_h-\boldsymbol{u};\psi,\phi) \le& \,\|e_{\boldsymbol{u}}\|\,\|\nabla \psi\|_{\infty}\,\|\phi\| \le M C\|e_{\boldsymbol{u}}\|_{1,h}.\label{eq12}
\end{aligned}
\end{equation}
Therefore we have from \labelcref{eq11,eq12},
\begin{align*}
C\,\| (e^{\boldsymbol{u}},e^p,e^\psi) \|
&\le
\mathcal{A}_h^{(\boldsymbol{u}, \psi)}(\boldsymbol{u},p,\psi;\boldsymbol{v},q,\phi) -\mathcal{A}_h^{(\boldsymbol{u}_h, \psi_h)}(\boldsymbol{u}_h,p_h,\psi_h;\boldsymbol{v},q,\phi) + \bar{E} M\|\nabla e_{\psi}\| +  M C\|e_{\boldsymbol{u}}\|_{1,h} + C\|\boldsymbol{u}^r_h\|_{1,h},
\end{align*}
choosing $M>0$ sufficiently small, we arrive at
\begin{align}
\frac{C}{2}\,\| (e^{\boldsymbol{u}},e^p,e^\psi) \|
\le
\mathcal{A}_h^{(\boldsymbol{u}, \psi)}(\boldsymbol{u},p,\psi;\boldsymbol{v},q,\phi) -\mathcal{A}_h^{(\boldsymbol{u}_h, \psi_h)}(\boldsymbol{u}_h,p_h,\psi_h;\boldsymbol{v},q,\phi) + C\|\boldsymbol{u}^r_h\|_{1,h}.\label{eq13}
\end{align}
Since $\mathcal{A}_h^{(\boldsymbol{u}_h, \psi_h)}(\boldsymbol{u}_h,p_h,\psi_h, \boldsymbol{v}_h, 0, \phi_h) - F_h^{\psi_h}(\boldsymbol{v}_h) - G_h(\phi_h) = 0, \ \forall (\boldsymbol{v}_h, q_h, \phi_h) \in \boldsymbol{V}_h\times Q_h\times \Phi_h, $ then we have from \eqref{eq13}:
\begin{align*}
\frac{C}{2}\,\| (e^{\boldsymbol{u}},e^p,e^\psi)\|
&\le
F^{\psi}(\boldsymbol{v}) 
- F_h^{\psi_h}(\boldsymbol{v}_h) + G(\phi) - G_h(\phi_h) - \mathcal{B}^{\psi_h}(\boldsymbol{u}_h,p_h,\psi_h;\boldsymbol{v}-\boldsymbol{v}_h,q,\phi-\phi_h), \\
 &\le \int_\Omega \Big\{ \boldsymbol{f} + [g - \mathcal{K}({\psi})]\boldsymbol{E} - \boldsymbol{f}_h - [g_h - \mathcal{K}({\psi_h})]\boldsymbol{E} \Big\} \cdot \boldsymbol{v}\,\dd{x} +  \int_\Omega \Big\{ \boldsymbol{f}_h + [g_h - \mathcal{K}({\psi_h})]\boldsymbol{E} \Big\} \cdot (\boldsymbol{v}-\boldsymbol{v}_h) \,\dd{x}\\
&\qquad + \int_{\Omega} (g-g_h)\phi\,\dd{x}
+ \int_{\Omega} g_h(\phi-\phi_h) \,\dd{x} - \mathcal{B}^{\psi_h}(\boldsymbol{u}_h,p_h,\psi_h;\boldsymbol{v}-\boldsymbol{v}_h,q,\phi-\phi_h) + C\|\boldsymbol{u}^r_h\|_{1,h}.
\end{align*}
using integration by parts gives, we get
\begin{align}\label{eq14}
    \frac{C}{2}\,\| (e^{\boldsymbol{u}},e^p,e^\psi)\|
&\le  T_1 + T_2 + T_3 + T_4 + T_5 + C\|\boldsymbol{u}^r_h\|_{1,h},
\end{align}
where
\begin{align*}
    T_1 &= \sum_{K \in \mathcal{T}_h} \bigg[ \int_K \Big\{ \boldsymbol{f}_h + [g_h - \mathcal{K}({\psi_h})]\boldsymbol{E} +\mu \Delta \boldsymbol{u}_h - \nabla p_h - \boldsymbol{u}_h\cdot\nabla\psi_h E   \Big\} \cdot (\boldsymbol{v}-\boldsymbol{v}_h)\, \dd{x} \\
    &\qquad + \int_K \Big\{ \boldsymbol{f} - \boldsymbol{f}_h + (g -g_h)E - (\mathcal{K}({\psi})-\mathcal{K}({\psi_h})) \boldsymbol{E}  \Big\} \cdot \boldsymbol{v}\, \dd{x} \bigg], \\
    T_2 &= \sum_{K \in \mathcal{T}_h}  \int_{\partial K} \Big\{(p_h\boldsymbol{I}-\mu\nabla \boldsymbol{u}_h)\cdot \boldsymbol{n}_K\Big\} \cdot (\boldsymbol{v}-\boldsymbol{v}_h) \,\dd{s},\\
    T_3 &= \sum_{K \in \mathcal{T}_h}  \int_K\nabla\cdot\boldsymbol{u}_h q \,\dd{x},\\
     T_4 &= \sum_{K \in \mathcal{T}_h} \bigg[ \int_K \Big\{g_h - \mathcal{K}(\psi_h)-\boldsymbol{u}_h\cdot\nabla\psi_h + \varepsilon\Delta\boldsymbol{u}_h \Big\}(\phi-\phi_h) \,\dd{x} + \int_{K} (g-g_h)\phi\,\dd{x}\bigg], \\
      T_5 &= \sum_{K \in \mathcal{T}_h}  \int_{\partial K} \varepsilon\nabla\psi_h\cdot \boldsymbol{n}_K(\phi-\phi_h) \,\dd{s}.
\end{align*}
Applying the Cauchy-Schwarz inequality to $T_1$ implies
\begin{align*}
T_1 &\le \left(\sum_{K \in \mathcal{T}_h}h_K^2 \, \| \boldsymbol{R}_K \|_{0,K}^2 \right)^{1/2}
\left( \sum_{K \in \mathcal{T}_h} h_K^{-2} \, \| \boldsymbol{v} - \boldsymbol{v}_h \|_{0,K}^2 \right)^{1/2} + \Big(\| \boldsymbol{f} - \boldsymbol{f}_h \| + \bar{E}\| g - g_h \| +\bar{E}\|\mathcal{K}({\psi})-\mathcal{K}({\psi_h}) \|  \Big)  \| \boldsymbol{v} \|, \\
&\le
\left(\sum_{K \in \mathcal{T}_h}h_K^2 \, \| \boldsymbol{R}_K \|_{0,K}^2\right)^{1/2} C \, \| \nabla \boldsymbol{v} \|\, +\, \Big(\| \boldsymbol{f} - \boldsymbol{f}_h \| + \,\bar{E}\,\| g - g_h \| \Big)\| \boldsymbol{v} \|\, + \,\epsilon\,\bar{E}\,\bar{\mathcal{K}}\|{\psi}-\psi_h \|\,(\frac{1}{\epsilon} \|\boldsymbol{v} \|).
\end{align*}
Next, we write $T_2$ as a sum over interior facets and apply the Cauchy–Schwarz inequality to obtain
\begin{align*}
T_2
&= \sum_{e \in \mathcal{E}_h}
\int_e \frac{1}{2} \llbracket (p_h\boldsymbol{I}-\mu\nabla \boldsymbol{u}_h)\cdot \boldsymbol{n}_K \rrbracket \cdot (\boldsymbol{v}-\boldsymbol{v}_h) \, \dd{s}, \\
&\le\left(\sum_{e \in \mathcal{E}_h}h_e \, \| \boldsymbol{R}_e \|_{0,e}^2\right)^{1/2}\left(\sum_{e \in \mathcal{E}_h}h_e^{-1} \, \| (\boldsymbol{v}-\boldsymbol{v}_h) \|_{0,e}^2\right)^{1/2}, \\
&\le\left(\sum_{e \in \mathcal{E}_h}h_e \, \| \boldsymbol{R}_e \|_{0,e}^2\right)^{1/2}C\, \| \nabla \boldsymbol{v} \|.
\end{align*}
The bound for $T_3$ is defined as
\begin{align*}
    T_3 &= \left(\sum_{K \in \mathcal{T}_h} \, \|R_{1,K} \|_{0,K}^2\right)^{1/2}\|q\|.
\end{align*}
 Similarly, owing to the Cauchy–Schwarz inequality, we establish the following bounds for $T_4$ and $T_5$:
\begin{align*}
    T_4 &\le \left(\sum_{K \in \mathcal{T}_h}h_K^2 \, \| R_{2,K} \|_{0,K}^2 \right)^{1/2}
\left( \sum_{K \in \mathcal{T}_h} h_K^{-2} \, \| \phi - \phi_h \|_{0,K}^2 \right)^{1/2} + \| g - g_h \| \,\| \phi \|, \\
&\le \left(\sum_{K \in \mathcal{T}_h}h_K^2 \, \| R_{2,K} \|_{0,K}^2 \right)^{1/2} C \, \| \nabla \phi \|
+ \| g - g_h \| \,\| \phi \|.\\
T_5 &= \sum_{e \in \mathcal{E}_h}\int_e \frac{1}{2} \llbracket  \varepsilon\nabla\psi_h\cdot \boldsymbol{n}_K \rrbracket \cdot (\phi-\phi_h) \, \dd{s},\\
&\le\left(\sum_{e \in \mathcal{E}_h}h_e \, \| R_{1,e} \|_{0,e}^2\right)^{1/2}\left(\sum_{e \in \mathcal{E}_h}h_e^{-1} \, \|\phi-\phi_h\|_{0,e}^2\right)^{1/2}, \\
&\le\left(\sum_{e \in \mathcal{E}_h}h_e \, \| R_{1,e} \|_{0,e}^2\right)^{1/2}C\, \| \nabla \phi \|.
\end{align*}
Finally, for sufficiently small $\epsilon$, the estimate \eqref{thm2:eq} is obtained by inserting the bounds for $T_1, T_2, T_3, T_4$, and $T_5$ into \eqref{eq14} together with \Cref{lma5.2}.
\end{proof}

\subsection{Efficiency}
The efficiency of the estimator is analyzed using the classical technique 
of polynomial bubble functions; see \cite{verfurth2005robust}. 
To this end, we introduce an element bubble function $b_K$, 
which is supported on the element $K$. 
These bubble functions satisfy $b_K \in H_0^1(K), 
\|b_K\|_{\infty,K} 
= 1.$ For every element $K$, the following standard estimates hold; see Verfürth~\cite{verfurth2013posteriori}:
\begin{align}\label{eqbk}
\|b_K v\|_{0,K} 
\le C \|v\|_{0,K}, \qquad
\|v\|_{0,K} 
\le C \|b_K^{1/2} v\|_{0,K}, \qquad
\|\nabla (b_K v)\|_{0,K} 
\le C h_K^{-1} \|v\|_{0,K}.
\end{align}

\begin{lemma}\label{lma4.3}
Let $K$ be an element of $\mathcal{T}_h$ then the following estimates hold:
\begin{align*}
h_K \|\boldsymbol{R}_K\|_{0,K} 
&\le C \Big(\|\nabla( \boldsymbol{u} - \boldsymbol{u}_h)\|_{0,K} \,
+ \,\|p - p_h\|_{0,K} + \|\nabla(\psi - \psi_h)\|_{0,K} \,
+ \,h_K \|\boldsymbol{f} - \boldsymbol{f}_h\|_{0,K}\, + \,h_K \|g - g_h\|_{0,K}\Big), \\
 \|R_{1,K}\|_{0,K} 
&\le C\, \|\nabla( \boldsymbol{u} - \boldsymbol{u}_h)\|_{0,K} , \\
h_K \|R_{2,K}\|_{0,K} 
&\le C \Big(\|\nabla( \boldsymbol{u} - \boldsymbol{u}_h)\|_{0,K} \,
 + \,\|\nabla(\psi - \psi_h)\|_{0,K} \,
+ \,h_K \|g - g_h\|_{0,K}\Big),
\end{align*}
where $C$ is a positive constant. Moreover, it follows that
\[
\Psi_{R_K}^2 
\le 
C \Big(\|\nabla( \boldsymbol{u} - \boldsymbol{u}_h)\|_{0,K}^2 \,
+ \,\|p - p_h\|_{0,K}^2 + \|\nabla(\psi - \psi_h)\|_{0,K}^2 \,
+ \,h_K^2 \|\boldsymbol{f} - \boldsymbol{f}_h\|_{0,K}^2\, + \,h_K^2 \|g - g_h\|_{0,K}^2\Big).
\]
\end{lemma}
\begin{proof}
For each $K \in \mathcal{T}_h$, we define $\boldsymbol{W}_K = b_K \boldsymbol{R}_K$. 
Then, using second inequality from \eqref{eqbk}, we have
\begin{align*}
\frac{1}{C^2}\|\boldsymbol{R}_K\|_{0,K}^2
&\le \|b_K^{1/2} \boldsymbol{R}_K\|_{0,K}^2 = \int_K \boldsymbol{R}_K \cdot \boldsymbol{W}_K = \int_K \Big( \boldsymbol{f}_h + [g_h - \mathcal{K}(\psi_h)]\boldsymbol{E} +  \mu \Delta \boldsymbol{u}_h - \nabla p_h
- \boldsymbol{u}_h \cdot \nabla \psi_h \boldsymbol{E} \Big) \cdot \boldsymbol{W}_K .
\end{align*}

Noting that
$ \Big(\boldsymbol{f} + [g - \mathcal{K}(\psi)]\boldsymbol{E} +  \mu \Delta \boldsymbol{u} - \nabla p
- \boldsymbol{u} \cdot \nabla \psi \boldsymbol{E}\Big)\Big|_K = 0$
for the exact solution $(\boldsymbol{u},p, \psi)$, we subtract the above identity,
integrate by parts, and use that $\boldsymbol{W}_K|_{\partial K} = 0$ to obtain
\begin{align}
\frac{1}{C^2}\|\boldsymbol{R}_K\|_{0,K}^2
&\le \int_K \Big[ \mu \nabla (\boldsymbol{u} - \boldsymbol{u}_h)
- (p - p_h) \boldsymbol{I}\Big]  \nabla \cdot \boldsymbol{W}_K\, \dd{x}\, 
+ \int_K \big[ \boldsymbol{u} \cdot \nabla \psi - \boldsymbol{u}_h \cdot \nabla \psi_h) \boldsymbol{E} \big] \cdot \boldsymbol{W}_K\, \dd{x} \nonumber\\
& \qquad  + \int_K \Big[ \boldsymbol{f}_h + [g_h - \mathcal{K}(\psi_h)]\boldsymbol{E} -  \boldsymbol{f} - [g - \mathcal{K}(\psi)]\boldsymbol{E} \Big]  \boldsymbol{W}_K\, \dd{x} ,\nonumber \\[0.8ex]
&\le T_1 + T_2 + T_3,\label{eqrkt}
\end{align}
where
\begin{align*}
T_1 &= \int_K \Big[ \mu \nabla (\boldsymbol{u} - \boldsymbol{u}_h)
- (p - p_h)\boldsymbol{I} \Big]  \nabla \cdot \boldsymbol{W}_K\, \dd{x}, \\
T_2 &= \int_K \Big[ \boldsymbol{u} \cdot \nabla \psi - \boldsymbol{u}_h \cdot \nabla \psi_h \boldsymbol{E} \Big] \cdot \boldsymbol{W}_K\, \dd{x}, \\
T_3 &= \int_K \Big[ (\boldsymbol{f}_h -  \boldsymbol{f}) + (g_h -g)\boldsymbol{E}  - [\mathcal{K}(\psi_h)- \mathcal{K}(\psi)]\boldsymbol{E}  \Big]  \boldsymbol{W}_K\, \dd{x}  .
\end{align*}
Now,
\begin{align*}
T_1 
&= \int_K \big[ \mu (\boldsymbol{u}-\boldsymbol{u}_h)
      - (p-p_h) \boldsymbol{I} \big]  \nabla \cdot \boldsymbol{W}_K \, \dd{x}  \,
      \le C \Big( \|\nabla\boldsymbol{u}-\nabla\boldsymbol{u}_h\|_{0,K}
      + \|p-p_h\|_{0,K} \Big)
      \|\nabla \cdot \boldsymbol{W}_K\|_{0,K}.
\end{align*}
and,
\begin{align*}
T_2 
&= \int_K (\boldsymbol{u}\cdot\nabla \psi - \boldsymbol{u}_h \cdot \nabla \psi_h)
    \boldsymbol{E}\cdot \boldsymbol{W}_K  \, dx, \\
    &\le \bar{E}\int_K \Big[(\boldsymbol{u}-\boldsymbol{u}_h)\nabla \psi
     + \boldsymbol{u}_h(\nabla \psi - \nabla \psi_h)\Big]
     \cdot \boldsymbol{W}_K \, dx, \\[0.5ex]
&\le M\bar{E}C_{\mathrm{Sob}}^2
   \Big( \|\boldsymbol{u}-\boldsymbol{u}_h\|_{1,K}
        + \|\nabla \psi - \nabla \psi_h\|_{0,K} \Big)
   \|\boldsymbol{W}_K\|_{1,K} , \\[0.5ex]
   &\le \,C
   \left( \|\nabla\boldsymbol{u}-\nabla\boldsymbol{u}_h\|_{1,K}
        + \|\nabla \psi - \nabla \psi_h\|_{0,K} \right)
   h_K^{-1} \|R_K\|_{0,K}.
\end{align*}
and,
\begin{align*}
T_3 
&= \int_K \Big[
   \boldsymbol{f}_h - \boldsymbol{f}
   + (\boldsymbol{g}_h-\boldsymbol{g})\boldsymbol{E}
   - (K(\psi_h)-K(\psi))
   \Big] \cdot \boldsymbol{W}_K \, dx, \\
&\le \left(
   \|\boldsymbol{f}-\boldsymbol{f}_h\|_{0,K}
   + \|\boldsymbol{g}-\boldsymbol{g}_h\|_{0,K}
   + \|\psi-\psi_h\|_{0,K}
   \right)
   \|\boldsymbol{W}_K\|_{0,K}, \\[0.5ex]
&\le C h_K
   \left(
   \|\boldsymbol{f}-\boldsymbol{f}_h\|_{0,K}
   + \|\boldsymbol{g}-\boldsymbol{g}_h\|_{0,K}
   + \|\psi-\psi_h\|_{0,K}
   \right) h_K^{-1} \|R_K\|_{0,K}.
\end{align*}
Using $T_1$, $T_2$, and $T_3$ in \eqref{eqrkt},
\begin{align*}
C^{-2} \|\boldsymbol{R}_K\|_{0,K}^2
&\le C \Big[
   \|\nabla(\boldsymbol{u}- \boldsymbol{u}_h)\|_{0,K}
   + \|p-p_h\|_{0,K}
   + \|\nabla\psi-\nabla\psi_h\|_{0,K}  \\
&\qquad + h_K \|\boldsymbol{f}-\boldsymbol{f}_h\|_{0,K}
   + h_K \|\boldsymbol{g}-\boldsymbol{g}_h\|_{0,K}
   + h_K \|\psi-\psi_h\|_{0,K}
   \Big]
   h_K^{-1} \|R_K\|_{0,K}.
   \intertext{Therefore,}
   h_K \|\boldsymbol{R}_K\|_{0,K}
&\le C  \Big(\|\nabla( \boldsymbol{u} - \boldsymbol{u}_h)\|_{0,K} \,
+ \,\|p - p_h\|_{0,K} + \|\nabla(\psi - \psi_h)\|_{0,K} \,
+ \,h_K \|\boldsymbol{f} - \boldsymbol{f}_h\|_{0,K}\, + \,h_K \|g - g_h\|_{0,K}\Big).
\end{align*}
The remaining two bounds can be established by following the similar arguments. Furthermore, the final estimate follows directly from \eqref{eqpsidf} together with the above bounds.
\end{proof}

Let $e$ be an interior edge shared by two neighboring elements $K$ and $K'$, 
and define the associated patch $\omega_e = K \cup K'$. 
An edge bubble function $b_e$ is defined on $e$ such that it is strictly positive 
in the interior of the patch and vanishes on its boundary. 
These bubble functions satisfy 
$b_e \in H_0^1(\omega_e),
\|b_e\|_{\infty,\omega_e} = 1.$
For every edge $e$, the following inequalities hold \cite{verfurth2013posteriori}:
\begin{align}\label{eqbe}
\|v\|_{0,e} 
\le C \|b_e^{1/2} v\|_{0,e},
\qquad
\|b_e v\|_{0,K} 
\le C h_e^{1/2} \|v\|_{0,e}, 
\qquad 
\|\nabla (b_e v)\|_{0,K} 
\le C h_e^{-1/2} \|v\|_{0,e}, 
\qquad \forall K \in \omega_e.
\end{align}

\begin{lemma}\label{lma4.4}
    For any element $K\in\mathcal{T}_h$, the jump residual satisfies
    \begin{align*}
        h_e\|\boldsymbol{R}_e\|^2_{0,e} 
        &\le C \sum\limits_{K\in \omega_e} \Big(\|\nabla( \boldsymbol{u} - \boldsymbol{u}_h)\|_{0,K}^2 \,
+ \,\|p - p_h\|_{0,K}^2 + \|\nabla(\psi - \psi_h)\|_{0,K}^2 \,
+ \,h_K^2 \|\boldsymbol{f} - \boldsymbol{f}_h\|_{0,K}^2\, + \,h_K^2 \|g - g_h\|_{0,K}^2\Big),\\
        h_e\|R_{1,e}\|^2_{0,e} 
        &\le C \sum\limits_{K\in \omega_e} \Big(\|\nabla( \boldsymbol{u} - \boldsymbol{u}_h)\|_{0,K}^2 \,
+ \,\|p - p_h\|_{0,K}^2 + \|\nabla(\psi - \psi_h)\|_{0,K}^2 \,
+ \,h_K^2 \|g - g_h\|_{0,K}^2\Big).
    \intertext{Moreover,}
     \Psi_{e_K}^2 &\le C \sum\limits_{K\in \omega_e} \Big(\|\nabla( \boldsymbol{u} - \boldsymbol{u}_h)\|_{0,K}^2 \,
+ \,\|p - p_h\|_{0,K}^2 + \|\nabla(\psi - \psi_h)\|_{0,K}^2 \,
+ \,h_K^2 \|\boldsymbol{f} - \boldsymbol{f}_h\|_{0,K}^2\, + \,h_K^2 \|g - g_h\|_{0,K}^2\Big).
    \end{align*}
\end{lemma}
\begin{proof}
Let $e$ be an interior facet and let us define a facet bubble function in the form
$\boldsymbol{W}_e = \sum_{e \subset \partial K} \frac{h_e}{2} \boldsymbol{R}_{e}\, b_e .$
Then, by the second property in \eqref{eqbe}, and recalling that the classical solution $(\boldsymbol{u},p,\psi)$ satisfies $\llbracket \bigl( p \boldsymbol{I} - \mu\nabla \boldsymbol{u}\bigr)\boldsymbol{n} \rrbracket|_e=0$, we obtain
\begin{align*}
h_e \|\boldsymbol{R}_{e}\|_{0,e}^2
&\le C \big(  \llbracket \bigl( p_h \boldsymbol{I} - \mu\nabla \boldsymbol{u}_h\bigr)\boldsymbol{n} \rrbracket , \boldsymbol{W}_e \big)_e, \\
&= C \big( \llbracket \bigl( p_h \boldsymbol{I} - \mu\nabla \boldsymbol{u}_h\bigr)\boldsymbol{n} \rrbracket -\llbracket \bigl( p \boldsymbol{I} - \mu\nabla \boldsymbol{u}\bigr)\boldsymbol{n} \rrbracket , \boldsymbol{W}_e \big)_e,
\end{align*}
integrating by parts over the patch $\omega_e$, we obtain
\begin{align*}
\big( \llbracket \bigl( p_h \boldsymbol{I} - \mu\nabla \boldsymbol{u}_h\bigr)\boldsymbol{n} \rrbracket -\llbracket \bigl( p \boldsymbol{I} - \mu\nabla \boldsymbol{u}\bigr)\boldsymbol{n} \rrbracket , \boldsymbol{W}_e \big)_e
&= \sum_{K \in \omega_e}
\Bigg[
\int_K \big(-\mu \Delta (u-u_h) + \nabla (p-p_h)\big)\, \boldsymbol{W}_e\, \dd{x} 
\\ & \quad + \int_K \big(-\mu \nabla (u-u_h) + (p-p_h)\boldsymbol{I}\big): \nabla \boldsymbol{W}_e\, \dd{x}
\Bigg].
\end{align*}
Since the exact solution $(\boldsymbol{u},p, \psi)$ satisfies
$
(-\mu \Delta \boldsymbol{u} + \nabla p)\big|_K
= (\boldsymbol{f} + (g - k(\psi) - \boldsymbol{u} \cdot \nabla \psi)\boldsymbol{E})\big|_K,
$
we have
\begin{align}
h_e \|\boldsymbol{R}_e\|_{0,e}^2
&= \sum_{K \subset \omega_e}
\Bigg[
\int_K \Big(\boldsymbol{f}_h + (g_h - k(\psi_h) - \boldsymbol{u}_h \cdot \nabla \psi_h)\boldsymbol{E} + \mu \Delta \boldsymbol{u}_h - \nabla p_h\Big)\, \boldsymbol{W}_e\, \dd{x} \nonumber\\
&\qquad
+ \int_K \Big( \boldsymbol{f} + (g - k(\psi))\boldsymbol{E} - \boldsymbol{f}_h - (g_h - k(\psi_h) )\boldsymbol{E}\Big)\cdot\boldsymbol{W}_e\, \dd{x}\nonumber
\\
&\qquad
+ \int_K \Big( \boldsymbol{u}_h \cdot \nabla \psi_h - \boldsymbol{u} \cdot \nabla \psi\Big)\boldsymbol{E}\cdot\boldsymbol{W}_e\, \dd{x} \nonumber\\
&\qquad
+ \int_K \Big(-\mu \nabla (\boldsymbol{u}-\boldsymbol{u}_h) + (p-p_h)\boldsymbol{I}\Big)
 : \nabla \boldsymbol{W}_e\, \dd{x}
\Bigg], \nonumber\\
&= T_1 + T_2 + T_3 + T_4 .\label{eqret}
\end{align}
Now using the definition of $\boldsymbol{R}_K$, and then combining the Cauchy- Schwarz inequality with the second inequality of \eqref{eqbe}, we obtain
\begin{align*}
    T_1 &\le 
     \Big(\sum\limits_{K\in \omega_e} h_K^2\|\boldsymbol{R}_K\|_{0,K}^2\Big)^{1/2}\Big(\sum\limits_{K\in \omega_e} h_K^{-2}\|\boldsymbol{W}_e\|_{0,K}^2\Big)^{1/2}, \\
     &\le 
     \Big(\sum\limits_{K\in \omega_e} h_K^2\|\boldsymbol{R}_K\|_{0,K}^2\Big)^{1/2} h_e^{1/2}\|\boldsymbol{R}_e\|_{0,e},
\end{align*}
and,
\begin{align*}
    T_2 &\le \bigg[ 
    \Big(\sum\limits_{K\in \omega_e} h_K^2\|\boldsymbol{f}-\boldsymbol{f}_h\|_{0,K}^2\Big)^{1/2} \!
    +\Big(\sum\limits_{K\in \omega_e} h_K^2\|g-g_h\|_{0,K}^2\Big)^{1/2}\! +\Big(\sum\limits_{K\in \omega_e} h_K^2\|\psi-\psi_h\|_{0,K}^2\Big)^{1/2}  \bigg]\Big(\sum\limits_{K\in \omega_e} h_K^{-2}\|\boldsymbol{W}_e\|_{0,K}^2\Big)^{1/2}, \\
    &\le \bigg[ 
    \Big(\sum\limits_{K\in \omega_e} h_K^2\|\boldsymbol{f}-\boldsymbol{f}_h\|_{0,K}^2\Big)^{1/2} \!
    +\Big(\sum\limits_{K\in \omega_e} h_K^2\|g-g_h\|_{0,K}^2\Big)^{1/2} \! +\Big(\sum\limits_{K\in \omega_e} h_K^2\|\psi-\psi_h\|_{0,K}^2\Big)^{1/2}  \bigg]h_e^{1/2}\|\boldsymbol{R}_e\|_{0,e},
\end{align*}
and using the third inequality of \eqref{eqbe}, we have
\begin{align*}
    T_3 &\le 
    \bar{E}\sum\limits_{K\in \omega_e}\int_K \Big(- \boldsymbol{u}_h \cdot (\nabla(\psi- \psi_h)) - (\boldsymbol{u} - \boldsymbol{u}_h) \cdot \nabla \psi\Big)\cdot\boldsymbol{W}_e\, \dd{x}, \\
    &\le C\bar{E}\bigg[\Big(\sum\limits_{K\in \omega_e} \|\nabla(\psi- \psi_h)\|_{0,K}^2\Big)^{1/2} + \Big(\sum\limits_{K\in \omega_e} \|\nabla( \boldsymbol{u} - \boldsymbol{u}_h)\|_{0,K}^2\Big)^{1/2} \bigg]\Big(\sum\limits_{K\in \omega_e}\|\boldsymbol{W}_e\|_{1,K}^2\Big)^{1/2},\\
    &\le C \bigg[\Big(\sum\limits_{K\in \omega_e} \|\nabla(\psi- \psi_h)\|_{0,K}^2\Big)^{1/2} + \Big(\sum\limits_{K\in \omega_e} \|\nabla( \boldsymbol{u} - \boldsymbol{u}_h)\|_{0,K}^2\Big)^{1/2}\bigg] h_e^{1/2}\|\boldsymbol{R}_e\|_{0,e},  \\
     T_4 &\le C \Big(\sum\limits_{K\in \omega_e} \|\nabla(\boldsymbol{u}- \boldsymbol{u}_h)\|_{0,K}^2 \,
    +\, \|p-p_h\|_{0,K}^2\Big)^{1/2} \Big(\sum\limits_{K\in \omega_e} \|\nabla\boldsymbol{W}_e\|_{0,e}^2\Big)^{1/2}, \\
    &\le C \Big(\sum\limits_{K\in \omega_e} \|\nabla(\boldsymbol{u}- \boldsymbol{u}_h)\|_{0,K}^2\, 
    +\, \|p-p_h\|_{0,K}^2\Big)^{1/2}h_e^{1/2}\|\boldsymbol{R}_e\|_{0,e}.
\end{align*}
Combining the estimates for $T_1$, $T_2$, $T_3$, and $T_4$ with \eqref{eqret}, we obtain the desired result. Similarly, we can prove the other two bounds. Furthermore, the final estimate follows directly from \eqref{eqpsidf} together with the above bounds.
\end{proof}
This lemma establishes the efficiency bound associated with the trace residual $\Psi_{J_K}$.
\begin{lemma}\label{lma4.5}
    For any element $K\in \mathcal{T}_h$, the trace residual satisfies the following estimate:
    \begin{align*}
         \Psi_{J_K}^2 \le C \|\nabla(\boldsymbol{u}-\boldsymbol{u}_h)\|_{0,K}^2.
    \end{align*}
\end{lemma}
\begin{proof}
Since, for exact solution $(\boldsymbol{u}, p, \psi)$, we have 
\[
\boldsymbol{u}\cdot\boldsymbol{n}|_e = 0, \qquad \text{and} \qquad \mu \boldsymbol{n}^t \nabla \boldsymbol{u}\tau^i+\beta \boldsymbol{u}\tau^i|_e=0, \qquad \forall e \in \Gamma_{\mathrm{Nav}}.
\]
Then, it follows from the trace inequality that
    \begin{align*}
        \Psi_{J_K}^2 &= \sum_{e \in \Gamma_{\mathrm{Nav}}} \Big( h_e \|  R_{J_K}^1 \|_{0,e}^2  + h_e^{-1} \|  R_{J_K}^2 \|^2_{0,e}\Big), \\
        &\le \sum_{e \in \Gamma_{\mathrm{Nav}}} \bigg( h_e \Big\| \sum_{i=1}^{d-1} \mu \boldsymbol{n}^t \nabla \boldsymbol{u}_h\tau^i+\beta \boldsymbol{u}_h\tau^i\Big\|_{0,e}^2 + h_e^{-1} \|\boldsymbol{u}_h\cdot\boldsymbol{n}\|_{0,e}^2\bigg),\\
        &\le \sum_{e \in \Gamma_{\mathrm{Nav}}} \bigg( h_e \Big\| \sum_{i=1}^{d-1} \mu \boldsymbol{n}^t \nabla (\boldsymbol{u}_h-\boldsymbol{u}) \tau^i+\beta (\boldsymbol{u}_h-\boldsymbol{u}) \tau^i\Big\|_{0,e}^2 + h_e^{-1} \|(\boldsymbol{u}_h-\boldsymbol{u})\cdot\boldsymbol{n}\|_{0,e}^2\bigg),\\[0.7ex]
        &\le C \|\nabla (\boldsymbol{u}-\boldsymbol{u}_h)\|_{0,K}^2.
    \end{align*}
\end{proof}
\begin{theorem}
 Suppose $(\boldsymbol{u},p,\psi) \in \boldsymbol{V}\times Q \times \Phi$ and $(\boldsymbol{u}_h,p_h,\psi_h) \in \boldsymbol{V}_h\times Q_h \times \Phi_h$ be the solution of the problem \eqref{wf} and \eqref{ns2} respectively. Then there exist a constant $C$, independent of $h$, such that
 \begin{align*}
     \Psi \le C\Big(\|(\boldsymbol{u}-\boldsymbol{u}_h, p-p_h, \psi-\psi_h)\|\, +\, \Theta\Big),
 \end{align*}
 where $\Psi$ and $\Theta$ is given by \eqref{eedo}.
\end{theorem}
\begin{proof}
    Combine the \Cref{lma4.3,lma4.4,lma4.5} implies the required results.
\end{proof}

\section{Numerical experiments}\label{sec:ne}
In this section, we present numerical experiments to verify the theoretical convergence rates and to demonstrate the performance of the proposed methods. All computations are carried out using the open-source finite element library FEniCS \cite{alnaes2015fenics}, together with the MUMPS distributed direct solver \cite{amestoy2000mumps} and a Newton solver. In some test cases, uniform meshes are employed, whereas in others, adaptive mesh refinement is utilized. In particular, starting from an initial mesh $\mathcal{T}_{0,h}$, we perform the standard adaptive refinement cycle
\[
\text{Solve} \to \text{Estimate} \to \text{Mark} \to \text{Refine}
\]
to generate a sequence of nested, shape-regular meshes \{ $\mathcal{T}_l$\} with corresponding mesh sizes $h_l$. At each iteration, local error indicators $\Psi_K$ are computed for all elements $K$ in the current mesh $\mathcal{T}_h$, and those elements $K \in \mathcal{T}_h$ satisfying
\[
\Psi_K \geq \Tilde{\theta} \max \{\Psi_K \colon K \in \mathcal{T}_h\}
\]
are selected for refinement, where $\Tilde{\theta} \in (0,1)$ is a user-defined parameter. We assess the quality of the a posteriori error estimator through the so-called effectivity
index, which is required to remain bounded as $h$ approaches zero and is defined by 
\begin{align*}
\text{Effectivity}
=
\frac{\Psi}
{\|(\boldsymbol{u}-\boldsymbol{u}_h,\, p-p_h,\, \psi-\psi_h)\|}.
\end{align*}

\begin{example}[\textbf{Convergence test}]\label{ex1}
    {\normalfont We consider the domain $\Omega = (0,1)^2$ and two viscosity values 
$\mu \in \{0.01, 1\}$. The exact solution is given by
\begin{align*}
\boldsymbol{u}(x,y) = \boldsymbol{curl}\!\left( x_1^{2}(1-x_1)^{2}\, x_2^{2} \right),
\qquad
p(x,y) = \sin(\pi x)\sin(\pi y), \qquad
\psi(x,y) = x(1-x)\,y(1-y).
\end{align*}
The source terms $\boldsymbol{f}$ and $g$ are computed using the above exact solution.
We impose the Navier--slip boundary condition on the boundaries $x=1$ and $y=1$, 
while essential boundary conditions are enforced on the remaining parts of $\Gamma$. The numerical results are computed using the 
$\mathbb{P}_2^2$--$\mathbb{P}_1$--$\mathbb{P}_2$ finite element pair. 
The parameters are chosen as 
$\boldsymbol{E} = (1,-1)^{t}$, $k_0 = 1$, $k_1 = 1$, 
$\gamma = 10$, $\varepsilon = 1.0$, and $\beta = 1$.

\Cref{tab1} presents the \textit{a priori} errors 
$\|\boldsymbol{u}-\boldsymbol{u}_h\|_{1,h}$, 
$\|p-p_h\|$, 
$\|\psi-\psi_h\|_{1}$, 
$\|(\boldsymbol{u}-\boldsymbol{u}_h,\, p-p_h,\, \psi-\psi_h)\|$, 
and $\Psi$. 
The corresponding convergence rates are in good agreement with the theoretical results.

The numerical results demonstrate that the proposed method maintains accuracy even for small viscosity values. The last column in \Cref{tab1} shows that the effectivity index remains bounded as $h \to 0$ for different values of $\mu.$
   } \end{example}

\begin{table*}[!t]
\caption{Convergence results for velocity, pressure, potential, total error and estimator for $\beta =1$.}
\label{tab1}
\centering
\footnotesize
\setlength{\tabcolsep}{4pt}
\renewcommand{\arraystretch}{1.2}
\begin{tabular}{cccccccccccccc}
\toprule
\multirow{2}{*}{$h$} &
\multicolumn{2}{c}{ $\|\boldsymbol{u}-\boldsymbol{u}_h\|_{1,h}$} &
\multicolumn{2}{c}{$\|p-p_h\|$} &
\multicolumn{2}{c}{$\|\psi-\psi_h\|_1$} &
\multicolumn{2}{c}{Total Error} &
\multicolumn{2}{c}{Estimator $\Psi$} &
\multirow{2}{*}{Effectivity} \\
\cmidrule(lr){2-3}\cmidrule(lr){4-5}\cmidrule(lr){6-7}
\cmidrule(lr){8-9}\cmidrule(lr){10-11}
& Error & OC
& Error & OC
& Error & OC
& Error & OC
& Error & OC \\
\midrule
$\mu$ = 0.01 & & & & & & & & & & & \\
\midrule
0.3536
& 1.3858E+00 & --
& 2.7786E-02 & --
& 8.2991E-03 & --
& 1.3944E+00 & --
& 1.3569E+00 & --
& 0.9731 \\

0.1768
& 2.0355E-01 & 2.7673
& 6.5996E-03 & 2.0739
& 2.1109E-03 & 1.9751
& 2.0570E-01 & 2.7611
& 2.0493E-01 & 2.7271
& 0.9963 \\

0.0884
& 3.0291E-02 & 2.7484
& 1.6183E-03 & 2.0279
& 5.3057E-04 & 1.9923
& 3.0826E-02 & 2.7383
& 3.2300E-02 & 2.6656
& 1.0478 \\

0.0442
& 4.8430E-03 & 2.6449
& 4.0237E-04 & 2.0079
& 1.3283E-04 & 1.9980
& 4.9763E-03 & 2.6310
& 5.7974E-03 & 2.4780
& 1.1650 \\

0.0221
& 8.2329E-04 & 2.5564
& 1.0045E-04 & 2.0021
& 3.3219E-05 & 1.9995
& 8.5657E-04 & 2.5384
& 1.1804E-03 & 2.2962
& 1.3780 \\

0.0110
& 1.4651E-04 & 2.4904
& 2.5103E-05 & 2.0005
& 8.3056E-06 & 1.9999
& 1.5483E-04 & 2.4679
& 2.6279E-04 & 2.1673
& 1.6973 \\

0.0055
& 2.7217E-05 & 2.4284
& 6.2752E-06 & 2.0001
& 2.0765E-06 & 1.9999
& 2.9296E-05 & 2.4019
& 6.1747E-05 & 2.0895
& 2.1077 \\
\midrule
$\mu$ = 1
&  & 
&  & 
&  & 
&  & 
&  & 
&  \\
\midrule
0.3536
& 6.0187E-02 & --
& 2.8208E-02 & --
& 8.2771E-03 & --
& 6.8722E-02 & --
& 3.8372E-01 & --
& 5.5836 \\

0.1768
& 1.4636E-02 & 2.0399
& 6.6411E-03 & 2.0866
& 2.1109E-03 & 1.9713
& 1.6779E-02 & 2.0341
& 9.6020E-02 & 1.9986
& 5.7226 \\

0.0884
& 3.5352E-03 & 2.0497
& 1.6213E-03 & 2.0343
& 5.3057E-04 & 1.9922
& 4.0698E-03 & 2.0437
& 2.3385E-02 & 2.0377
& 5.7460 \\

0.0442
& 8.6322E-04 & 2.0340
& 4.0256E-04 & 2.0098
& 1.3283E-04 & 1.9980
& 9.9654E-04 & 2.0299
& 5.7226E-03 & 2.0308
& 5.7425 \\

0.0221
& 2.1288E-04 & 2.0197
& 1.0046E-04 & 2.0026
& 3.3219E-05 & 1.9995
& 2.4616E-04 & 2.0173
& 1.4121E-03 & 2.0189
& 5.7364 \\

0.0110
& 5.2827E-05 & 2.0107
& 2.5104E-05 & 2.0007
& 8.3056E-06 & 1.9999
& 6.1141E-05 & 2.0094
& 3.5048E-04 & 2.0104
& 5.7324 \\

0.0055
& 1.3155E-05 & 2.0057
& 6.2752E-06 & 2.0002
& 2.0765E-06 & 1.9999
& 1.5234E-05 & 2.0048
& 8.7291E-05 & 2.0054
& 5.7299 \\
\bottomrule
\end{tabular}
\end{table*}
\medskip
\begin{example}[\textbf{Non-convex domains}]\label{ex2}
    {\normalfont 
   We consider the non-convex C-shaped, L-shaped, and T-shaped domains 
$\Omega_C = ((-1,1)\times(-1,1))
\setminus
((-0.2,1)\times(-0.5,0.5)),\,
\Omega_L = (-1,1)^2 \setminus (0,1)^2,$ 
and 
$\Omega_T = \bigl((-1.5,1.5)\times(0,1)\bigr) \cup \bigl((-0.5,0.5)\times(-2,0)\bigr),$ 
respectively. 
We take the source functions $\boldsymbol{f} = (1,1)^T$ and $g = 1$. 
The parameters are chosen as $\boldsymbol{E} = (0, -1)^{t}, k_0 = 1, k_1 = 1, \beta = 1, \gamma = 25, \mu = 1,$ and $\varepsilon = 1$. 

Since the exact solution is unknown for these problems, we focus on the behavior of the numerical solution. The presence of re-entrant corners induces singularities in the solution. As a result, uniform mesh refinement yields suboptimal convergence rates and does not adequately capture the local behavior near the corners. In contrast, adaptive mesh refinement performs significantly better.
As observed in \Cref{fig1}, the adaptive method selectively refines the mesh near the re-entrant corners. Such localized refinement effectively captures the singular behavior and improves the accuracy of the numerical solution.

Furthermore, \Cref{fig3,fig4,fig5} present detailed plots of the numerical solutions. 
These figures clearly show that adaptive refinement produces a finer mesh near the re-entrant corners, leading to improved accuracy and a better representation of the solution.

\begin{figure}[!htbp] 
    \centering
    \begin{subfigure}[b]{0.3\textwidth}
        \centering
        \includegraphics[width=\textwidth]{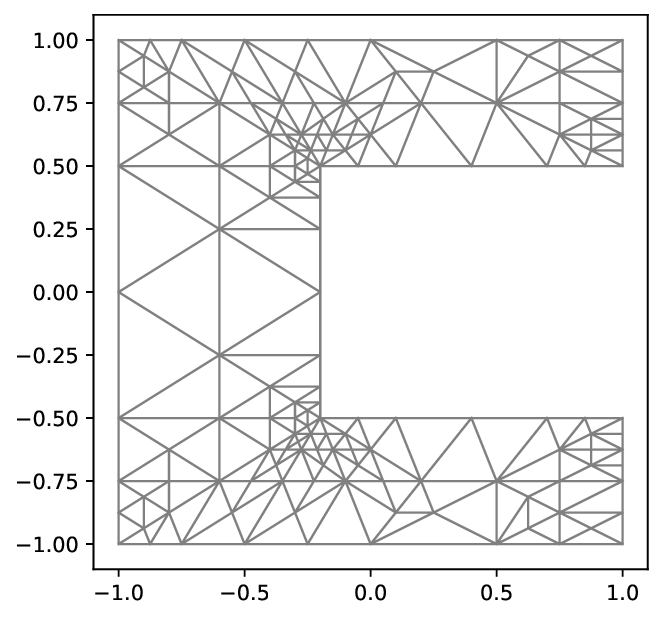}
         \caption{Dofs - $1205$}
    \end{subfigure}
    \hfill
    \begin{subfigure}[b]{0.3\textwidth}
        \centering
        \includegraphics[width=\textwidth]{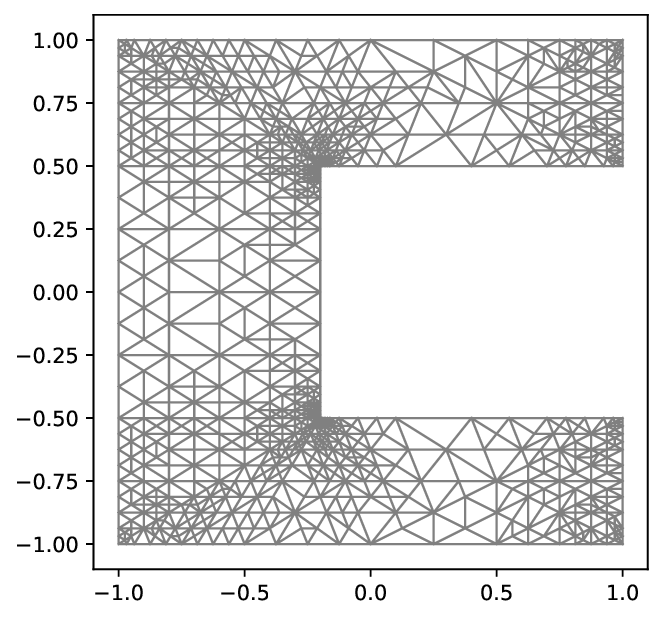}
        \caption{Dofs - $9766$}
    \end{subfigure}
    \hfill
    \begin{subfigure}[b]{0.3\textwidth}
        \centering
        \includegraphics[width=\textwidth]{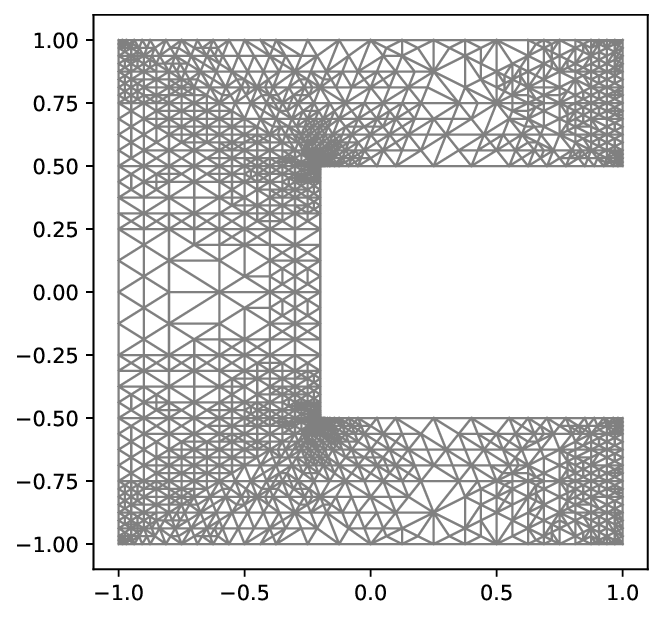}
        \caption{Dofs - $19297$}
    \end{subfigure}
     \hfill
    \begin{subfigure}[b]{0.3\textwidth}
        \centering
        \includegraphics[width=\textwidth]{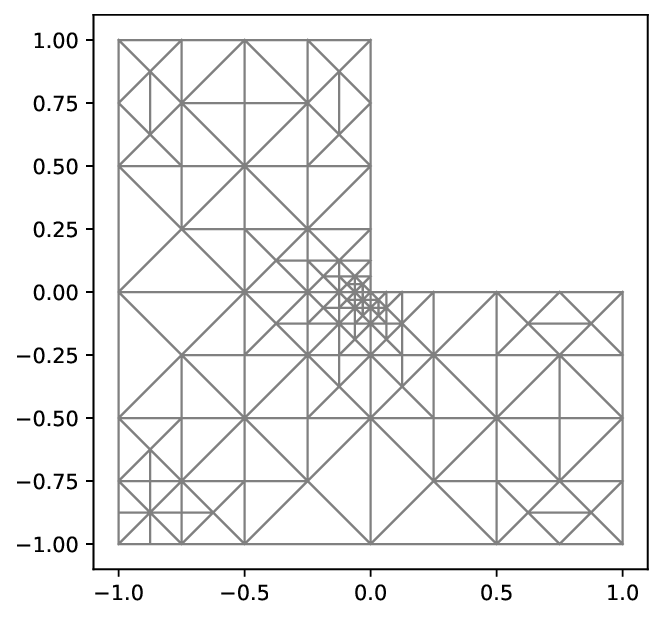}
         \caption{Dofs - $1137$}
    \end{subfigure}
    \hfill
    \begin{subfigure}[b]{0.3\textwidth}
        \centering
        \includegraphics[width=\textwidth]{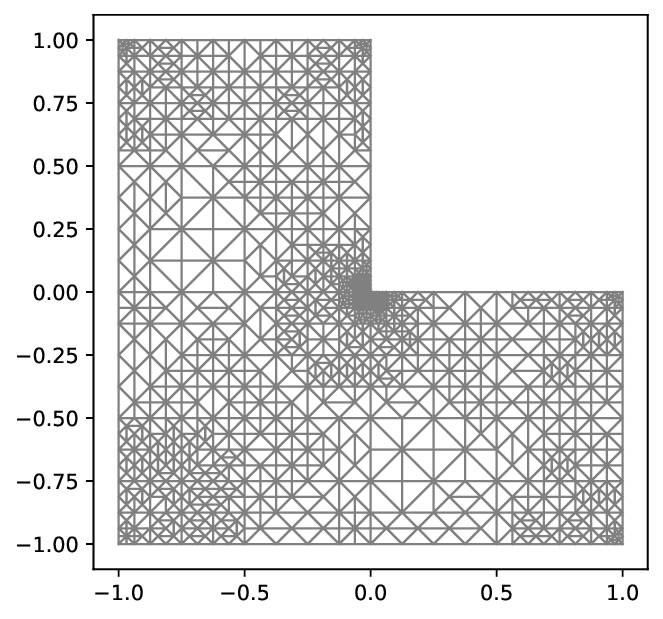}
        \caption{Dofs - $11834$}
    \end{subfigure}
    \hfill
    \begin{subfigure}[b]{0.3\textwidth}
        \centering
        \includegraphics[width=\textwidth]{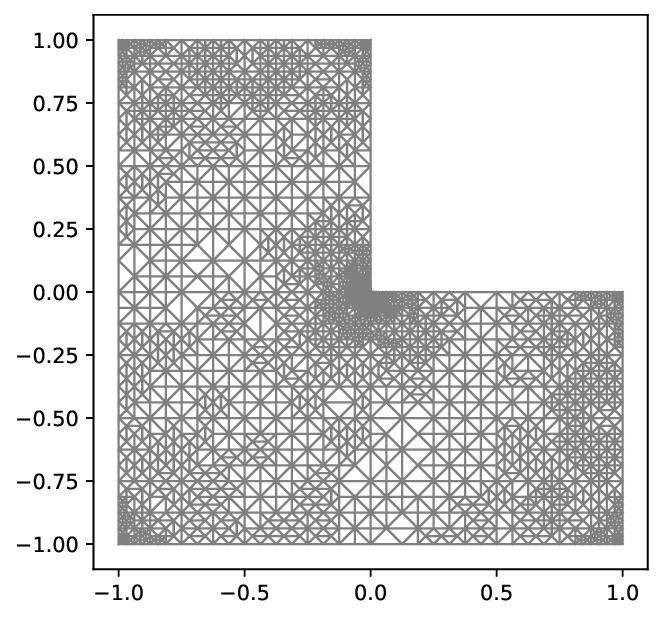}
        \caption{Dofs - $20788$}
    \end{subfigure}
         \begin{subfigure}[b]{0.3\textwidth}
        \centering
        \includegraphics[width=\textwidth]{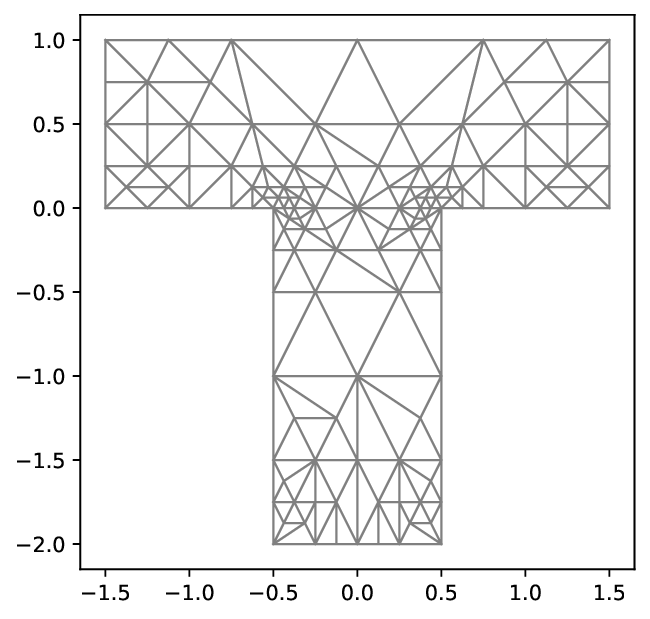}
         \caption{Dofs - $1014$}
    \end{subfigure}
    \hfill
    \begin{subfigure}[b]{0.3\textwidth}
        \centering
        \includegraphics[width=\textwidth]{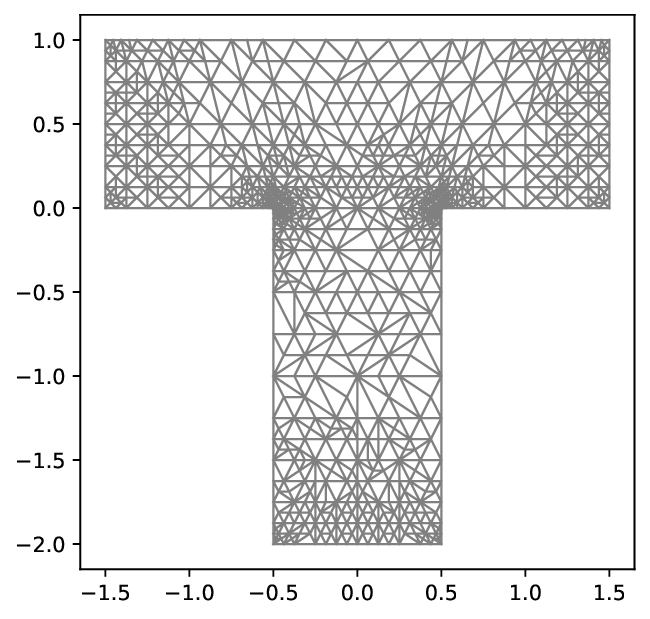}
        \caption{Dofs - $10313$}
    \end{subfigure}
    \hfill
    \begin{subfigure}[b]{0.3\textwidth}
        \centering
        \includegraphics[width=\textwidth]{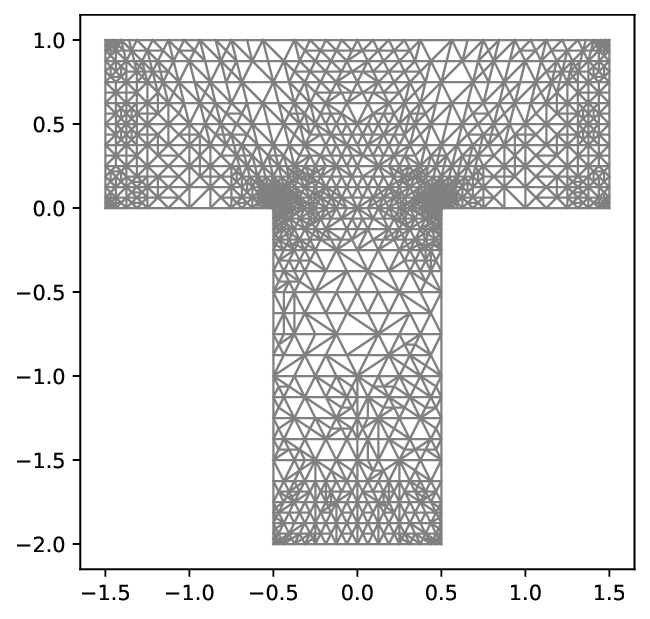}
        \caption{Dofs - $21219$}
    \end{subfigure}
    \caption{\Cref{ex2}: The refined meshes obtained by using the adaptive strategy for C, L and T shape.}\label{fig1}
 \end{figure}

\begin{figure}[!htbp]
    \begin{subfigure}[b]{0.3\textwidth}
        \centering
        \includegraphics[width=\textwidth]{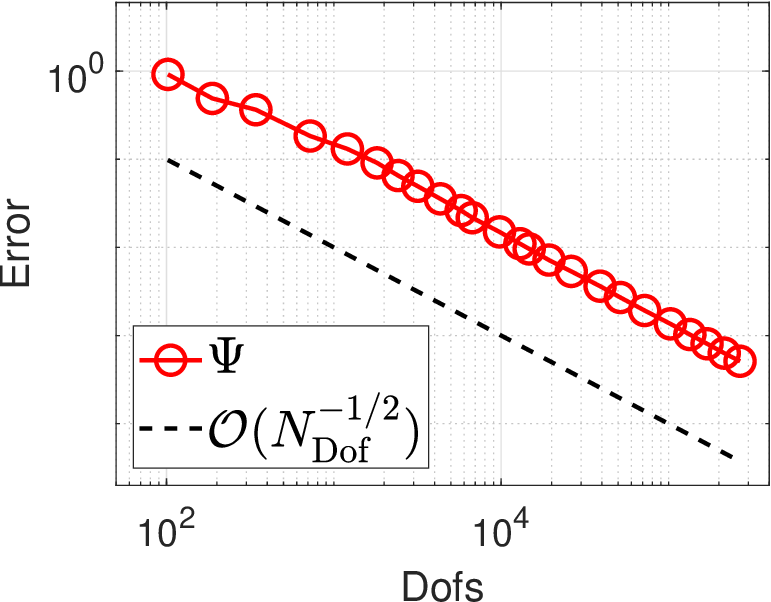}
         \caption{C shape}
    \end{subfigure}
    \hfill
    \begin{subfigure}[b]{0.3\textwidth}
        \centering
        \includegraphics[width=\textwidth]{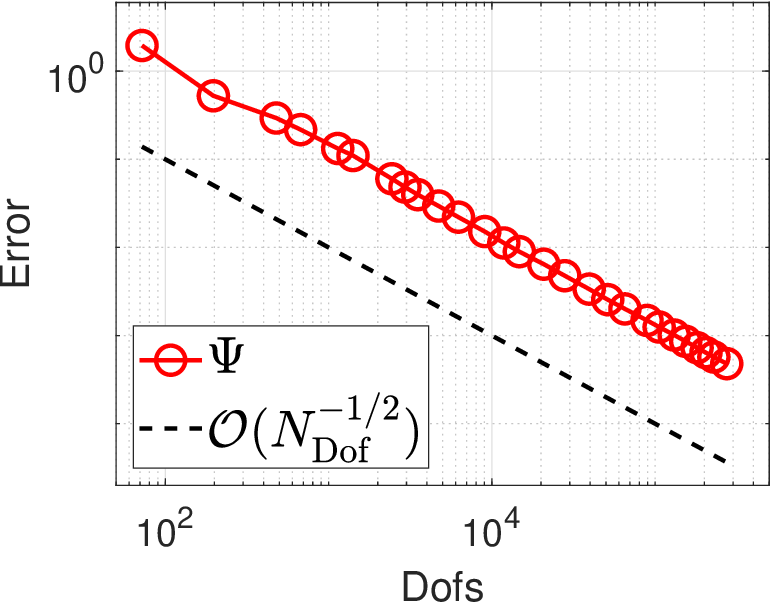}
        \caption{L shape}
    \end{subfigure}
    \hfill
    \begin{subfigure}[b]{0.3\textwidth}
        \centering
        \includegraphics[width=\textwidth]{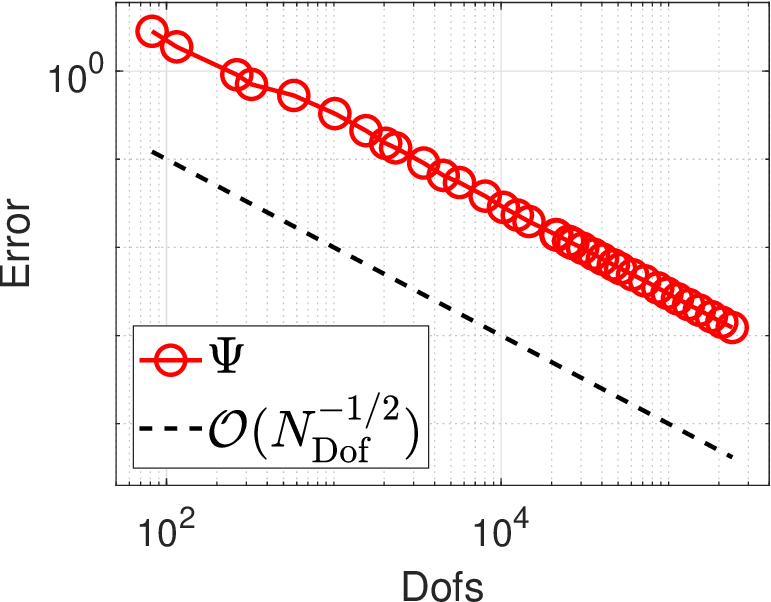}
        \caption{T shape}
    \end{subfigure}
    \caption{\Cref{ex2}: Convergence plots for C, L and T shapes.}\label{fig2}
\end{figure}

\begin{figure}[!htbp]
    \centering
    \begin{subfigure}[b]{0.3\textwidth}
        \centering
        \includegraphics[width=\textwidth]{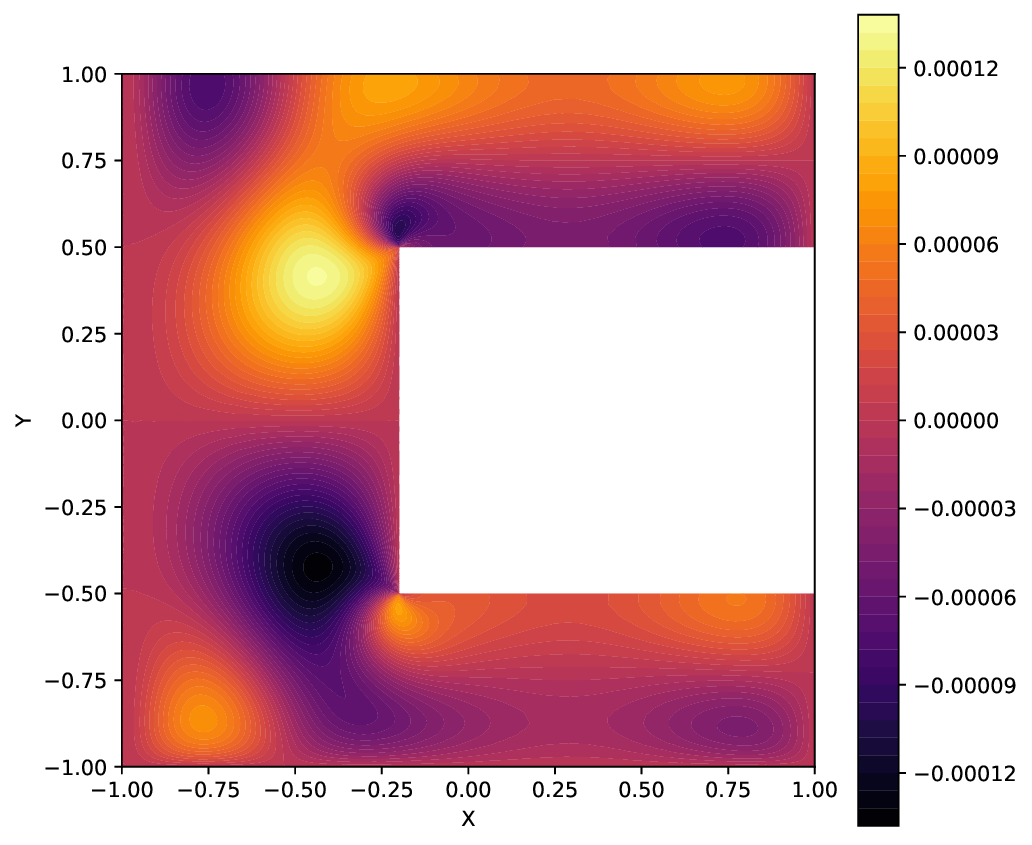}
        \caption{$u_{1_h}$}
    \end{subfigure}
    \begin{subfigure}[b]{0.3\textwidth}
        \centering
        \includegraphics[width=\textwidth]{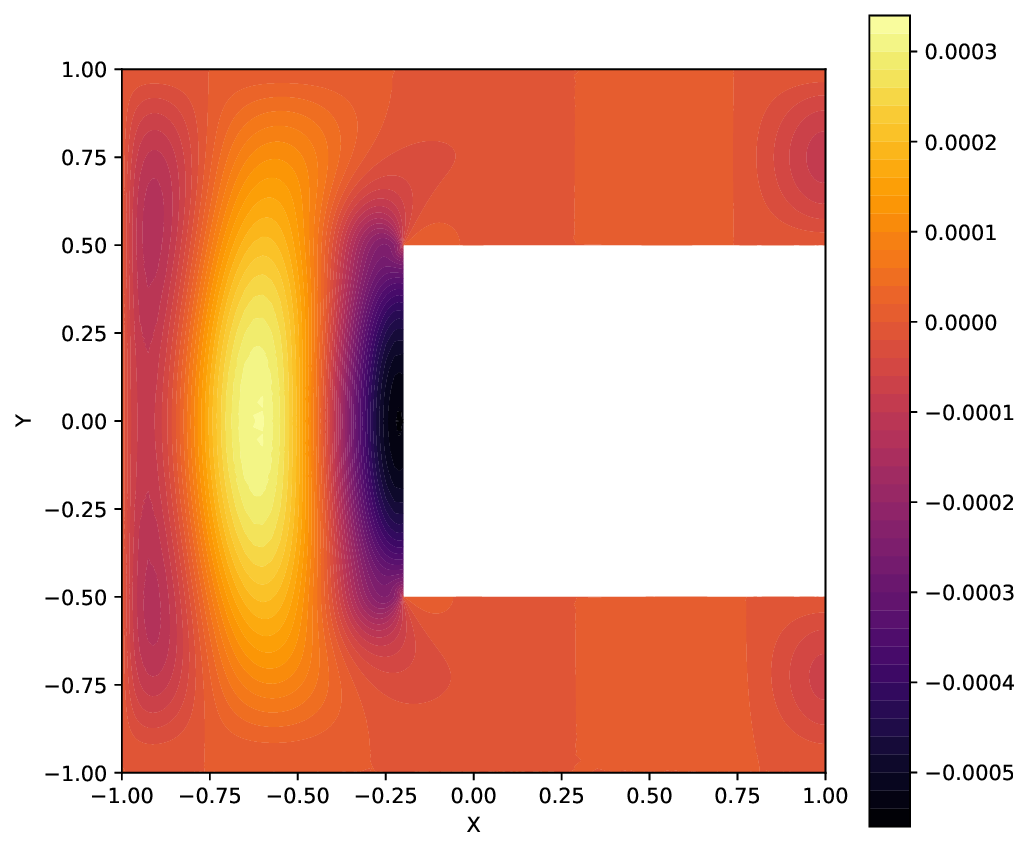}
        \caption{$u_{2_h}$}
    \end{subfigure}
    \begin{subfigure}[b]{0.3\textwidth}
        \centering
        \includegraphics[width=\textwidth]{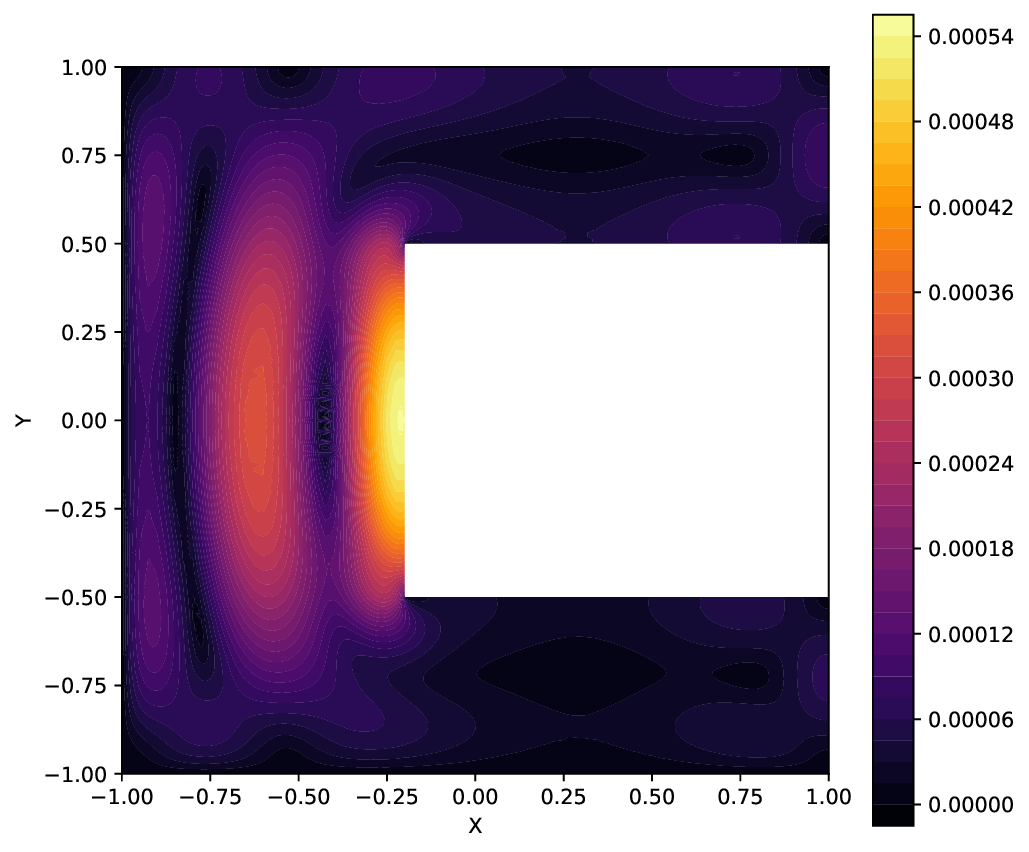}
        \caption{$|u_{h}|$}
    \end{subfigure}
    \centering
    \begin{subfigure}[b]{0.3\textwidth}
        \centering
        \includegraphics[width=\textwidth]{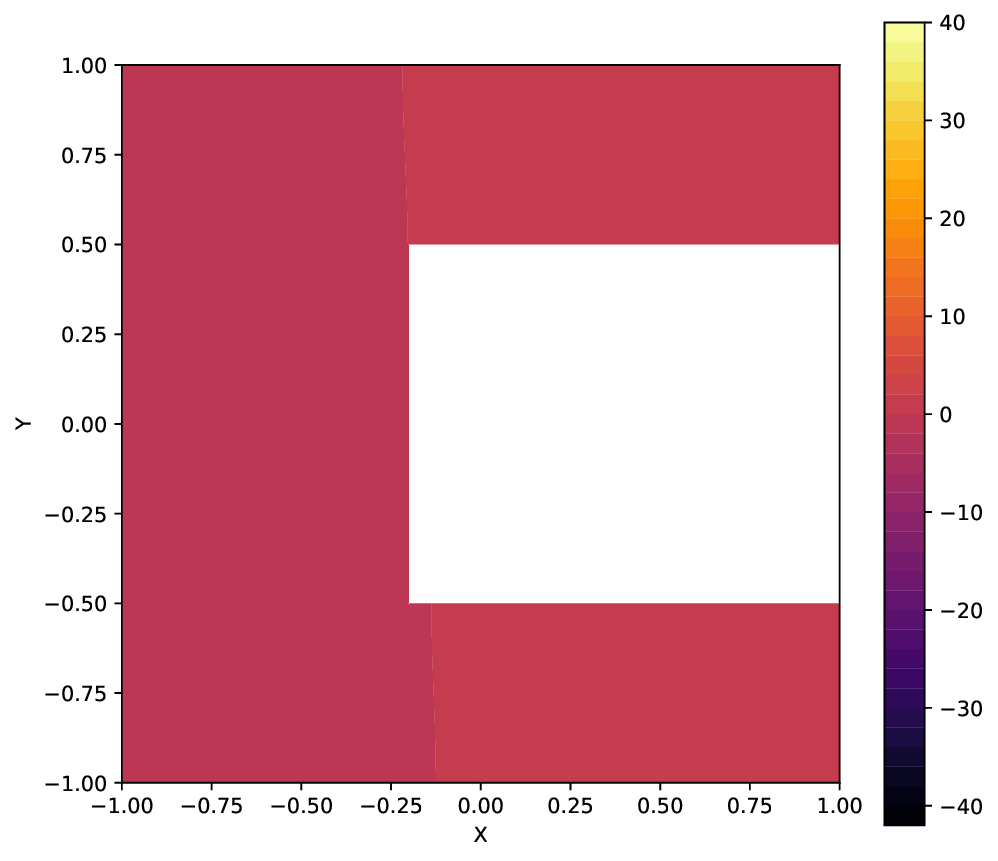}
        \caption{$p_h$}
    \end{subfigure}
    \begin{subfigure}[b]{0.3\textwidth}
        \centering
        \includegraphics[width=\textwidth]{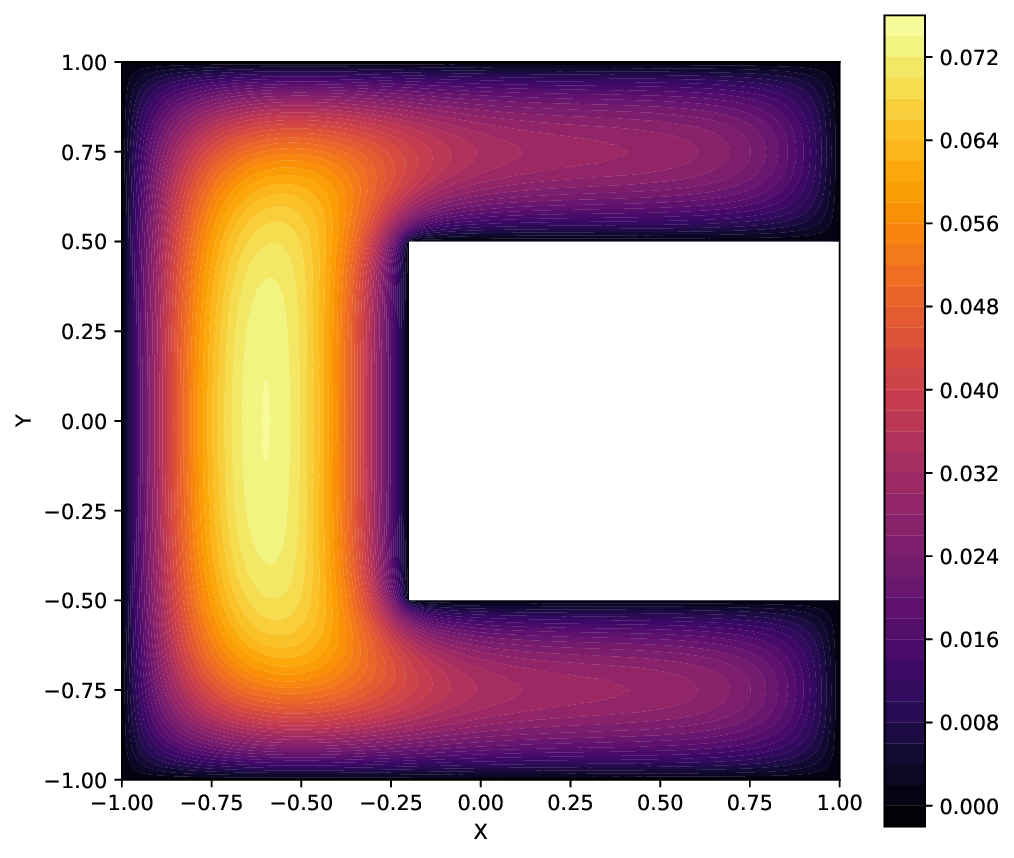}
        \caption{$\psi_h$}
    \end{subfigure}
    \caption{\Cref{ex2}: Plots of numerical solutions of the velocity $\boldsymbol{u}_h\! =\!(u_{1_h}, u_{2_h})$, pressure $p_h$ and potential $\psi_h$ for C-shape domain.}\label{fig3}
\end{figure}

\begin{figure}[!htbp]
    \centering
    \begin{subfigure}[b]{0.3\textwidth}
        \centering
        \includegraphics[width=\textwidth]{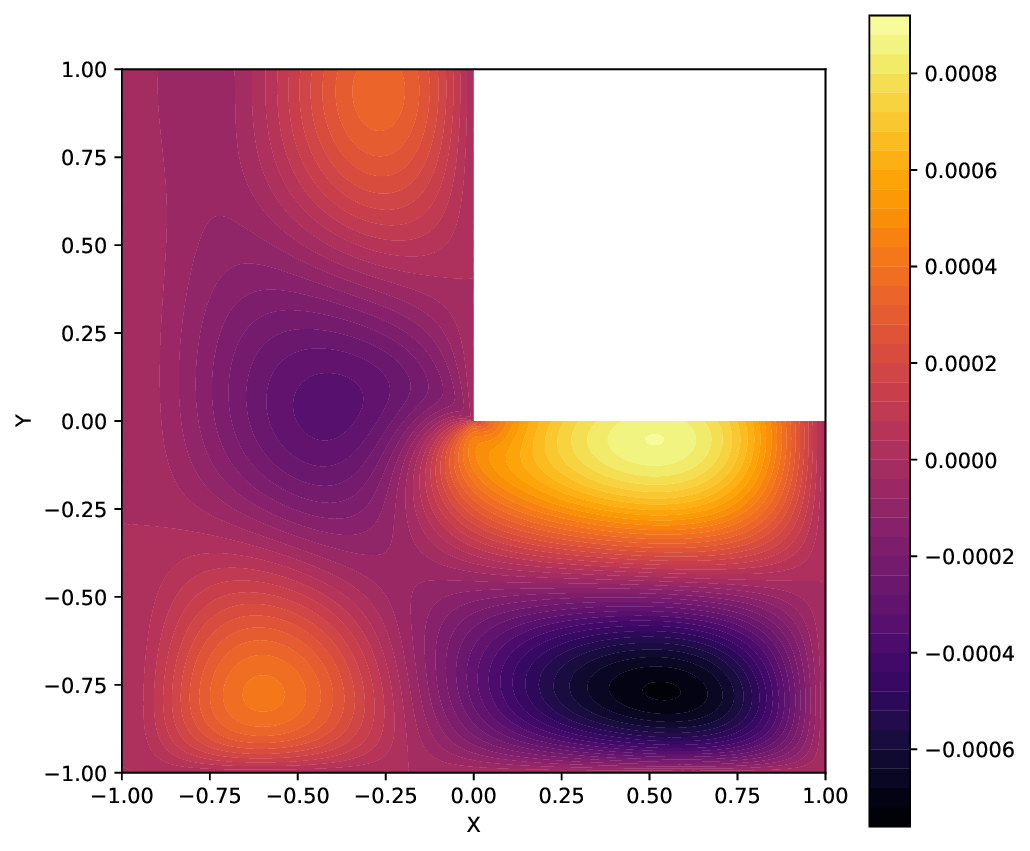}
        \caption{$u_{1_h}$}
    \end{subfigure}
    \begin{subfigure}[b]{0.3\textwidth}
        \centering
        \includegraphics[width=\textwidth]{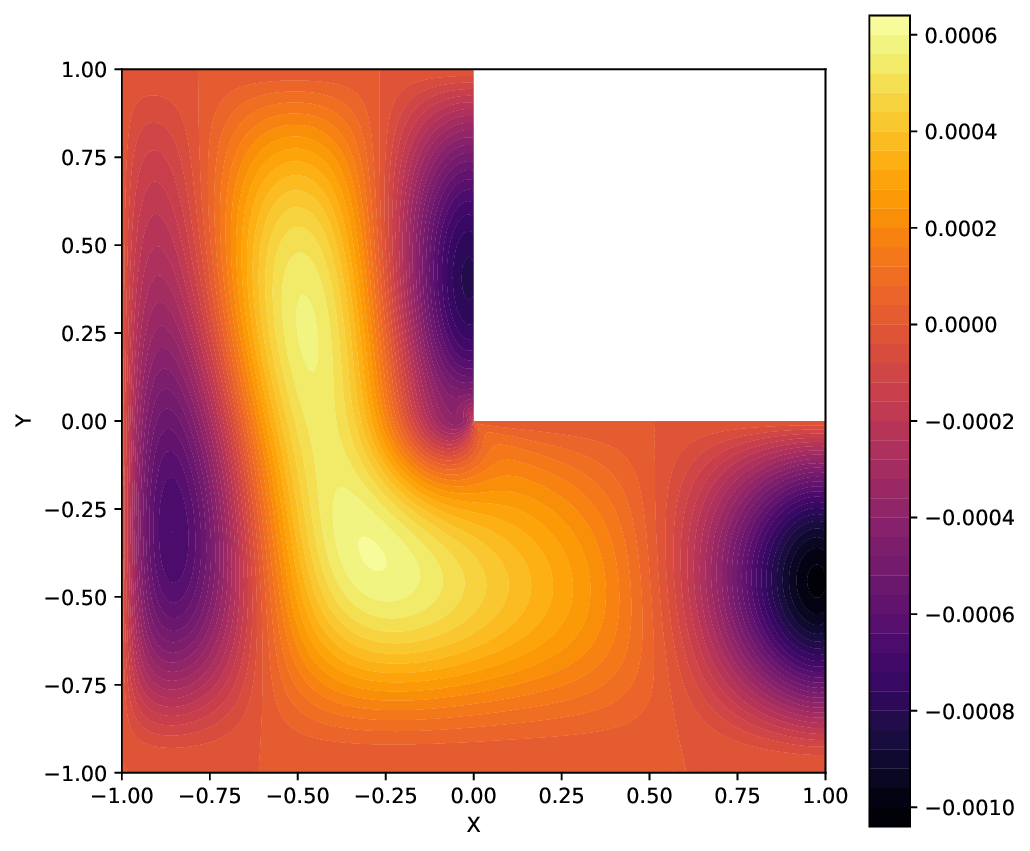}
        \caption{$u_{2_h}$}
    \end{subfigure}
    \begin{subfigure}[b]{0.3\textwidth}
        \centering
        \includegraphics[width=\textwidth]{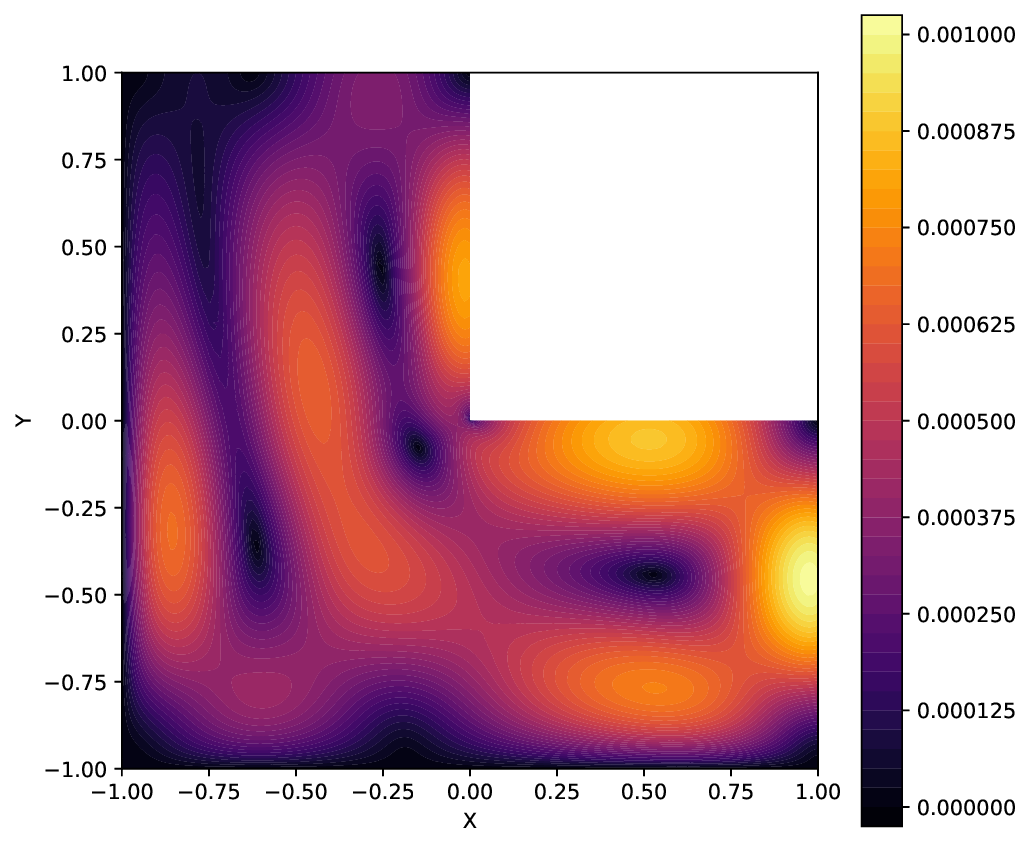}
        \caption{$|u_{h}|$}
    \end{subfigure}
    \centering
    \begin{subfigure}[b]{0.3\textwidth}
        \centering
        \includegraphics[width=\textwidth]{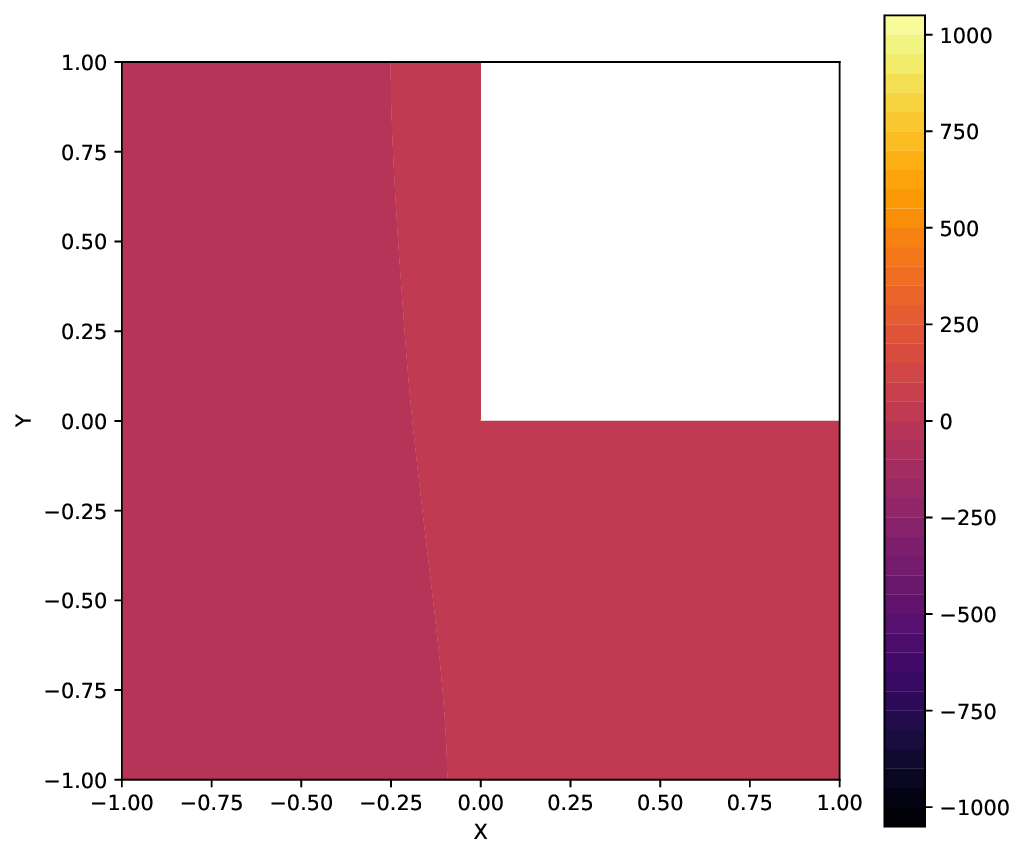}
        \caption{$p_h$}
    \end{subfigure}
    \begin{subfigure}[b]{0.3\textwidth}
        \centering
        \includegraphics[width=\textwidth]{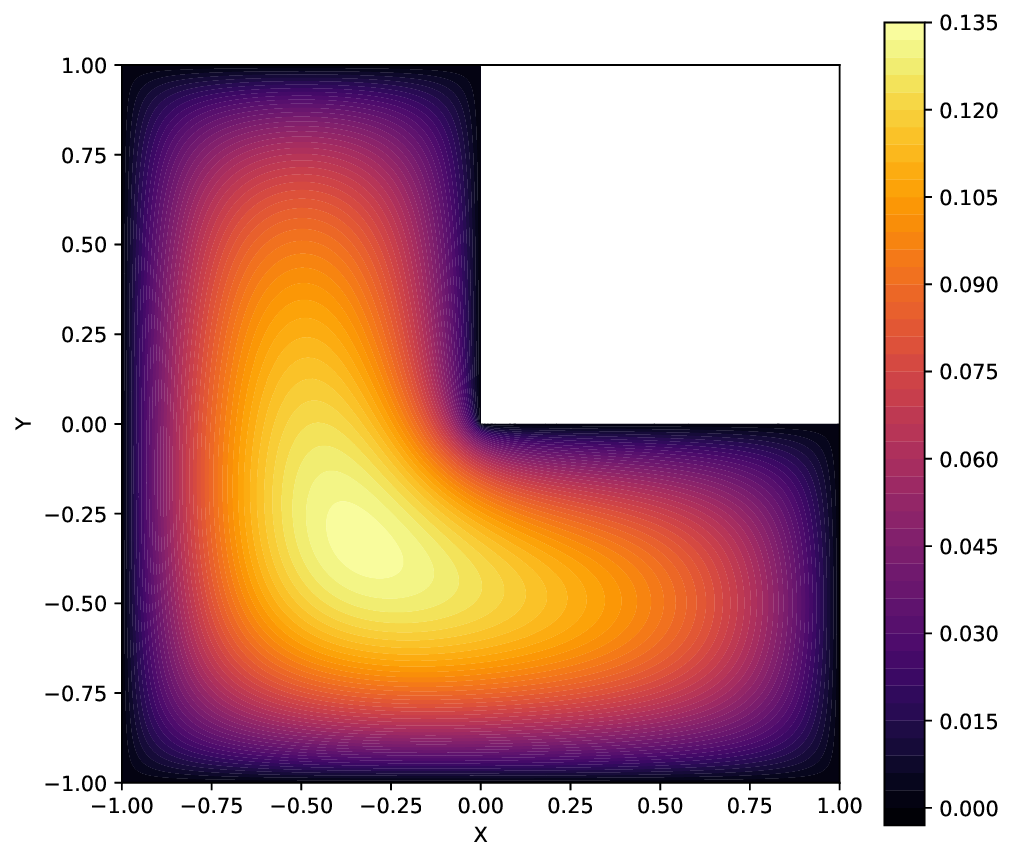}
        \caption{$\psi_h$}
    \end{subfigure}
    \caption{\Cref{ex2}: Plots of numerical solutions of the velocity $\boldsymbol{u}_h\! =\!(u_{1_h}, u_{2_h})$, pressure $p_h$ and potential $\psi_h$ for L-shape domain.}\label{fig4}
\end{figure}

\begin{figure}[!htbp]
    \centering
    \begin{subfigure}[b]{0.3\textwidth}
        \centering
        \includegraphics[width=\textwidth]{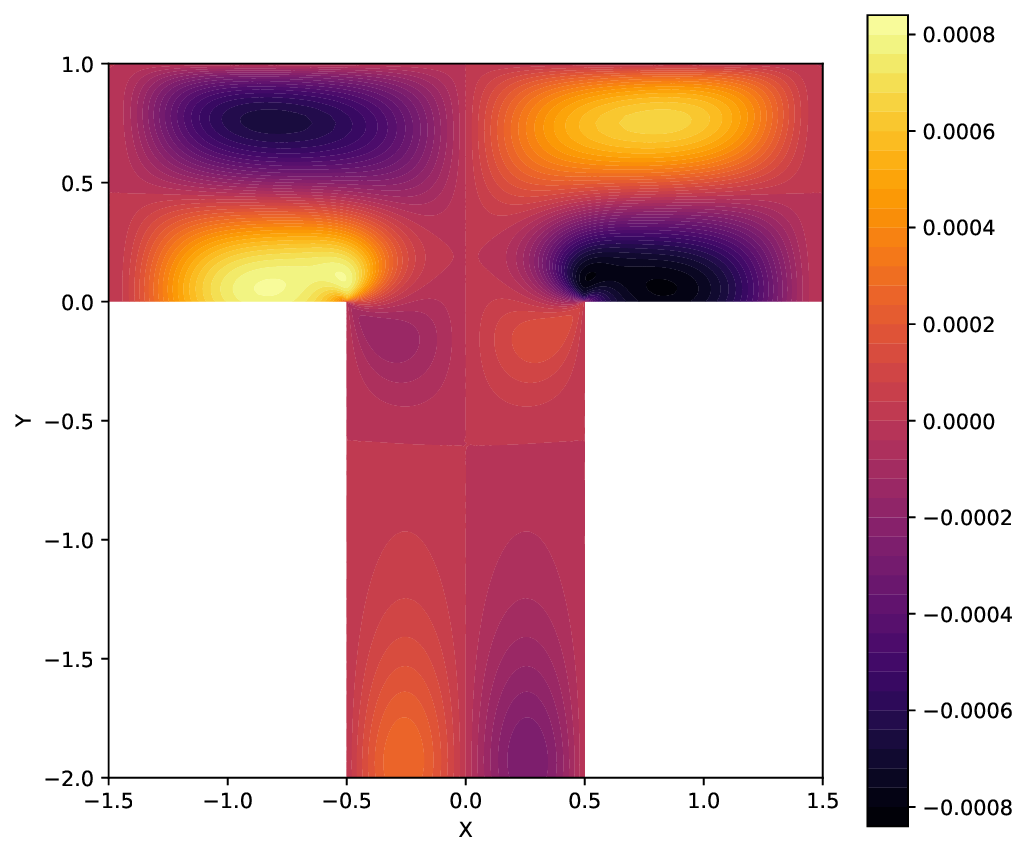}
        \caption{$u_{1_h}$}
    \end{subfigure}
    \begin{subfigure}[b]{0.3\textwidth}
        \centering
        \includegraphics[width=\textwidth]{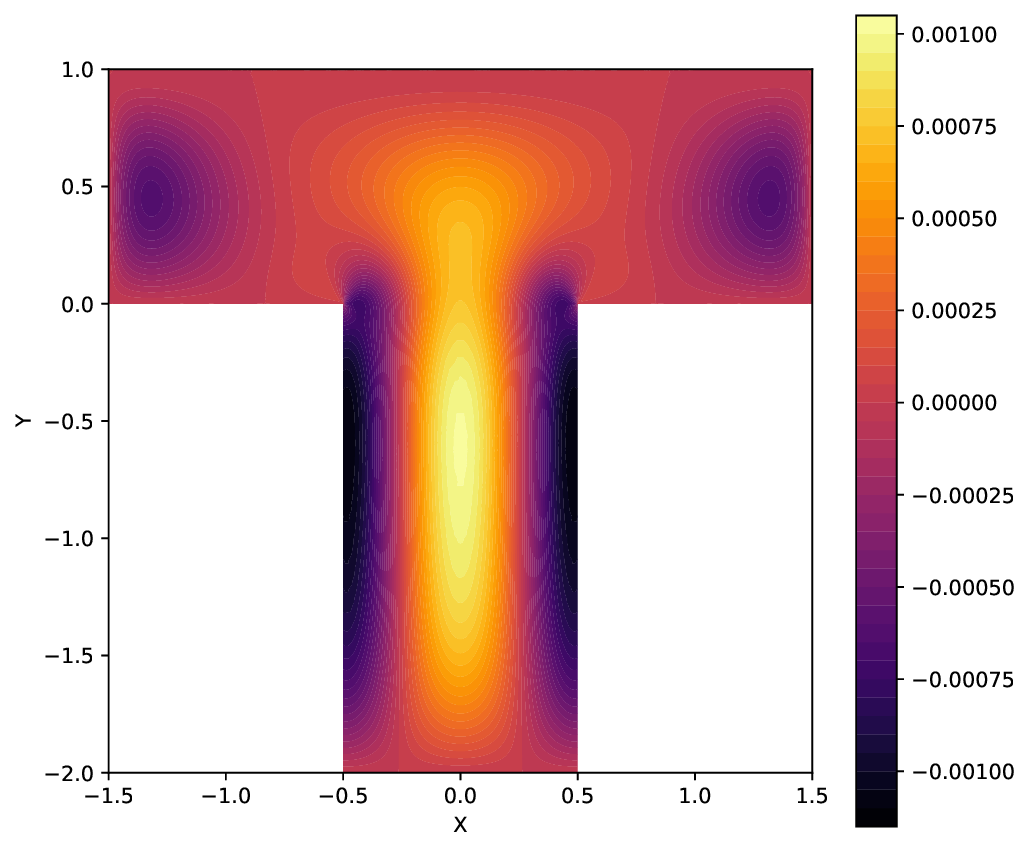}
        \caption{$u_{2_h}$}
    \end{subfigure}
    \begin{subfigure}[b]{0.3\textwidth}
        \centering
        \includegraphics[width=\textwidth]{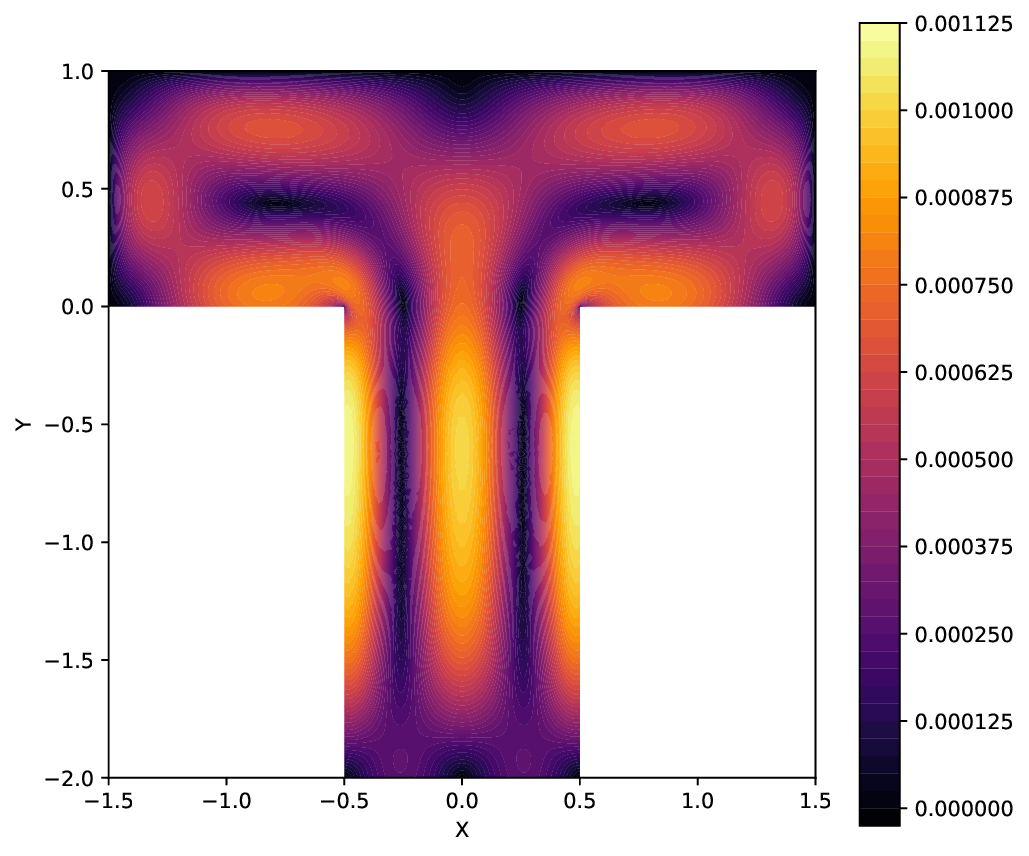}
        \caption{$|u_{h}|$}
    \end{subfigure}
    \centering
    \begin{subfigure}[b]{0.3\textwidth}
        \centering
        \includegraphics[width=\textwidth]{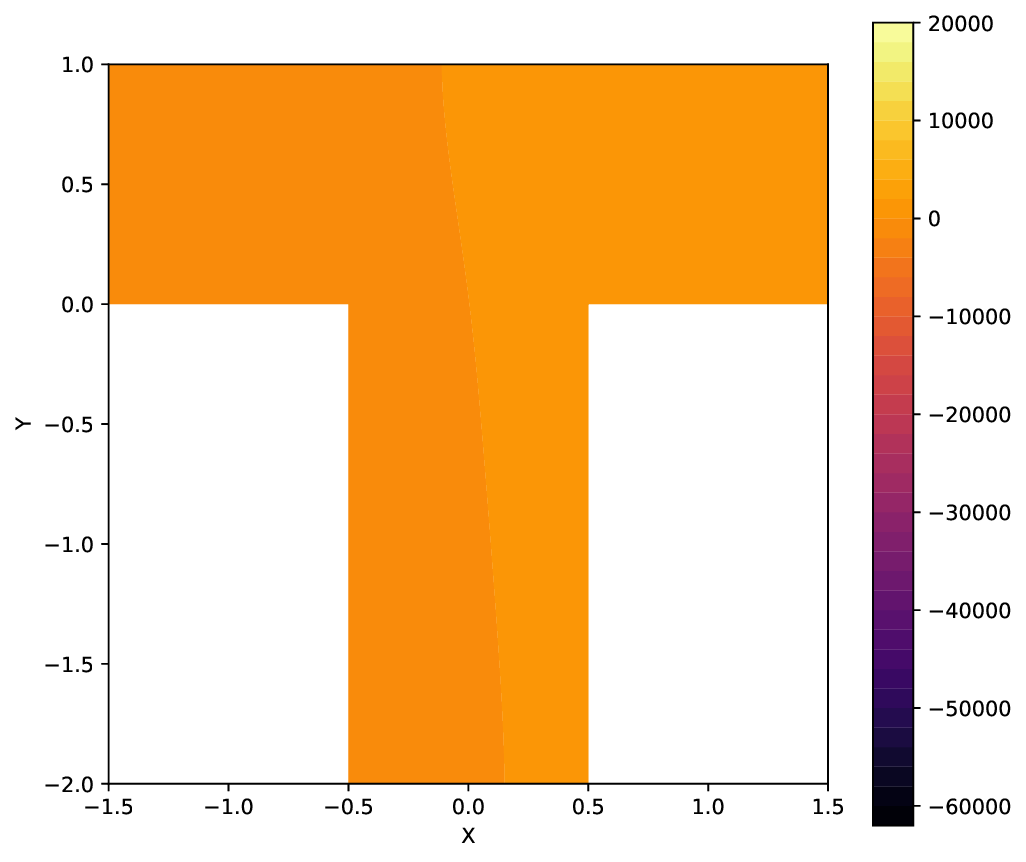}
        \caption{$p_h$}
    \end{subfigure}
    \begin{subfigure}[b]{0.3\textwidth}
        \centering
        \includegraphics[width=\textwidth]{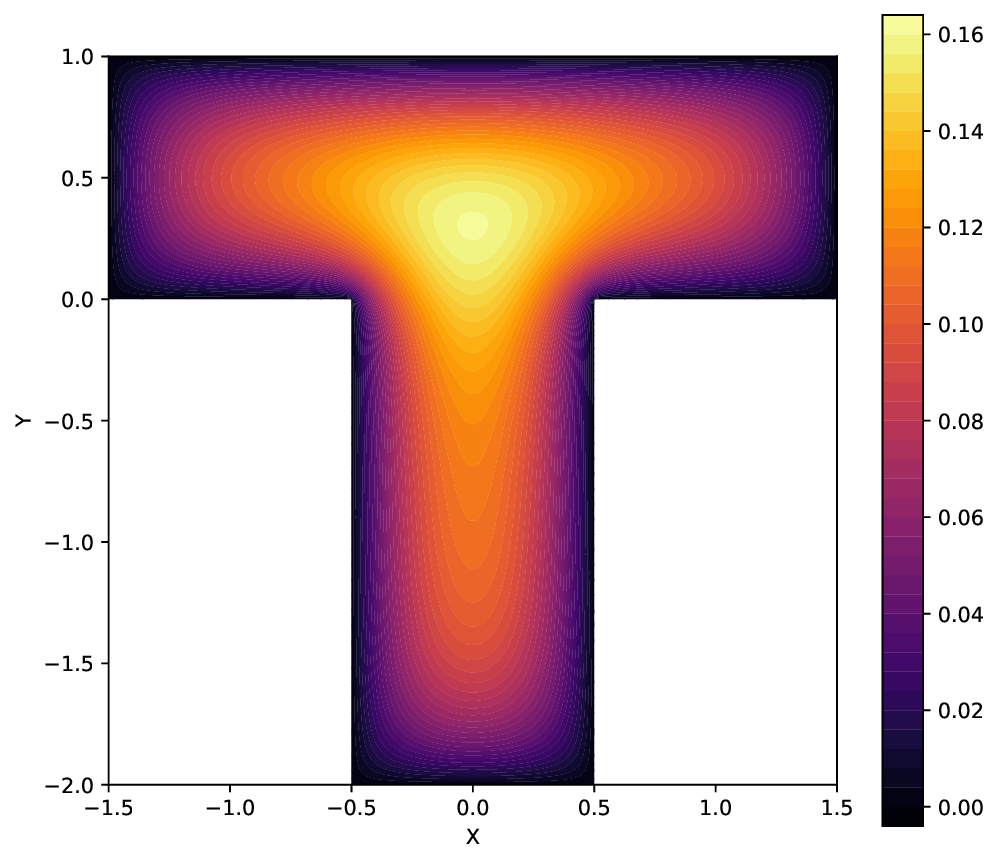}
        \caption{$\psi_h$}
    \end{subfigure}
    \caption{\Cref{ex2}: Plots of numerical solutions of the velocity $\boldsymbol{u}_h\! =\!(u_{1_h}, u_{2_h})$, pressure $p_h$ and potential $\psi_h$ for T-shape domain.}\label{fig5}
\end{figure}
}\end{example}
\medskip
\begin{example}[\textbf{Adaptive mesh refinement for a boundary layers problem}]\label{ex3}
{\normalfont
In this example, we consider the triangular domain 
$\Omega = \{(x,y)\,\colon\, x>0,\, y>0,\, x+y<1\}$ 
with $\boldsymbol{E} = (-10,0)^{t}$, $\varepsilon = 10$, $\mu = 1$, $k_0 = 1$, $k_1 = 10$, $\beta = 10$, and $\gamma = 50$. 
The source functions $\boldsymbol{f}$ and $g$ are chosen such that the manufactured solutions are
\begin{align*}
\boldsymbol{u}(x,y) = \boldsymbol{curl}\Big(x\, y^{2} (1-x-y)^{2}
\Big(
1 - x - \frac{e^{-50x} - e^{-50}}{1 - e^{-50}}
\Big)\Big),
\qquad
p(x,y) = \frac{\cos(2 \pi y)}{1024},
\qquad
\psi(x,y) = \frac{e^x + e^y}{1024}.
\end{align*}

Since boundary layers are present in this problem, uniform mesh refinement may lead to slow convergence. 
To address this issue, we apply an adaptive mesh refinement strategy that focuses on regions where boundary layers occur. 
The adaptively refined meshes shown in \Cref{fig6} illustrate how the refinement is concentrated in these critical areas.

As the mesh is refined adaptively, we observe that the convergence rates improve significantly. 
Optimal convergence behavior appears once the mesh becomes sufficiently fine, as shown in \Cref{fig7}(a) and \Cref{fig7}(b). 
The effectivity index is also shown in \Cref{fig7}(c).
The improved convergence rates indicate that the adaptive method successfully detects and resolves regions with large errors. 
Finally, \Cref{fig8} presents detailed plots of the numerical solution.

}\end{example}

\begin{figure}[!htbp]
    \begin{subfigure}[b]{0.3\textwidth}
        \centering
        \includegraphics[width=\textwidth]{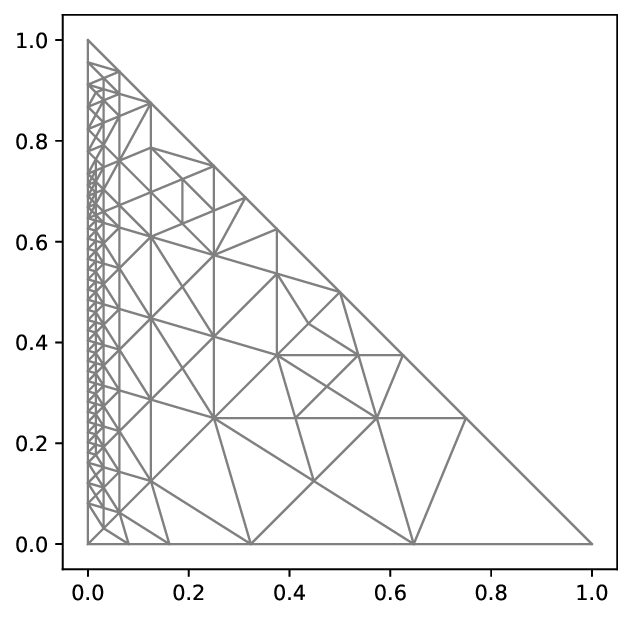}
         \caption{Dofs - $1312$}
    \end{subfigure}
    \hfill
    \begin{subfigure}[b]{0.3\textwidth}
        \centering
        \includegraphics[width=\textwidth]{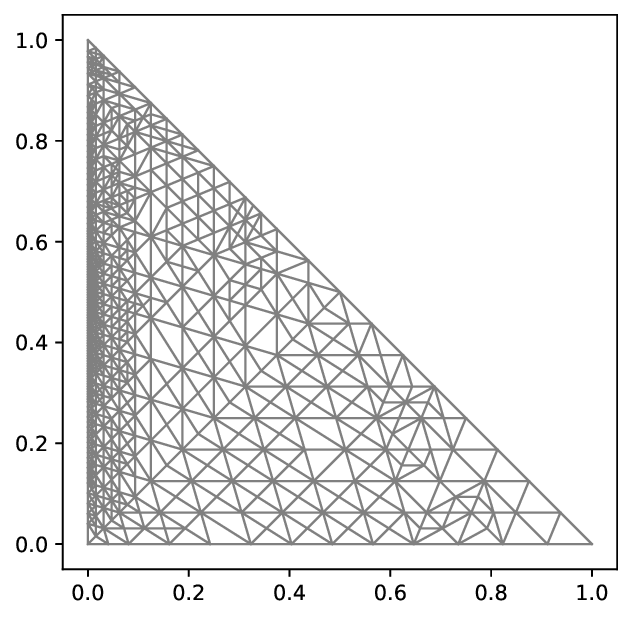}
        \caption{Dofs - $6409$}
    \end{subfigure}
    \hfill
    \begin{subfigure}[b]{0.3\textwidth}
        \centering
        \includegraphics[width=\textwidth]{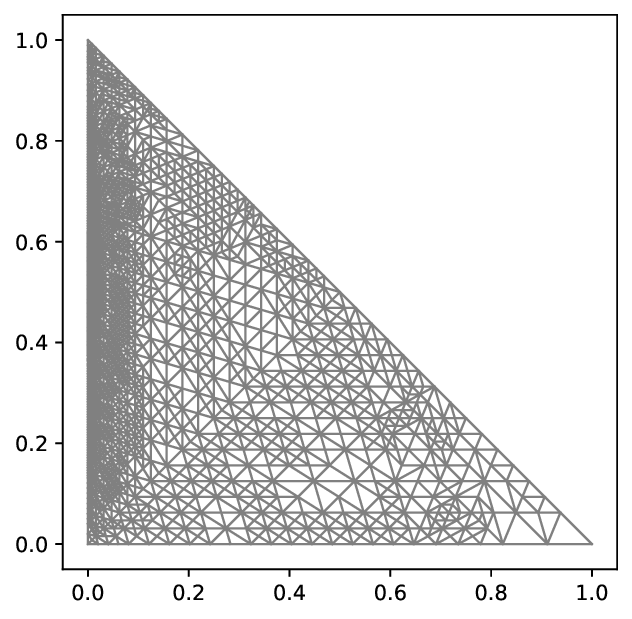}
        \caption{Dofs - $21671$}
    \end{subfigure}
    \caption{\Cref{ex3}: The refined mesh obtained by using the adaptive strategy.}\label{fig6}
\end{figure}

\begin{figure}[!htbp]
    \centering
    \begin{subfigure}[b]{0.3\textwidth}
        \centering
        \includegraphics[width=\textwidth]{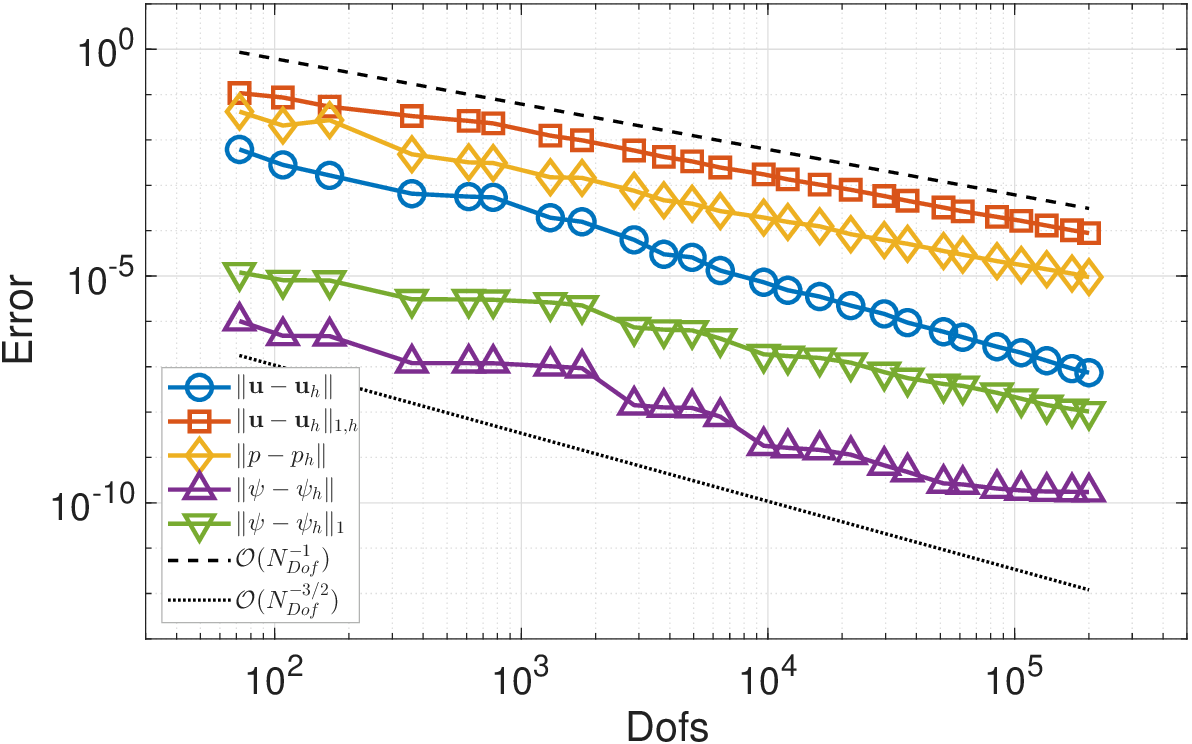}
        \caption{}
    \end{subfigure}
    \begin{subfigure}[b]{0.3\textwidth}
        \centering
        \includegraphics[width=\textwidth]{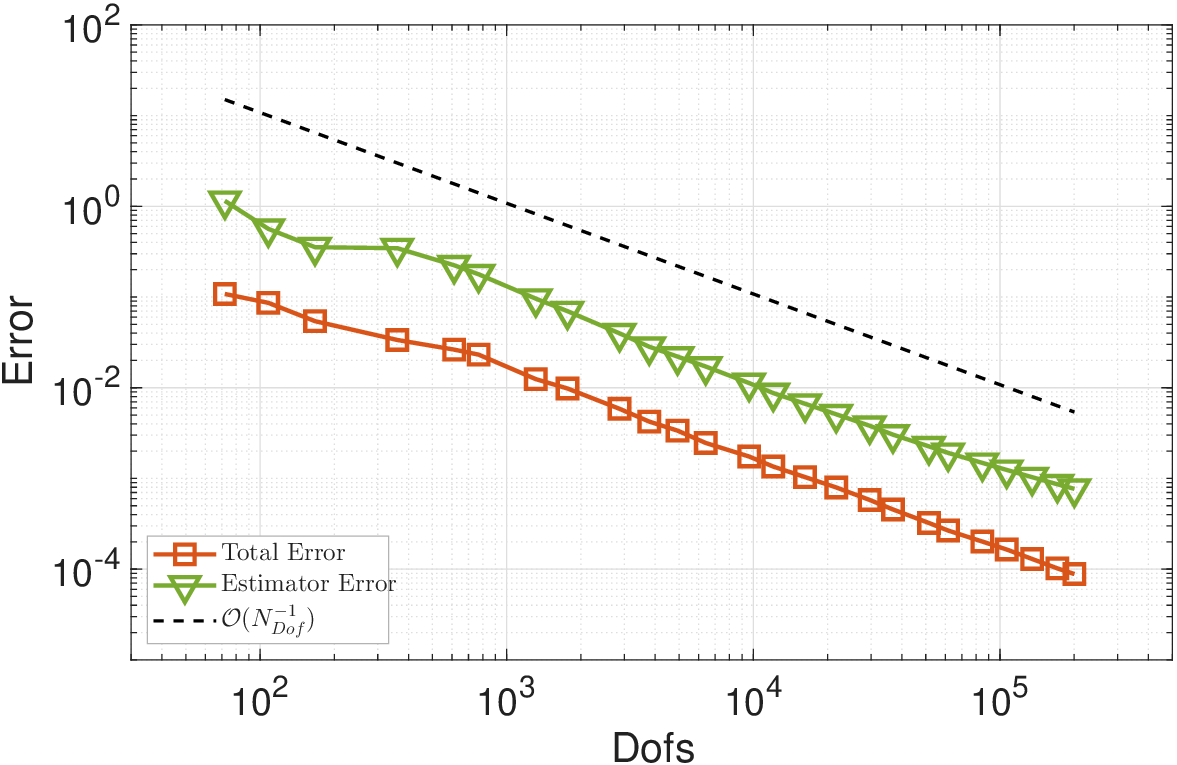}
        \caption{}
    \end{subfigure}
    \begin{subfigure}[b]{0.3\textwidth}
        \centering
        \includegraphics[width=\textwidth]{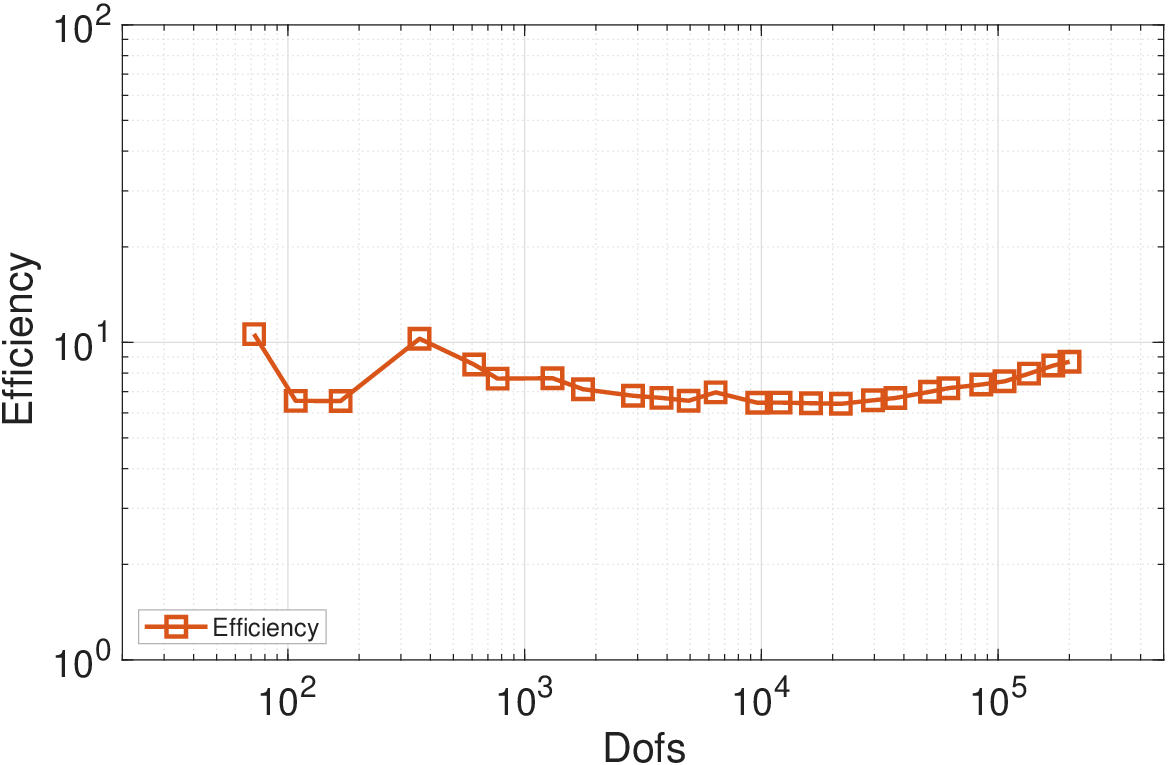}
        \caption{}
    \end{subfigure}
    \caption{\Cref{ex3}: Convergence plot in \ref{fig7}(a) and \ref{fig7}(b) and Efficiency plot in \ref{fig7}(c).}\label{fig7}
\end{figure}

\begin{figure}[!htbp]
    \centering
    \begin{subfigure}[b]{0.3\textwidth}
        \centering
        \includegraphics[width=\textwidth]{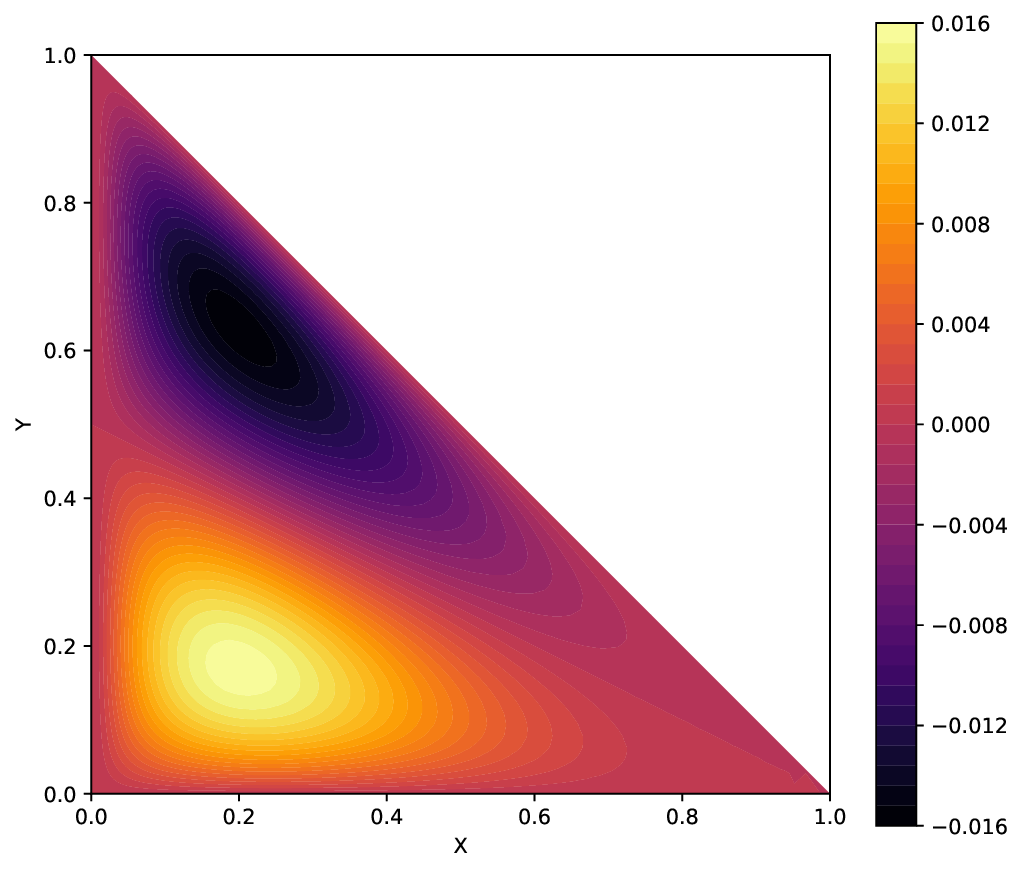}
        \caption{$u_{1_h}$}
    \end{subfigure}
    \begin{subfigure}[b]{0.3\textwidth}
        \centering
        \includegraphics[width=\textwidth]{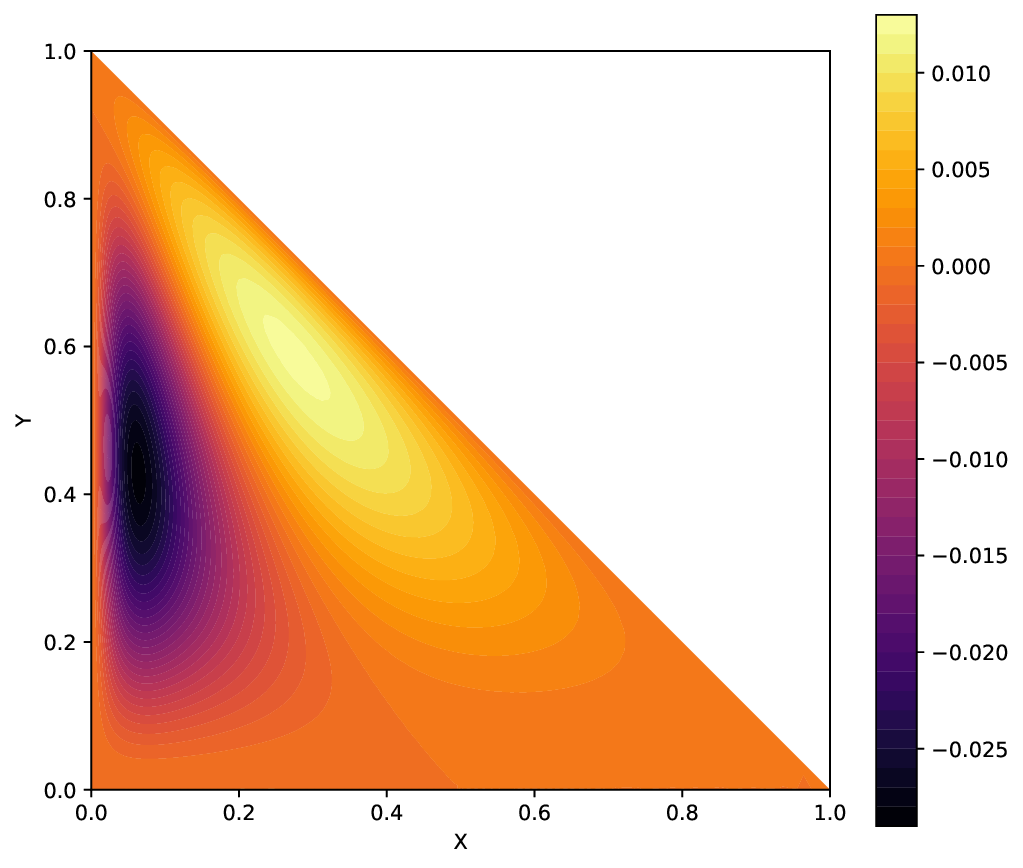}
        \caption{$u_{2_h}$}
    \end{subfigure}
    \begin{subfigure}[b]{0.3\textwidth}
        \centering
        \includegraphics[width=\textwidth]{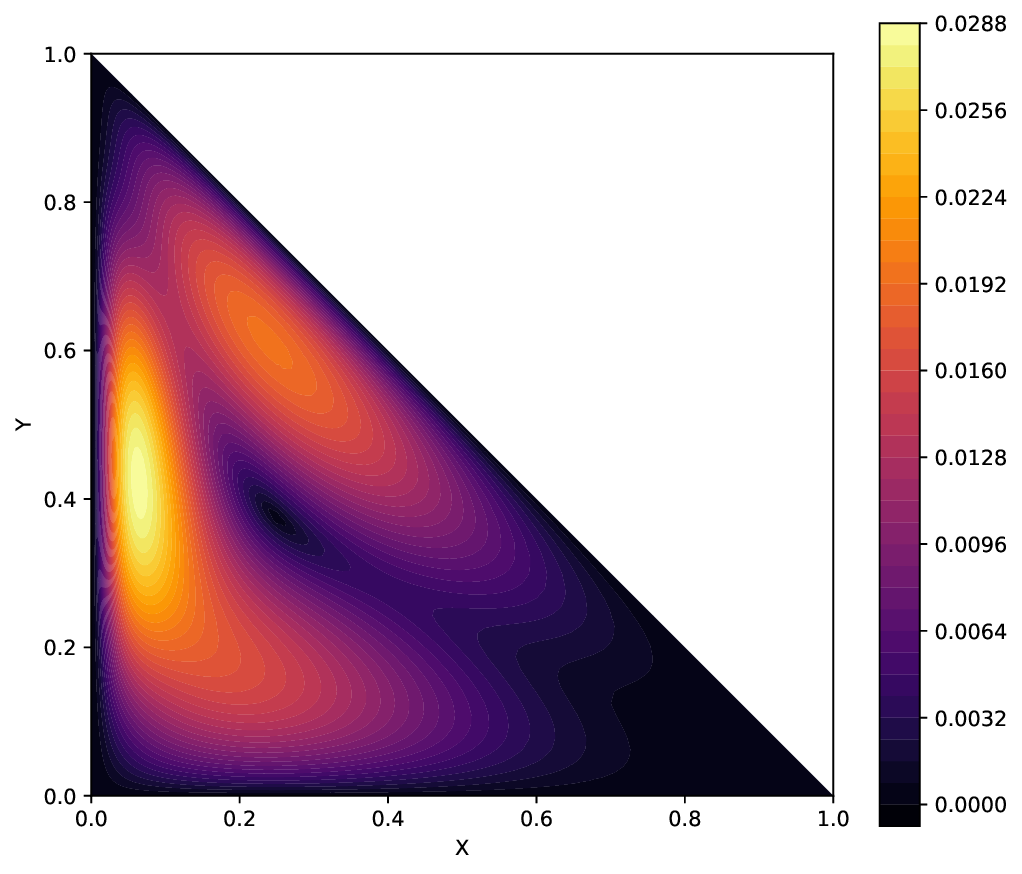}
        \caption{$|u_{h}|$}
    \end{subfigure}
    \centering
    \begin{subfigure}[b]{0.3\textwidth}
        \centering
        \includegraphics[width=\textwidth]{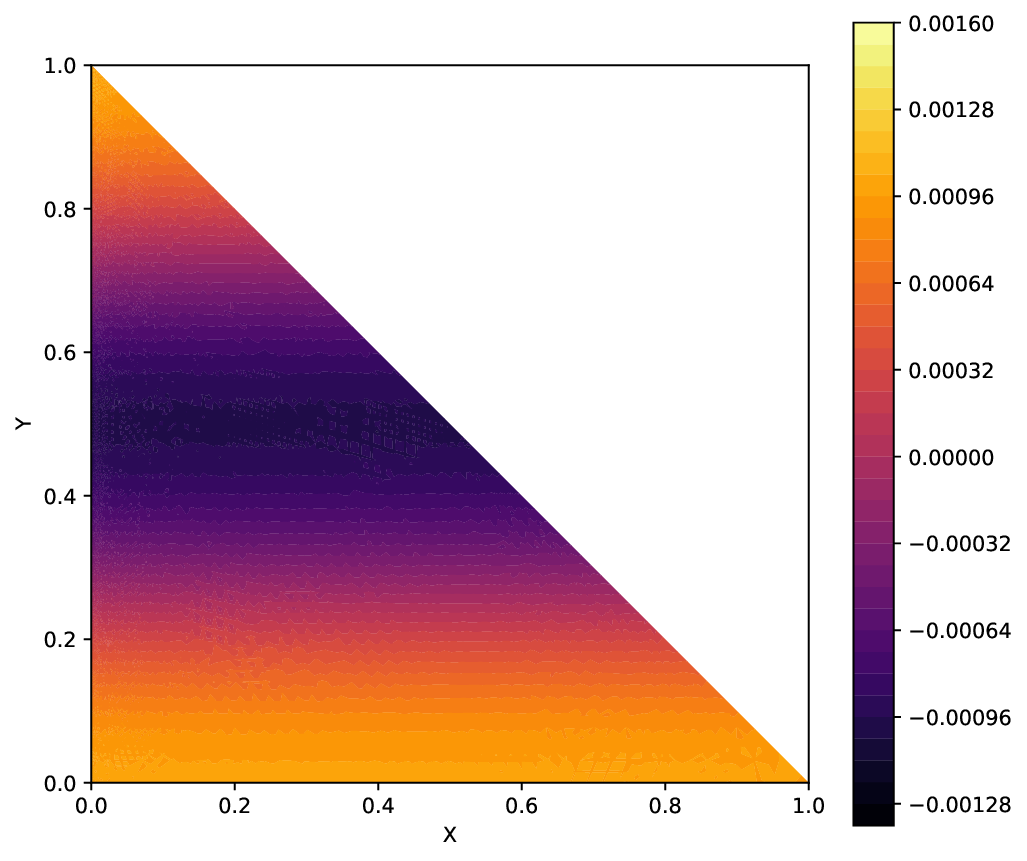}
        \caption{$p_h$}
    \end{subfigure}
    \begin{subfigure}[b]{0.3\textwidth}
        \centering
        \includegraphics[width=\textwidth]{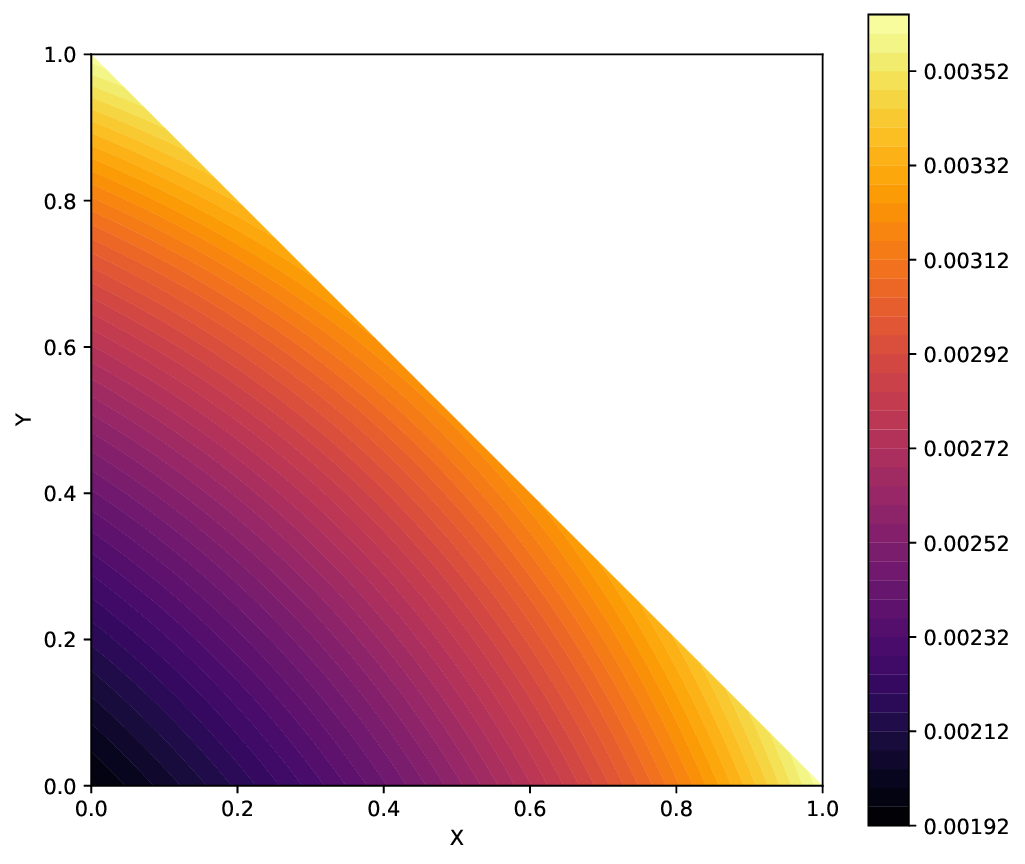}
        \caption{$\psi_h$}
    \end{subfigure}
    \caption{\Cref{ex3}: Plots of numerical solutions of the velocity $\boldsymbol{u}_h\! =\!(u_{1_h}, u_{2_h})$, pressure $p_h$ and potential $\psi_h$.}\label{fig8}
\end{figure}
\medskip
\begin{example}[\textbf{Flow past through a rectangular pipe with a circular hole}]\label{ex4}
    {\normalfont Consider a more realistic problem in which we study the flow inside a rectangular pipe $(0, 2.2) \times (0, 0.41)$ containing an obstruction (a circular hole) centered at $(0.2, 0.2)$ with radius $0.1$. The coefficients are chosen as $\mu = 1$, $\varepsilon = 1.0$, $\boldsymbol{E} = (-1, 0.0)^{t}$, $k_0 = 1$, $k_1 = 1$, $\beta = 1$, and $\gamma = 50$. We prescribe the inflow condition
\begin{align*}
\boldsymbol{u}_D = \Big(\frac{4y(0.41-y)}{0.41^2}, 0\Big), \qquad \psi_D = \cos(\pi x + \pi y).
\end{align*}
On the walls, we impose a zero boundary condition for $\boldsymbol{u}$, while a Navier boundary condition is applied on the circular boundary. on the outlet, we impose do nothing boundary condition.

The function $\psi$ is taken to be zero on all other boundaries. Since the exact solution for this problem is not known, we focus on the behavior of the numerical solution.

We begin with an initial mesh consisting of $1114$ elements and apply an adaptive refinement strategy. This adaptive refinement efficiently improves the quality of the numerical solution by concentrating the mesh in regions where higher resolution is needed. The numerical solutions obtained using this approach are shown in \Cref{fig9}, which illustrate the improved accuracy and resolution achieved through adaptive refinement.
    }
\end{example}
\begin{figure}
\centering
    \begin{subfigure}[b]{\textwidth}
        \centering
        \includegraphics[width=\textwidth]{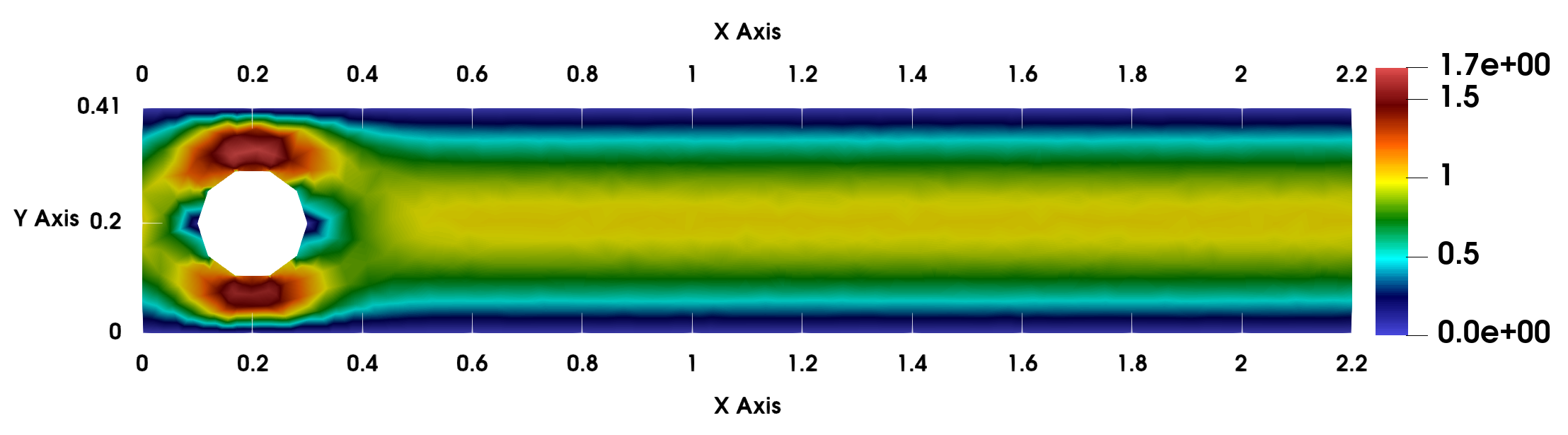}
        \caption{Velocity}
    \end{subfigure}
    \begin{subfigure}[b]{\textwidth}
        \centering
        \includegraphics[width=\textwidth]{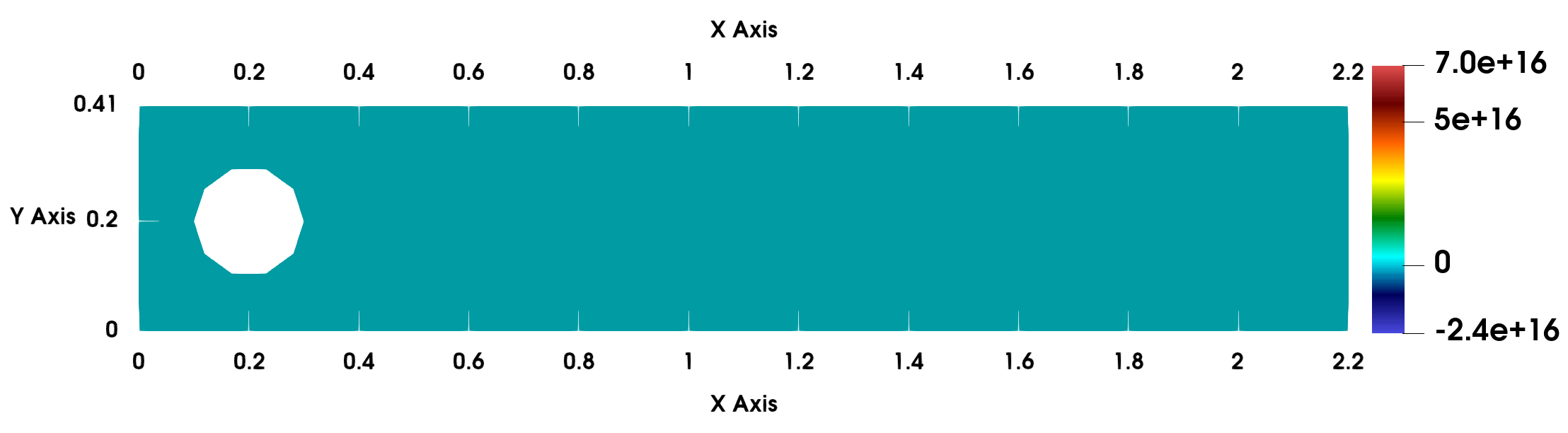}
        \caption{Pressure}
    \end{subfigure}
    \begin{subfigure}[b]{\textwidth}
        \centering
        \includegraphics[width=\textwidth]{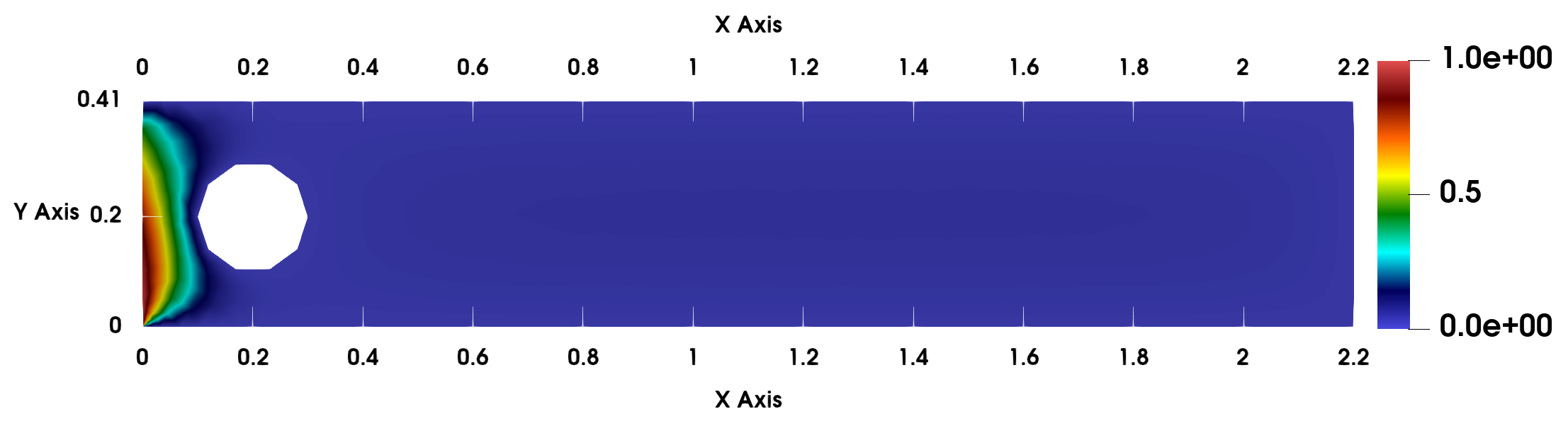}
        \caption{Potential}
    \end{subfigure}
    \caption{\Cref{ex4}: Plots of numerical solutions of the velocity $\boldsymbol{u}_h$, pressure $p_h$ and potential $\psi_h$.}\label{fig9}
\end{figure}

\section{Conclusion}\label{sec:conc}

In this work, we studied the Stokes-Poisson-Boltzmann (SPB) equations with slip boundary conditions on a Lipschitz domain under minimal regularity assumptions. 
A symmetric variant of Nitsche's method was employed to weakly impose the slip boundary conditions.
Under suitable small data assumptions, we established the existence and uniqueness of the discrete solution using Banach's fixed-point theorem combined with the Babu\v{s}ka-Brezzi theory and the Minty-Browder theorem.
In addition, we derived optimal \emph{a priori} error estimates for the proposed finite element scheme and proved the optimal convergence of the method in the energy norm. 
Furthermore, a residual-based \emph{a posteriori} error estimator is developed and analyzed, with its reliability and efficiency rigorously established.
Finally, several numerical experiments were presented to confirm the
theoretical convergence rates and to demonstrate the robustness and
accuracy of the proposed approach.
\vspace{0.5cm}
\newline 
\noindent
\section*{Declarations}
\begin{itemize}
\item \textbf{Ethical Approval:} Not applicable.
\item \textbf{Availability of supporting data:} Data will be made available on request.
\item \textbf{Competing interests:} The authors declare no potential conflict of interest.
\item \textbf{Funding:} The first author is grateful for the financial support to conduct his research, provided by ``The Ministry of Education, Government of India (Prime Minister Research Fellowship, PMRF ID: 2802469)''.
\item \textbf{Acknowledgments:} We acknowledge the funding agency.
\end{itemize}

\bibliography{references}  
\bibliographystyle{unsrt}
\end{document}